\documentclass[pdflatex,sn-mathphys-num]{sn-jnl}
\usepackage[utf8]{inputenc}
\usepackage{textcomp}
\usepackage{amsmath}
\usepackage{amssymb}
\usepackage[version=4]{mhchem}
\usepackage{siunitx}
\usepackage{hhline}

\usepackage[utf8]{inputenc}
\usepackage[T1]{fontenc}
\usepackage{xcolor}
\usepackage{booktabs}
\usepackage{amsfonts}
\usepackage{nicefrac}
\usepackage{microtype}
\usepackage{graphicx}
\usepackage{soul}
\usepackage{tikz}
\usetikzlibrary{shapes.geometric} 

\sethlcolor{lightgray}

\renewcommand{\figurename}{Fig.}
\usepackage{graphicx}%
\usepackage{multirow}%

\usepackage{mathrsfs}%
\usepackage{manyfoot}%
\usepackage{textcomp}%
\usepackage{booktabs}%
\usepackage{algorithm}%
\usepackage{algorithmicx}%
\usepackage{algpseudocode}%
\usepackage{listings}%
\usepackage{bm}
\usepackage{soul}
\usepackage{adjustbox}
\usepackage{tablefootnote}
\usepackage[normalem]{ulem}
\usepackage{placeins}

\usepackage{pifont}

\usepackage{enumitem}
\setlist[enumerate]{itemsep=0mm, topsep=0pt}
\setlist[enumerate]{nosep}

\usepackage{subcaption}
\newcommand{\pos}[1]{\bm{r}_{#1}}
\newcommand{\vel}[1]{\bm{v}_{#1}}
\newcommand{\vella}[1]{\bm{v}_{#1}^ \la}

\newcommand{\ma}[1]{M_{#1}}

\newcommand{\mad}[1]{M_{#1}^{dep}}
\newcommand{\tof}{t}
\renewcommand{\t}{T}
\newcommand{\la}{\Gamma}

\newcommand{\dist}{s^\la}

\newcommand{\sma}[1]{a_{#1}}
\newcommand{\ecc}[1]{e_{#1}}
\newcommand{\inc}[1]{i_{#1}}
\newcommand{\aop}[1]{\omega_{#1}}
\newcommand{\raan}[1]{\Omega_{#1}}
\newcommand{\mm}[1]{n_{#1}}
\newcommand{\oet}[1]{\text{\bm{{\oe}}}_{#1}}

\newcommand{\norm}[1]{\left|#1\right|}
\newcommand{\dv}[1]{\Delta \bm{v}_{#1}}

\newcommand{\fv}{\bm{X}}
\renewcommand{\c}{\bm{F}}

\newcommand{\csr}{\bm{\lambda_{r}}}
\newcommand{\csv}{\bm{\lambda_{v}}}
\newcommand{\primer}{\bm{p}}

\newcommand{\mb}{\begin{bmatrix}}
\newcommand{\me}{\end{bmatrix}}
\newtheorem{prop}{Proposition}

\definecolor{colgay}{RGB}{0, 120, 180}

\newcommand{\transpose}{\intercal}

\definecolor{colFamOne}{RGB}{0,114,189}      
\definecolor{colFamTwo}{RGB}{217,83,25}      
\definecolor{colFamThree}{RGB}{237,177,32}   
\definecolor{colFamFour}{RGB}{126,47,142}    
\definecolor{colFamFive}{RGB}{119,172,48}    
\definecolor{colFamSix}{RGB}{77,190,238}     
\definecolor{colFamSeven}{RGB}{162,20,47}    
\definecolor{colFamEight}{RGB}{0,100,0}      
\definecolor{colFamNine}{RGB}{255,128,0}     
\definecolor{colFamTen}{RGB}{0,153,136}      
\definecolor{colFamEleven}{RGB}{204,0,204}   
\definecolor{colFamTwelve}{RGB}{51,51,255}   
\definecolor{colFamThirteen}{RGB}{255,0,127} 
\definecolor{colFamFourteen}{RGB}{139,69,19} 
\definecolor{colFamFourteen}{RGB}{139,69,19} 
\definecolor{colFamFifteen}{RGB}{75,0,130}    
\definecolor{colFamSixteen}{RGB}{128,128,0}   
\definecolor{goals}{RGB}{0, 0, 255}

\definecolor{colFamOneNew}{HTML}{80318E} 
\definecolor{colFamTwoNew}{HTML}{D9551C}      
\definecolor{colFamThreeNew}{HTML}{80318E} 
\definecolor{colFamFourNew}{HTML}{80318E}    
\definecolor{colFamFiveNew}{HTML}{77AB30}    
\definecolor{colFamSixNew}{HTML}{49BDEE}     
\definecolor{colFamSevenNew}{HTML}{A1122D}    
\definecolor{colFamEightNew}{HTML}{106E10}      
\definecolor{colFamNineNew}{HTML}{FF7F00}     
\definecolor{colFamTenNew}{HTML}{009785}      
\definecolor{colFamElevenNew}{HTML}{009785}   
\definecolor{colFamTwelveNew}{HTML}{3030FF} 

\definecolor{colMin}{rgb}{0.000, 0.243, 0.800}
\definecolor{colMax}{rgb}{0.753, 0.000, 0.188}
\definecolor{colSad}{rgb}{0.133, 0.545, 0.133}

\newcommand{\linefamsolid}[1]{%
  \tikz[baseline=-0.5ex]{
    \draw[line width=1.5pt, color=#1] (0,0) -- (2.0,0);
  }%
}

\newcommand{\redfilledstar}{%
  \tikz[baseline=-0.5ex]\node[
    star,
    star points=5,
    star point ratio=2.5,
    draw=black,
    fill=red,
    line width=0.4pt,
    inner sep=0pt,
    minimum size=0.9em
  ] {};
}

\newcommand{\redfilledcircle}{%
  \tikz[baseline=-0.5ex]\node[
    circle,
    draw=black,
    fill=red,
    line width=0.4pt,
    inner sep=0pt,
    minimum size=0.7em
  ] {};
}

\newcommand{\redfilledsquare}{%
  \tikz[baseline=-0.5ex]\node[
    rectangle,
    draw=black,
    fill=red,
    line width=0.4pt,
    inner sep=0pt,
    minimum size=0.7em
  ] {};
}

\newcommand{\redfilledtriangle}{%
  \tikz[baseline=-0.5ex]\node[
    regular polygon,
    regular polygon sides=3,
    draw=black,
    fill=red,
    line width=0.4pt,
    inner sep=0pt,
    minimum size=0.9em
  ] {};
}

\newcommand{\aOneD}{$10000$}        
\newcommand{\eOneD}{$0.1$}          
\newcommand{\iOneD}{$0^\circ$}            
\newcommand{\OmOneD}{$0^\circ$}           
\newcommand{\omOneD}{$0^\circ$}           
\newcommand{\MzeroOneD}{$0^\circ$}        
\newcommand{\aOneA}{$8000$}         
\newcommand{\eOneA}{$0.1$}          
\newcommand{\iOneA}{$30^\circ$}        
\newcommand{\OmOneA}{$45^\circ$}       
\newcommand{\omOneA}{$45^\circ$}       
\newcommand{\MzeroOneA}{$0^\circ$}        

\definecolor{colort}{rgb}{1.00,0.07,0.65}   

\newcommand{\taken}[1]{\textcolor{black}{#1}}

\usepackage[acronym]{glossaries}
\newacronym{tpbvp}{TPBVP}{Two-Point Boundary Value Problem}

\newacronym{stm}{STM}{State Transition Matrix}
\newacronym{pvt}{PVT}{Primer Vector Theory}

\newacronym{lvlh}{LVLH}{Local Vertical Local Horizontal}

\newacronym{tiot}{TIOT}{Two-Impulse fuel-Optimal rendezvous Transfer}
\newacronym{tiots}{TIOTs}{Two-Impulse fuel-Optimal rendezvous Transfers}
\glsdisablehyper
\DeclareNewFootnote{A}[gobble]
\begin{document}

\title[Families of Two-Impulse Optimal Rendezvous Transfers Between Elliptic Orbits]{Families of Two-Impulse Optimal Rendezvous Transfers Between Elliptic Orbits}

\author[1]{\fnm{Beom} \sur{Park}}\email{park1103@purdue.edu}
\author[1]{\fnm{Kathleen C.} \sur{Howell}}\email{howell@purdue.edu}
\author*[2]{\fnm{Jaewoo} \sur{Kim}}\email{jw.kim@kaist.ac.kr}
\author[2]{\fnm{Jaemyung} \sur{Ahn}}\email{jaemyung.ahn@kaist.ac.kr}

\affil[1]{School of Aeronautics and Astronautics, Purdue University, West Lafayette, IN, 47906}
\affil[2]{Department of Aerospace Engineering, Korea Advanced Institute of Science and Technology, Daejeon, 34141, Republic of Korea}


\abstract{The classical fuel-optimal two-impulse rendezvous problem between Keplerian orbits is revisited from a family-based perspective. Conventional approaches often yield isolated optimal solutions whose mutual relationships remain unclear; yet, when re-parameterized appropriately, seemingly unrelated optima are revealed to be connected members of continuous solution families. To expose this structure, the proposed framework enforces a subset of first-order necessary optimality conditions and traces the resulting one-parameter families via numerical continuation. The families are classified using Hessian-based criteria and Primer Vector Theory, and are projected onto porkchop plots to connect the angular and temporal domains. Representative case studies reveal the emergence, merging, and disappearance of locally optimal branches under variations in orbital geometry, supplying a global map of the solution landscape. This complementary perspective clarifies the robustness of optimal solutions and identifies alternative near-optimal transfers in the vicinity of a nominal trajectory.}

\maketitle

\section{Introduction}\label{sec:intro}

Since the dawn of spaceflight, fuel-optimal trajectories between Keplerian orbits have held both practical and theoretical significance. The most fundamental example is the Hohmann transfer \cite{hohmann1960}, representing the simplest case of a two-impulse maneuver connecting coplanar circular orbits without rendezvous constraints. Any relaxation of these assumptions, whether in orbital geometry, thrust model, or mission constraints, rapidly increases analytical complexity and has motivated decades of research \cite{gobetz1969survey, walsh2020review, morante2021survey}. For instance, generalizations may allow for arbitrary eccentricities or inclinations \cite{vinh1988optimal}. The number of burns may also vary; bi-elliptic transfers, for example, can outperform the Hohmann transfer when large energy or plane changes are required \cite{vallado2022fundamentals}. Introducing finite-thrust propulsion rather than impulsive burns further expands the feasible design space, significantly complicating the optimization problem \cite{lawden1963optimal, longuski2014optimal}. Furthermore, higher-fidelity dynamical models, such as those required for the cislunar environment \cite{holzinger2021primer, saloglu2024acceleration, oshima2024regularizing} or bodies with significant oblateness \cite{yang2022fast}, often render fuel-optimal strategies derived from two-body assumptions insufficient. Conversely, within linearized dynamical regimes, simplified analytical formulations remain valuable and have fostered numerous framework developments, e.g., \cite{prussing1995optimal, carter1991optimal, mcmahon2016linearized}.

Despite these extensive generalizations of fuel-optimal trajectory design, the current analysis returns to one of the most fundamental yet still incompletely understood scenarios: the \acrfull{tiots} between general elliptic, Keplerian orbits. The importance of this configuration cannot be overstated; it represents the simplest mission architecture, characterized by short transfer times and analytical tractability, while still exhibiting rich and nontrivial optimal structures. As noted by Saloglu et al. \cite{saloglu2023existence} and Saloglu and Taheri \cite{saloglu2025classification}, \acrshort{tiots} form the building blocks of more complex multiple-impulse solutions. Similarly, low-thrust transfers may also be continued from these simple impulsive-burn optimal trajectories \cite{zhu2017solving}. Therefore, understanding their organization, bifurcations, and global behaviors remains essential for extending insights to broader mission scenarios.

\taken{Although such transfers are conventionally analyzed in the time domain, where locally optimal solutions often appear as isolated, widely separated points, many of these seemingly unrelated optima are in fact connected members of a single continuous family. This family structure, obscured in the time domain representation, becomes apparent once the transfers are re-parameterized by the angular positions along each orbit. A concrete demonstration is presented in Section~\ref{sec:dimensionality}, after the necessary definitions are established.}\footnote{The quasi-periodic recurrence of favorable transfer geometries is well established in planetary mission design (e.g., Earth--Mars synodic windows) and has long informed heuristic search strategies for \acrshort{tiots} in long-horizon planning contexts, e.g., \citet{bang2018two}; however, a systematic global characterization of the underlying solution structure has remained largely absent.}

Motivated by this \taken{structure}, the current study develops a framework that systematically identifies and traces such families of \acrshort{tiots} between general elliptic orbits. By enforcing a subset of first-order necessary optimality conditions, the methodology reduces the three-dimensional search space to a one-parameter continuation problem, connecting isolated local minima into continuous solution curves. The resulting families reveal the global organization of optimal transfers, including the coalescence, emergence, and disappearance of distinct branches as orbital geometry varies, and supply a natural basis for assessing solution robustness.\footnote{\taken{A conceptually related philosophy, i.e., leveraging homotopy/continuation to expose the global structure of a trajectory solution space rather than isolated optima, also appears in low-thrust trajectory optimization, e.g., \citet{zhang2023solution}.}} By projecting these families onto the temporal domain, the framework further connects to conventional porkchop-plot analyses, bridging geometric insight with epoch-specific mission planning. 

The remainder of this paper is organized as follows. Section \ref{sec:background} introduces the preliminaries, including the Lambert problem formulation and a summary of existing TIOT strategies. Sections \ref{sec:optimization}--\ref{sec:framework} present the proposed methodology: Section \ref{sec:optimization} defines \textit{families} of TIOTs and formalizes the relationships among family members under a generalized notion of optimality; Section \ref{sec:domain} defines two optimality domains (angular and temporal), characterizes their asymptotic behavior, and establishes the correspondence between the two formulations; and Section \ref{sec:framework} describes the numerical procedure used to identify and analyze TIOT families. Section \ref{sec:cases} applies the framework to multiple departure--arrival orbit pairs to illustrate its use and to extract case-based insights. Finally, Section \ref{sec:conclusions} concludes the paper.

\section{\label{sec:background}Background}
This section provides the background for the present study. The Lambert problem, a key building block for TIOT solutions, is first introduced along with the nomenclature and conventions used throughout. Three established classes of TIOT strategies are then reviewed: (1) geometry-based analysis, (2) numerical optimization, and (3) \acrfull{pvt}.

\subsection{Problem Formulation}

The \acrshort{tiot} problem between two elliptic orbits is introduced. The fixed orbital elements of the departure and arrival ellipses are denoted as,
\begin{align}
    \oet{j} = \left[\sma{j}, \ecc{j}, \inc{j}, \raan{j}, \aop{j}\right]^\intercal,\quad \text{for }j=1,2,
\end{align}
where the symbols $a, e, i, \Omega, \omega$ represent the classical orbital elements, i.e., semi-major axis, eccentricity, inclination, right ascension of the ascending node, and the argument of periapsis, respectively, defined within an arbitrary inertial frame. While the formulation extends to any central body, the current investigation primarily focuses on Earth-centered scenarios. Subscripts $1,2$ correspond to the departure and arrival, respectively. The sixth element specifies the location of the spacecraft along the orbit. Mean anomalies are mainly adopted in the current work, i.e.,
\begin{align}
    \label{eq:M_1}& \ma{1} = n_1\t + \ma{1}^0, \\
    \label{eq:M_2^dep}& \mad{2} = n_2\t + \ma{2}^0, \\
    \label{eq:M_2}& \ma{2} = n_2(\t+\tof) + \ma{2}^0,
\end{align}
where $n_i$ is the mean motion supplied from $\sma{i}$, $\t$ is time measured from an arbitrary reference time, and $\ma{i}^0$ is the mean anomaly at the arbitrary reference time. The two-impulsive burn rendezvous \taken{maneuver} leads to one transfer conic; denote the transfer time (time-of-flight) as $\tof$. Note that $\mad{2}, \ma{2}$ refer to mean anomalies upon the departure time ($\t$) and arrival time ($\t+\tof$), respectively. Thus, the transfer arc links the position vectors along the departure and arrival ellipses at,
\begin{align}
    \label{eq:r1}\pos{1} &= \pos{1}(\t) = \pos{1}(T; \oet{1}, \ma{1}^0 ), \\
    \label{eq:r2}\pos{2} & = \pos{2}(T +\tof) = \pos{2}(T + \tof; \oet{2}, \ma{2}^0 ),
\end{align}
defined in an inertial frame. Finding a transfer arc satisfying the boundary constraints (Eqs. \eqref{eq:r1}-\eqref{eq:r2}) under Keplerian dynamics with a specified time-of-flight $\tof$ constitutes the Lambert problem \cite{battin1999introduction}. It solves for a transfer conic $\la$ defined as,
\begin{align}
    \label{eq:la} \la = \la(\pos{1}, \pos{2}, \tof; \mu, N, d_1, d_2).
\end{align}
The gravitational parameter of the central celestial body is denoted as $\mu$. The number of revolutions is $N$, defined as $N = \lfloor \theta/(2\pi)\rfloor $ with $\theta$ denoting the change in true anomaly along the Lambert arc $\la$. The switch $d_1$ resolves the ambiguity in $\theta$ and $2\pi - \theta$ for a given set of $\pos{1}$ and $\pos{2}$; $d_1$ ($\in \{\mathrm{s}, \mathrm{l}\}$) denotes short and long transfers with $0 < \theta - 2N\pi < \pi$ and $\pi < \theta - 2N\pi < 2\pi$, respectively. The last parameter $d_2$ ($\in \{0, 1\}$) resolves the additional geometrical ambiguity that arises in transfers with $N \geq 1$; for a given $N$, two distinct transfer ellipses connect the same boundary positions in the same flight time, differing in their semi-major axes. Following Russell \cite{russell2019solution}, $d_2$ distinguishes these two solutions by the location of the second focus of the transfer ellipse relative to the chord line. Then, the cost of the two-impulsive burn rendezvous \taken{maneuver} renders,
\begin{align}
    \label{eq:cost} J : = \norm{\dv{1}} + \norm{\dv{2}},
\end{align}
where $1,2$ are subscripts for the burns at the departure and final orbits \taken{and $\norm{\cdot}$ denotes the Euclidean norm of a vector}. The difference in the velocity vector constitutes the impulsive burns as,
\begin{align}
    \dv{1} = \vella{1} - \vel{1}, \quad \dv{2} = \vel{2} - \vella{2},
\end{align}
where $\vel{1}$ and $\vel{2}$ are velocity vectors along the departure and arrival orbits and $\vella{1}$ and $\vella{2}$ are velocity vectors along the Lambert arc $\la$ at the initial and final times, i.e., $\t$ and $\t + \tof$, respectively. For a fixed set of departure and arrival orbital elements, the cost function consists of the following independent variables and parameters,
\begin{align}
    J = J(\t, \tof; \oet{1}, \ma{1}^0, \oet{2}, \ma{2}^0, \mu, N, d_1, d_2).
\end{align}
It is typical to analyze different combinations for $N, d_1, d_2$ as they constitute discrete parameters. In the current analysis, $N = 0$ is analyzed, establishing the simplest case. For brevity, the parameters ($\oet{1}, \ma{1}^0, \oet{2}, \ma{2}^0, \mu, N, d_2$) are omitted in the expressions in the rest of the document as they are fixed values in the current analysis, leaving dependency only on $d_1 \in \{s, l\}$. Utilizing the relationship between the time and angles from Eqs. \eqref{eq:M_1}-\eqref{eq:M_2}, the cost is alternatively represented as,
\begin{align}
    \label{eq:cost_parameters} J = J(\t,\tof; d_1) = J(\ma{1}, \ma{2}, \tof; d_1) = J(\ma{1}, \mad{2}, \tof; d_1).
\end{align}
The alternative domains with angles aid the extraction of patterns, as demonstrated later. Consequently, the problem of locating \acrshort{tiots} is defined as finding the set of independent variables that (locally) minimizes the cost function $J$.

Lambert solvers serve as the foundation for any two-impulse rendezvous problem, since Eq. \eqref{eq:cost} requires solving Eq. \eqref{eq:la}. Although a general closed-form solution does not exist, Keplerian dynamics allows for geometric and semi-analytical exploration. While a comprehensive review lies beyond the scope of this work, notable contributions warrant acknowledgment. Research in this domain generally centers on two aspects: (1) computational efficiency and (2) algorithmic robustness. Singularities within the Lambert problem often impede the successful construction of the solution, motivating continuous improvements toward faster and more robust formulations. Following Lambert's initial formulation and Lagrange's derivation of the time-of-flight equations, modern contributions emerge from Lancaster and Blanchard \cite{lancaster1970solution}. Battin \cite{battin1999introduction} emphasizes iterative techniques with improved convergence properties. Gooding \cite{gooding1990procedure} develops a fast, robust method with guaranteed convergence using a universal variable approach, while Izzo \cite{izzo2015revisiting} proposes an efficient algorithm utilizing variable transformation and root-finding on a bounded domain. This study adopts Izzo's \cite{izzo2015revisiting} formulation for the internal Lambert arc routine. \taken{Although not employed in the present study, recent advances in solving the Lambert problem continue to emerge, e.g., \citet{negrete2025exact} as well as \citet{mcelreath2025universal}.}

\subsection{Locating \texorpdfstring{\acrshort{tiots}}{}: Existing Strategies}
Existing literature on \acrshort{tiots} can be broadly categorized into three approaches: (1) geometry-based analysis, (2) numerical optimization, and (3) \acrfull{pvt}. This section reviews each in turn, highlighting their respective contributions and limitations.

\subsubsection{Geometry-Based Analysis}
By harnessing the geometric properties of the departure, arrival, and transfer orbits, this approach aims for a (semi-)analytical characterization of \acrshort{tiots}. A divide-and-conquer strategy is typically adopted, starting with simplifying assumptions that are gradually lifted in subsequent research. For example, the Hohmann transfer \cite{hohmann1960} constitutes the simplest scenario with two coplanar, circular orbits, admitting a proven globally optimal solution \cite{hazelrigg1984globally, prussing1992simple, palmore1984elementary}. Transfers involving inclined circular orbits are discussed in Baker \cite{baker1966orbit}, while optimal transfers between coplanar ellipses are detailed in Pontani \cite{pontani2009simple} and Zaborsky \cite{zaborsky2014analytical}. In these analyses, bi-elliptic transfers serve as a competing option, augmenting the optimality landscape. However, to the authors' best knowledge, analytical criteria for locating optimal transfers within the most general formulation, involving non-coplanar, non-coaxial ellipses, remain elusive. In such complex configurations, geometric heuristics are often examined as proxies for true optimal solutions. For example, Vinh et al. \cite{vinh1988optimal} investigate the optimal ``nodal'' transfer, where two impulsive burns are placed on the common line of nodes; such a configuration is also analyzed in Baker \cite{baker1966orbit} and denoted as ``modified Hohmann transfer.'' Likewise, Lacruz and {\'A}ngeles \cite{lacruz2024coaxial} examine transfer geometries with certain properties, e.g., coaxial with the departure or arrival orbits.

\subsubsection{Numerical Methods: Grid Search and Optimization} 

Advances in Lambert solvers \cite{gooding1988solution, battin1984elegant, izzo2015revisiting, russell2019solution} enable large-scale parametric searches and visualizations such as porkchop plots. The porkchop plot serves as the de facto graphical strategy for surveying the solution space. The grid typically spans the $(\t, \tof)$ domain, or equivalently, the $(\t, \t+\tof)$ domain, where the Lambert arcs and resulting costs $J$ are evaluated iteratively. For example, \taken{consider} the baseline transfer scenario (Table~\ref{table:sample_orbits_first})\footnote{Unless otherwise stated, all numerical examples and family analyses in this study are referenced to the baseline transfer case defined in Table~\ref{table:sample_orbits_first}.}. The synodic period of the two orbits is approximately $P_{syn} = 2\pi / (\mm{2} - \mm{1}) \approx 7$ hours. The porkchop plot in Fig.~\ref{fig:porkchop_2d_3d_2d.png} is generated for $N = 0$ and $d_1 = \mathrm{l}$ (long transfer, $\pi < \theta < 2\pi$), with $0 \leq \t \leq 2P_{syn}$ and $0.2 \leq \tof \leq 4$ hours over a $200 \times 100$ grid. The colormap represents the cost $J$ defined in Eq. \eqref{eq:cost}, where darker blue regions correspond to locally optimal Lambert arcs connecting the two orbits. Although grid search provides a global view of the cost landscape, its limited sampling resolution may not capture all locally optimal solutions, and the resulting patterns offer limited insight into the connectivity and continuity of optimal configurations.

\begin{table}[htpb]
\centering
\caption{\label{table:sample_orbits_first}
Orbital elements defining the baseline transfer case between two Keplerian orbits in an Earth-centered inertial frame (notation is formally defined in Section~\ref{sec:background})}
\begin{tabular*}{\textwidth}{@{\extracolsep\fill}lcccccc}
\toprule
Orbit & $a$ [km] & $e$ [-] & $i$ [deg] & $\Omega$ [deg] & $\omega$ [deg] & $M^{0}$ [deg] \\
\midrule
Orbit 1 (Departure) & \aOneD & \eOneD & \iOneD & \OmOneD & \omOneD & \MzeroOneD \\
Orbit 2 (Arrival)   & \aOneA & \eOneA & \iOneA & \OmOneA & \omOneA & \MzeroOneA \\
\bottomrule
\end{tabular*}
\end{table}

 \begin{figure}[h!]
    \centering
    \begin{subfigure}[b]{0.98\textwidth}
        \centering
        \includegraphics[page=1, width=0.99\textwidth]{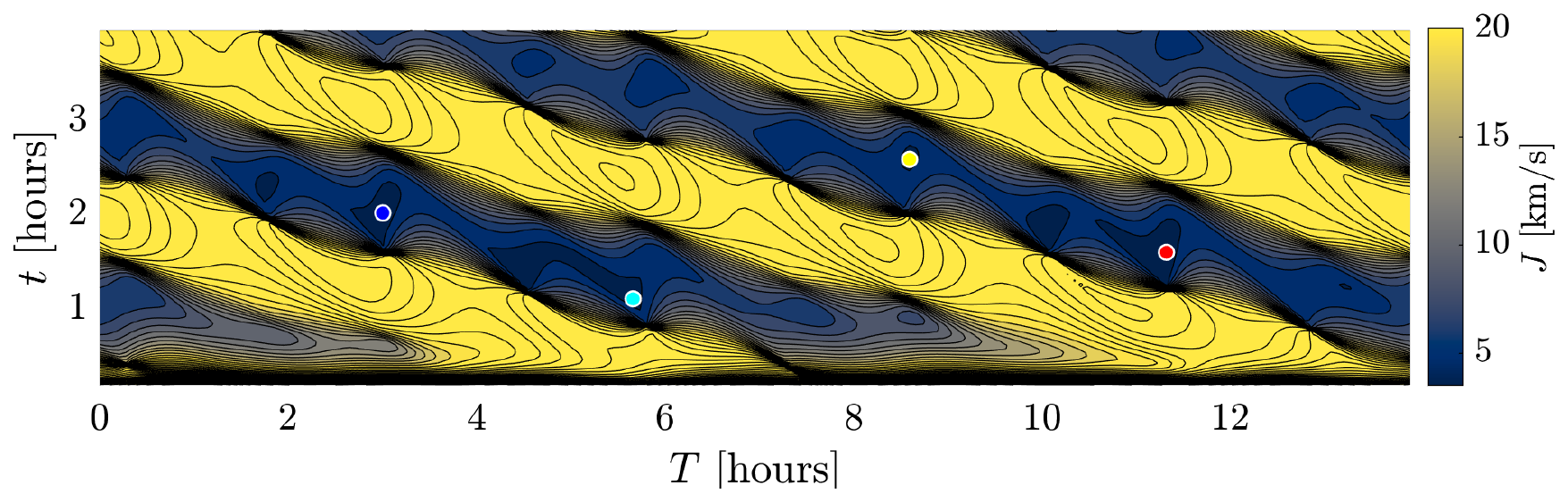}
        \caption{Cost landscape over departure epoch and flight time (markers indicate locally near-optimal transfers\taken{, with $J = 4.30$, $4.50$, $4.54$, and $3.87$~km/s for the blue, cyan, yellow, and red markers, respectively}).}
        \label{fig:porkchop_2d_3d_2d.png}
    \end{subfigure}
    \begin{subfigure}[b]{0.48\textwidth}
        \centering
        \includegraphics[page=3, width=0.99\textwidth]{figures_everything.pdf}
        \caption{Selected locally near-optimal transfers in the Earth-centered inertial frame (all trajectories are prograde, i.e., $i < 90^\circ$).}
        \label{fig:porkchop_2d_3d_transfer_geometry.png}
    \end{subfigure}
    \begin{subfigure}[b]{0.48\textwidth}
        \centering
        \includegraphics[page=2, width=0.99\textwidth]{figures_everything.pdf}
        \caption{Cost landscape re-parameterized by departure and arrival mean anomalies ($\ma{1}$, $\mad{2}$) and flight time.}
        \label{fig:porkchop_2d_3d_3d.png}
    \end{subfigure}
  \caption{Cost landscape and selected transfer geometries for the baseline transfer scenario (Table~\ref{table:sample_orbits_first}).}
  \label{fig:porkchop_2d_3d}
\end{figure}

Various optimization approaches supply another numerical route for locating \acrshort{tiots}. In early works, McCue \cite{mccue1963optimum} and Lee \cite{lee1964analysis} partially enforce necessary optimality conditions to generate contours and analytical insights, while McCue \cite{mccue1965numerical} applies a gradient descent algorithm to locate locally optimal solutions. These optimization efforts are closely tied to geometry-based analysis, as the necessary optimality conditions yield polynomial expressions derived from orbital geometry \cite{stark1961optimum, lee1964analysis}. More recently, \citet{bang2018two} employ a gradient-based search initialized from grid samples, and meta-heuristic optimization schemes \cite{abdelkhalik2007n, pontani2012particle} \taken{to} reduce reliance on the initial grid. Gradient-based methods \cite{mccue1965numerical, jezewski1968efficient} typically offer faster convergence but risk entrapment in local minima, whereas meta-heuristic approaches increase the likelihood of identifying globally optimal solutions at a higher computational cost. However, despite their flexibility, these numerical approaches often provide isolated ``point'' solutions, leaving the physical origins of recurring patterns in the cost landscape opaque.

\subsubsection{\label{sec:pvt}\texorpdfstring{\acrfull{pvt}}{}}

Lawden's \acrshort{pvt} \cite{lawden1963optimal} supplies an analytical framework for verifying the optimality of impulsive trajectories, complementary to the geometry-based and numerical approaches above. Under the impulsive thrust assumption, the costate vector $\left[\csr^\intercal, \csv^\intercal\right]^\intercal$ is augmented to the state vector, i.e., $\left[\bm{r}^\intercal, \bm{v}^\intercal \right] ^\intercal$. Then, the ``primer vector'' is defined as $\primer = -\csv$. According to Lawden's theory \cite{lawden1963optimal}, an optimal impulsive trajectory must satisfy four necessary conditions regarding $\primer$ and its time derivative $d{\primer}/d\tau$\footnote{The generic independent variable for time is $\tau$, differentiated from the transfer parameters $\t, \tof$.}: (1) $\primer$ and $d{\primer}/d\tau$ must be continuous throughout the transfer, (2) the magnitude of the primer vector must not exceed unity ($|\primer| \leq 1$) at any point along the trajectory, (3) at the instants of impulsive burns, the primer vector must be a unit vector aligned with the thrust direction (i.e., $\primer = \Delta\bm{v}/\norm{\Delta\bm{v}}$ and $\norm{\primer}=1$), and (4) if an intermediate coast arc exists between impulses, the derivative of the primer vector magnitude must vanish ($d\norm{\primer}/d\tau = 0$) at any interior point where $\norm{\primer}=1$. While applicable to general dynamical regimes, the specific equations of motion for Keplerian dynamics are briefly recalled as,
\begin{align}
    \label{eq:pvt_dynamics1}& \frac{d\bm{r}}{d\tau} = \bm{v}, \\
    \label{eq:pvt_dynamics2}& \frac{d\bm{v}}{d\tau} = \bm{g}(\bm{r}) = - \frac{\mu}{r^3}\bm{r}.
\end{align}
From optimal control theory \cite{longuski2014optimal}, the costates evolve according to 
\begin{align}
    \label{eq:pvt_dynamics3}& \frac{d\csr}{d\tau} = -\bm{G}\csv, \\
    \label{eq:pvt_dynamics4}& \frac{d\csv}{d\tau} = -\csr,
\end{align}
where $\bm{G} = \frac{\partial \bm{g}}{\partial \bm{r}} = \frac{\mu}{r^5}(3\bm{r} \bm{r}^\intercal - r^2 \bm{I})$ and $\bm{I}$ is a three-by-three identity matrix. The distinct advantage of \acrshort{pvt} over parametric optimization lies in determining the optimal number of impulses. While the parametric framework typically assumes a fixed number of burns, \acrshort{pvt} provides analytical criteria indicating whether intermediate burns would further reduce fuel costs. Building on Lawden's foundation, the literature extends \acrshort{pvt} to multiple-impulse transfers \cite{small1971minimum}, multi-revolution cases \cite{prussing2000class}, and the convergence of finite-thrust trajectories toward impulsive solutions; Taheri and Junkins \cite{taheri2020how} demonstrate that finite-thrust solutions converge to these impulsive limits, confirming the physical significance of the necessary conditions. Consequently, this analysis leverages \acrshort{pvt} as a complementary tool to post-assess the optimality of converged solutions from parametric optimization. However, \acrshort{pvt} remains inherently local; it verifies the optimality of individual trajectories but does not directly reveal the global topology or connectivity of solution families.

\section{\label{sec:optimization}\taken{Solution Dimensionality and One-Parameter Families}}

\taken{Building on the problem formulation in Section~\ref{sec:background}, this section examines the dimensionality of the solution space and introduces a continuation scheme for tracing one-parameter families of \acrshort{tiots}.} Section~\ref{sec:dimensionality} motivates the adoption of angular variables over the conventional time domain, Section~\ref{sec:cont_scheme} defines the continuation scheme based on partial optimality conditions, and Section~\ref{sec:gradient} details the required gradient information and the Hessian-based classification of stationary solutions.

\subsection{\label{sec:dimensionality}\taken{Motivation: Reducing the Search Dimensionality}}

The optimization of two-impulse transfers between elliptic orbits is inherently governed by three design variables \cite{mccue1963optimum}. This three-dimensional domain consists of the (1) departure and (2) arrival locations along the respective orbits and the (3) geometry of the transfer arc, often parameterized by the semi-latus rectum or, equivalently, the time-of-flight. In the present formulation, this three-dimensional domain is represented by the state vector $\bm{X} = [\ma{1}, \ma{2}, \tof]^\intercal$. Consequently, the cost $J$ constitutes a scalar field over this three-dimensional space.

As discussed in Section~\ref{sec:background}, directly solving the full optimization problem (i.e., enforcing $\nabla J = \bm{0}$) yields isolated point solutions, offering limited insight into the global solution structure. Beyond this limitation, the choice of independent variables fundamentally dictates the representation of these optimal structures. While the two angles $\ma{1}$ and $\ma{2}$ (or, $\mad{2}$) are bounded between $[0, 2\pi)$, such angles unravel as a function of $\t$ \taken{(or, a function of both $\t$ and $\tof$ for $\ma{2}$)}, recalling the alternative representations of $J$ from Eq. \eqref{eq:cost_parameters}. Consequently, the typical time domain, $(\t, \tof)$, is often inadequate for revealing the underlying connectivity of solution families. The unraveling of the periodic angular domain into an infinite time axis effectively fragments continuous structures into isolated points, obscuring their relationships. In the time domain, the locally optimal solutions satisfy the first-order necessary conditions,
\begin{align}
    \frac{\partial J}{\partial \t} = \frac{\partial J}{\partial \tof} = 0.
\end{align}
\taken{This behavior is illustrated for the baseline scenario in Fig.~\ref{fig:porkchop_2d_3d}. Figure~\ref{fig:porkchop_2d_3d_2d.png} depicts the cost $J$ over departure time and flight time, on which four locally optimal solutions (colored markers) appear widely separated in the $(\t,\tof)$ space. However, when re-parameterized by the angular positions along each orbit (Fig.~\ref{fig:porkchop_2d_3d_3d.png}), these solutions collapse into a compact neighborhood, confirming that they belong to a continuous family with shared transfer geometries (Fig.~\ref{fig:porkchop_2d_3d_transfer_geometry.png}).} This observation motivates the present framework to navigate the solution space directly within the $\ma{1}-\ma{2}-\tof$ domain.

\subsection{\label{sec:cont_scheme}Families of \texorpdfstring{\acrshort{tiots}}{}}

Building on the motivation from Section \ref{sec:dimensionality}, a numerical continuation scheme is devised to trace one-dimensional families of \acrshort{tiots} within the three-dimensional space. This process leverages a system of nonlinear equations defined over the state vector, $\fv = [\ma{1}, \ma{2}, \tof]^ \transpose$. To generate a continuous family, strictly two optimality conditions are enforced, leaving one degree of freedom unconstrained. With $x_{1} = x_{1}(\bm{X})$ and $x_{2} = x_{2}(\bm{X})$ denoting two linearly independent variables\footnote{The variables $x_{1}$ and $x_{2}$ are intentionally left generic here; their concrete assignments (e.g., $x_1 = \ma{1}, x_2 = \ma{2}$ or $x_1 = \t, x_2 = \tof$) are specified in Section~\ref{sec:domain}.}, the constraint vector is defined as,
\begin{align}
    \label{eq:c}\c(\fv) & = \left[ \frac{\partial J}{\partial x_1}, \frac{\partial J}{\partial x_2} \right]^\transpose = \bm{0},
\end{align}
enforcing the necessary stationary conditions. Then, $\fv, \c$ represent an underdetermined system with 2 equations and 3 unknowns.  By relaxing one of the three optimality conditions, the solution expands from an isolated point to a one-dimensional family parameterized by the remaining variable\footnote{This approach is inspired by \citet{mccarthy2023rephasing} and also leveraged in \citet{gul2024feature}; these prior investigations focus on transfers within non-Keplerian dynamics.}. Given a solution $\fv^*$ that satisfies $\c(\fv^*)= \bm{0}$, the Jacobian of the constraints with respect to the state vector renders,
\begin{align}
    \label{eq:jacobian} \frac{\partial \c}{\partial \fv} \biggr \rvert_{\fv^*}.
\end{align}
Being a matrix of size $2 \times 3$, the Jacobian generally allows a one-dimensional null space. This null vector ($\mathcal{N}(\fv^*)$) informs the linear prediction for the nearby solution along the family as,
\begin{align}
    \label{eq:guess} \fv^\mathrm{p} = \fv^* + ds\cdot \mathcal{N} (\fv^*),
\end{align}
where the superscript $\mathrm{p}$ stands for the ``predicted'' state and $ds$ is a scalar step size along the null space direction. To resolve this prediction into a valid solution, the constraints in Equation \eqref{eq:c} are enforced again via a differential corrections process (e.g., Newton-Raphson method). To render a square system, an ``arclength'' constraint may be enforced as well, 
\begin{align}
    F^A = (\fv - \fv^* )\cdot \mathcal{N} (\fv^*) - ds = 0,
\end{align}
explicitly controlling the step size along the null space direction. For more information regarding this ``pseudo-arclength'' continuation, refer to \citet{seydel2010practical}. 

\subsection{\label{sec:gradient}Gradient Information}

The continuation scheme (Eqs.~\eqref{eq:jacobian}--~\eqref{eq:guess}) requires derivatives of the cost function $J$ with respect to the state vector $\fv$. The first-order derivatives,
\begin{align}
    \label{eq:1st_deriv}\nabla J = \frac{\partial J}{ \partial \fv } = \left[ \frac{\partial J }{\partial \ma{1}} , \frac{\partial J }{\partial \ma{2}} , \frac{\partial J }{\partial \tof} \right]^\transpose,
\end{align}
are evaluated via the \acrfull{stm} along the Lambert arc, following the approach of \citet{schumacher2015uncertain} and \citet{arora2015partial}. Second-order derivatives of $J$ are also required for two purposes: (1) constructing the Jacobian $\partial \c / \partial \fv$ in Eq.~\eqref{eq:jacobian}, and (2) classifying stationary solutions. Derivations for both the first- and second-order derivatives are supplied in \hyperref[app:gradient]{Appendix~A}. Regarding the second-order classification, stationary solutions satisfying Eq.~\eqref{eq:c} may correspond to a minimum, maximum, or saddle of $J$, as determined by the eigenvalues $\lambda_A$ and $\lambda_B$ of the Hessian matrix,
\begin{align}
    \label{eq:hessian} \bm{H} = \mb \frac{\partial ^2J}{\partial x_1^2} & \frac{\partial ^2J}{\partial x_1 \partial x_2} \\ \frac{\partial ^2 J}{\partial x_1 \partial x_2} & \frac{\partial ^2 J}{\partial x_2^2} \me.
\end{align}
Specifically, the arc is classified as minimum ($\lambda_A,\lambda_B > 0$), maximum ($\lambda_A,\lambda_B < 0$), or saddle ($\lambda_A \lambda_B < 0$). Throughout this analysis, different colors denote the Hessian classification as in Table~\ref{table:colormap}. The term \acrshort{tiot} is utilized to collectively denote all stationary solutions satisfying the first-order optimality conditions, regardless of their second-order classification.

\begin{table}[htpb]
\centering
\caption{\label{table:colormap}Classification of stationary solutions based on Hessian information}
\begin{tabular}{p{2cm}c}
\toprule
Optimality & Color \\\midrule
Minimum & \textcolor{colMin}{\ding{108}} \\
Maximum & \textcolor{colMax}{\ding{108}} \\
Saddle & \textcolor{colSad}{\ding{108}} \\
\bottomrule
\end{tabular}
\end{table}

\section{\label{sec:domain}Optimality Domains and  Asymptotic Behaviors}

Since the relaxed-constraint formulation admits two independent variables $x_1$ and $x_2$ (Eq.~\eqref{eq:c}), the optimization problem admits multiple domain representations. This section investigates two such selections: the ``angular domain'' ($x_1=\ma{1}$, $x_2=\ma{2}$), offering analytical insights into the global topology and asymptotic behavior of optimal families, and the ``temporal domain'' ($x_1=\t$, $x_2=\tof$), directly applicable to mission design and porkchop-plot analysis. The two domains are presented sequentially, followed by a discussion of their asymptotic equivalence.

\subsection{Angular Domain (\texorpdfstring{$\ma{1}-\ma{2}$}{})}

Physically, optimization over the angular domain assumes that any geometric configuration of $\ma{1}$ and $\ma{2}$ is accessible given a sufficient time horizon. This formulation effectively decouples the geometric optimization from the specific epoch or waiting time ($\t$). Consequently, the optimality constraints result in,
\begin{align}
    \label{eq:angular_optimality} \left[ \frac{\partial J}{\partial\ma{1}} ,  \frac{\partial J}{\partial\ma{2}}\right]^\intercal = \bm{0}.
\end{align}
While this formulation renders the angular domain somewhat abstract from immediate mission timelines, it offers profound insights into the asymptotic behaviors of optimal families at the limits $\tof \rightarrow 0, \infty$. The main text focuses on the $\tof\rightarrow\infty$ case; for details regarding $\tof\rightarrow0$, refer to \hyperref[app:t0]{Appendix B}. 

As the number of revolutions for the Lambert arc is fixed ($N = 0$ in the current work), the limit $\tof\rightarrow\infty$ implies that the transfer arc asymptotically approaches a parabolic trajectory with respect to the central celestial body, i.e., 
\begin{align}
    \label{eq:la_m}\la \rightarrow \la_{\mathrm{para}} (\ma{1}, \ma{2}; d_1)
\end{align}
Compared to Eq. \eqref{eq:la}, the dependency on $\tof$ is omitted as $\tof \rightarrow \infty$. Thus, optimal families initiate (or terminate) at pairs of $\ma{1}, \ma{2}$ where the cost $J$ constructed with parabolic arcs, is locally stationary; the velocities $\vella{1}$ and $\vella{2}$ may be constructed semi-analytically following the process detailed in \hyperref[app:parabola]{Appendix C}. Note that while the time-dependence is removed, the parabolic Lambert arcs retain dependencies on the geometric parameter, $d_1$. Figures \ref{fig:tinf_short} and \ref{fig:tinf_long} present numerical experiments for short and long transfers, respectively, using the baseline scenario from Table \ref{table:sample_orbits_first}. In Figs. \ref{fig:tinf_short_m1m2.png} and \ref{fig:tinf_long_m1m2.png}, contour plots are generated with a representatively large time-of-flight value of $\tof = 25\cdot10^3$ seconds\footnote{For the baseline scenario, this value corresponds to approximately $2.5\,P_1$ (or $3.5\,P_2$), where $P_{1}$ and $P_{2}$ denote the orbital periods of the departure and arrival orbits. The specific choice is empirical; any comparably large value yields qualitatively identical contour structures, consistent with the asymptotic convergence toward the parabolic limit.} to approximate the behavior of $\tof \rightarrow \infty$. For comparison, Figs. \ref{fig:tinf_short_m1m2_analytic.png} and \ref{fig:tinf_long_m1m2_analytic.png} display the contours derived analytically using parabolic Lambert arcs via the process in \hyperref[app:parabola]{Appendix C}. The identified extremals are summarized in Table \ref{table:tinf}. \taken{A noteworthy observation is that, within this asymptotic limit, the short-transfer extremals are exclusively saddle points, whereas the long-transfer case additionally admits minima and a maximum (Table~\ref{table:tinf}). Whether this asymmetry is specific to the present geometry or reflects a more general property warrants further investigation.} In denoting each asymptotic case of the optimal Lambert families, the following notation is introduced,
\begin{align}
    \label{eq:asymptote_notation} \infty^{a}_{b,c},
\end{align}
where $a$ ($\in\left\{\mathrm{s}, \mathrm{l}\right\}$) denotes the short and long transfer scenarios. The optimality is classified in $b$ with $m, M, \sigma$ signifying the local minimum, maximum, and saddle, respectively. Finally, $c \geq 1$ is the index to enumerate multiple cases under the same category. Two sample Lambert arcs are illustrated in Fig. \ref{fig:tinf_lambert_arc}. Crucially, although a finite $\tof$ is selected for the numerical plots, the topological structures persist as they asymptotically approach the parabolic limit. Consequently, the extrema identified in this parabolic limit (Table \ref{table:tinf}) are not merely theoretical curiosities; they serve as robust initial seeds for the continuation process described in Section \ref{sec:cont_scheme}. By initializing at these asymptotic solutions, the optimal families can be systematically traced from the $\tof \rightarrow\infty$ regime back into the finite time-of-flight domain.

 \begin{figure}[htb!]
    \centering
    \begin{subfigure}[b]{0.48\textwidth}
        \centering
        \includegraphics[page=4, width=0.99\textwidth]{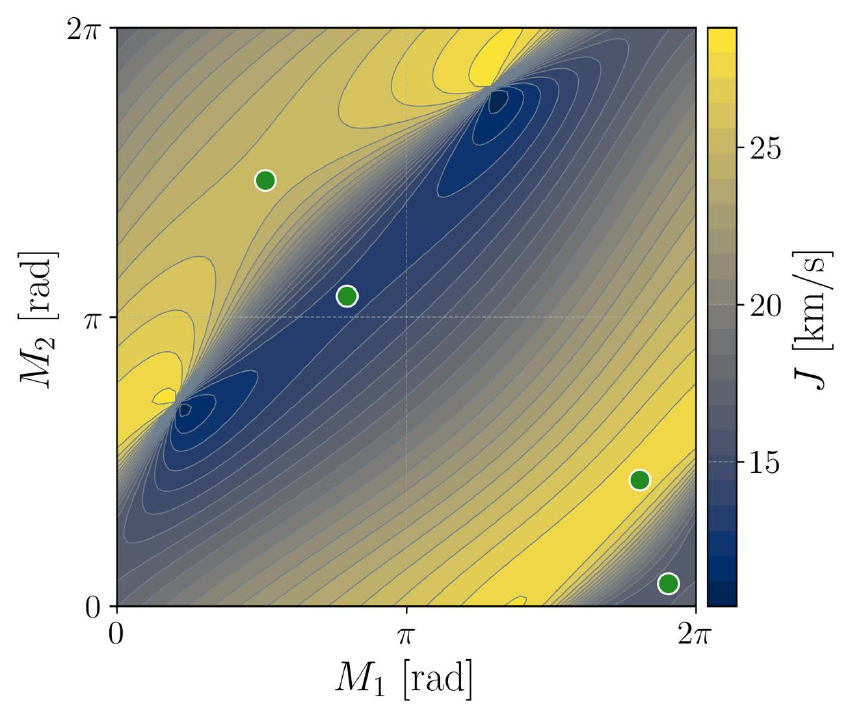}
        \caption{$\tof = 25\cdot10^{3}$ seconds.}
        \label{fig:tinf_short_m1m2.png}
    \end{subfigure}
    \begin{subfigure}[b]{0.48\textwidth}
        \centering
        \includegraphics[page=5, width=0.99\textwidth]{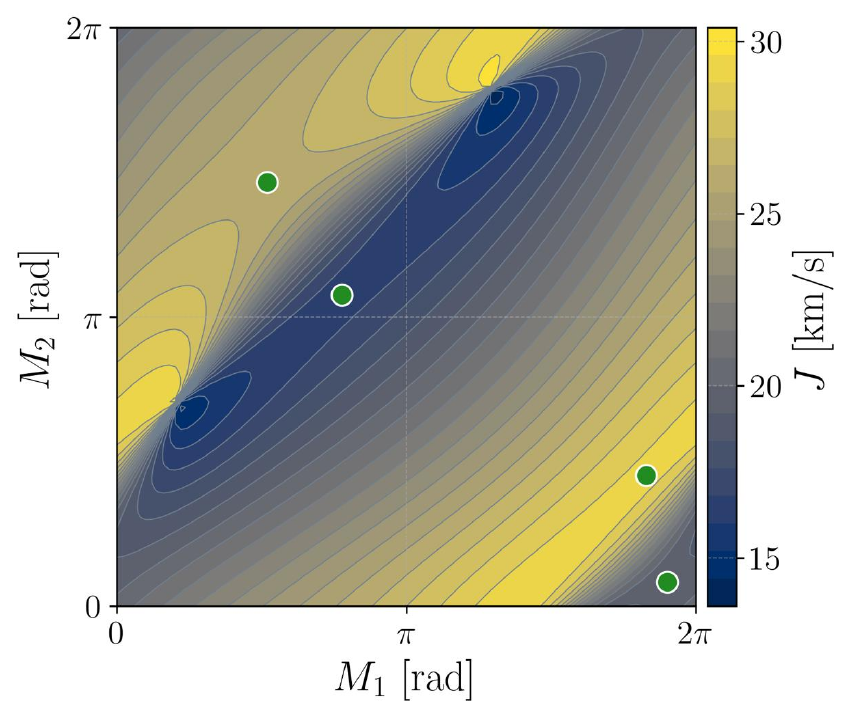}
        \caption{Parabolic Lambert arcs ($\tof \rightarrow \infty$).}
        \label{fig:tinf_short_m1m2_analytic.png}
    \end{subfigure}
  \caption{Comparison of cost landscapes in the angular domain ($d_1=\mathrm{s}$\taken{, short transfer}) as $\tof \rightarrow \infty$. Markers indicate stationary points classified by Hessian information (Table \ref{table:colormap}).}
  \label{fig:tinf_short}
\end{figure}

 \begin{figure}[hbt!]
    \centering
    \begin{subfigure}[b]{0.48\textwidth}
        \centering
        \includegraphics[page=6, width=0.99\textwidth]{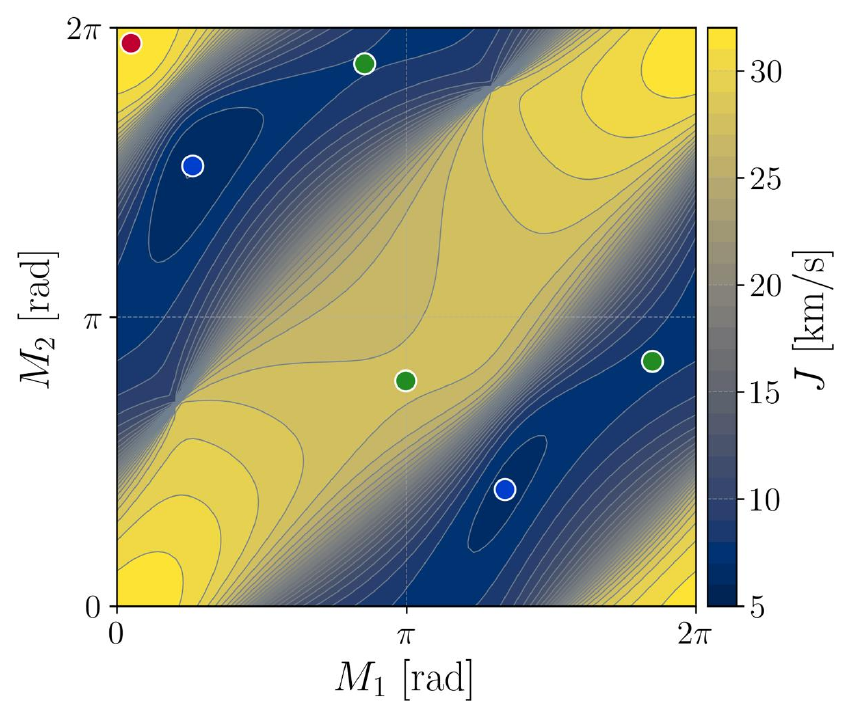}
        \caption{$\tof = 25\cdot10^{3}$ seconds.}
        \label{fig:tinf_long_m1m2.png}
    \end{subfigure}
    \begin{subfigure}[b]{0.48\textwidth}
        \centering
        \includegraphics[page=7, width=0.99\textwidth]{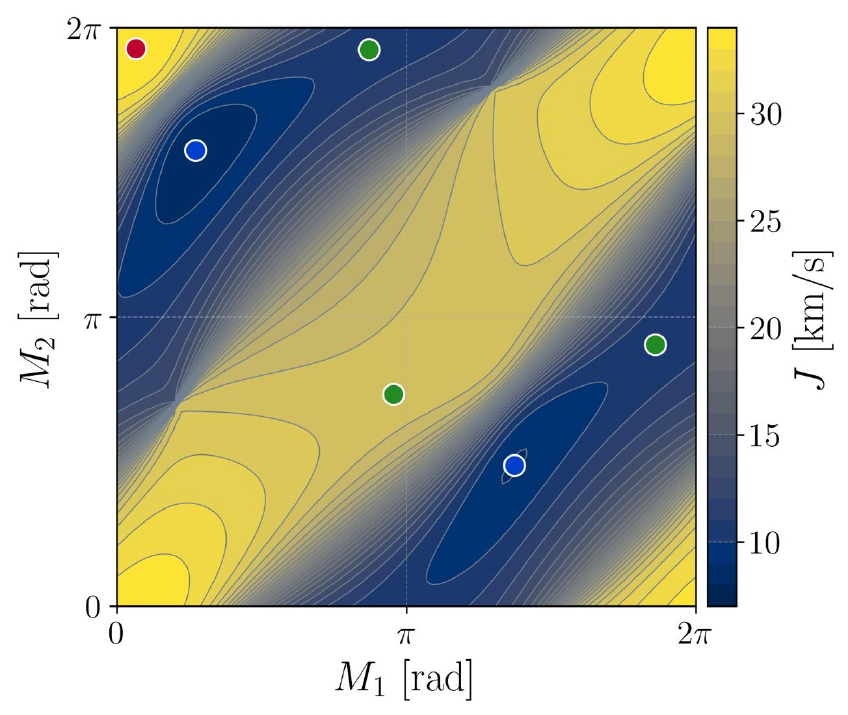}
        \caption{Parabolic Lambert arcs ($\tof \rightarrow \infty$).}
        \label{fig:tinf_long_m1m2_analytic.png}
    \end{subfigure}
  \caption{Comparison of cost landscapes in the angular domain ($d_1=\mathrm{l}$\taken{, long transfer}) as $\tof \rightarrow \infty$. Markers indicate stationary points classified by Hessian information (Table \ref{table:colormap}).}
  \label{fig:tinf_long}
\end{figure}

\begin{table}[htpb]
\centering
\caption{\label{table:tinf}Asymptotic stationary points ($\tof \rightarrow \infty$) in the angular domain for the baseline scenario (Table \ref{table:sample_orbits_first})}
\begin{tabular}{p{2cm} c c c}
\toprule
Label & $\ma{1}$ [rad] & $\ma{2}$ [rad] & \taken{$J$ [km/s]}\\\midrule
$\infty^\mathrm{s}_{\sigma,1}$ & $1.63$  & $4.60$ & \taken{26.24}\\
$\infty^\mathrm{s}_{\sigma,2}$ & $5.74$ & $1.42$ & \taken{29.02}\\
$\infty^\mathrm{s}_{\sigma,3}$ & $2.44$ & $3.38$ & \taken{16.48}\\
$\infty^\mathrm{s}_{\sigma,4}$ & $5.98$ & $0.26$ & \taken{19.38}\\
$\infty^\mathrm{l}_{M,1}$ & $0.21$  & $6.05$ & \taken{33.80}\\
$\infty^\mathrm{l}_{\sigma,1}$ & $3.00$ & $2.30$ & \taken{29.53}\\
$\infty^\mathrm{l}_{\sigma,2}$ & $2.73$  & $6.05$ & \taken{10.32}\\
$\infty^\mathrm{l}_{\sigma,3}$ & $5.84$ & $2.84$ & \taken{10.25}\\
$\infty^\mathrm{l}_{m,1}$ & $0.85$ & $4.95$ & \taken{7.99}\\
$\infty^\mathrm{l}_{m,2}$ & $4.31$ & $1.53$ & \taken{8.95}\\\bottomrule
\end{tabular}
\end{table}

 \begin{figure}[hbt!]
    \centering
    \begin{subfigure}[b]{0.48\textwidth}
        \centering
        \includegraphics[page=8, width=0.99\textwidth]{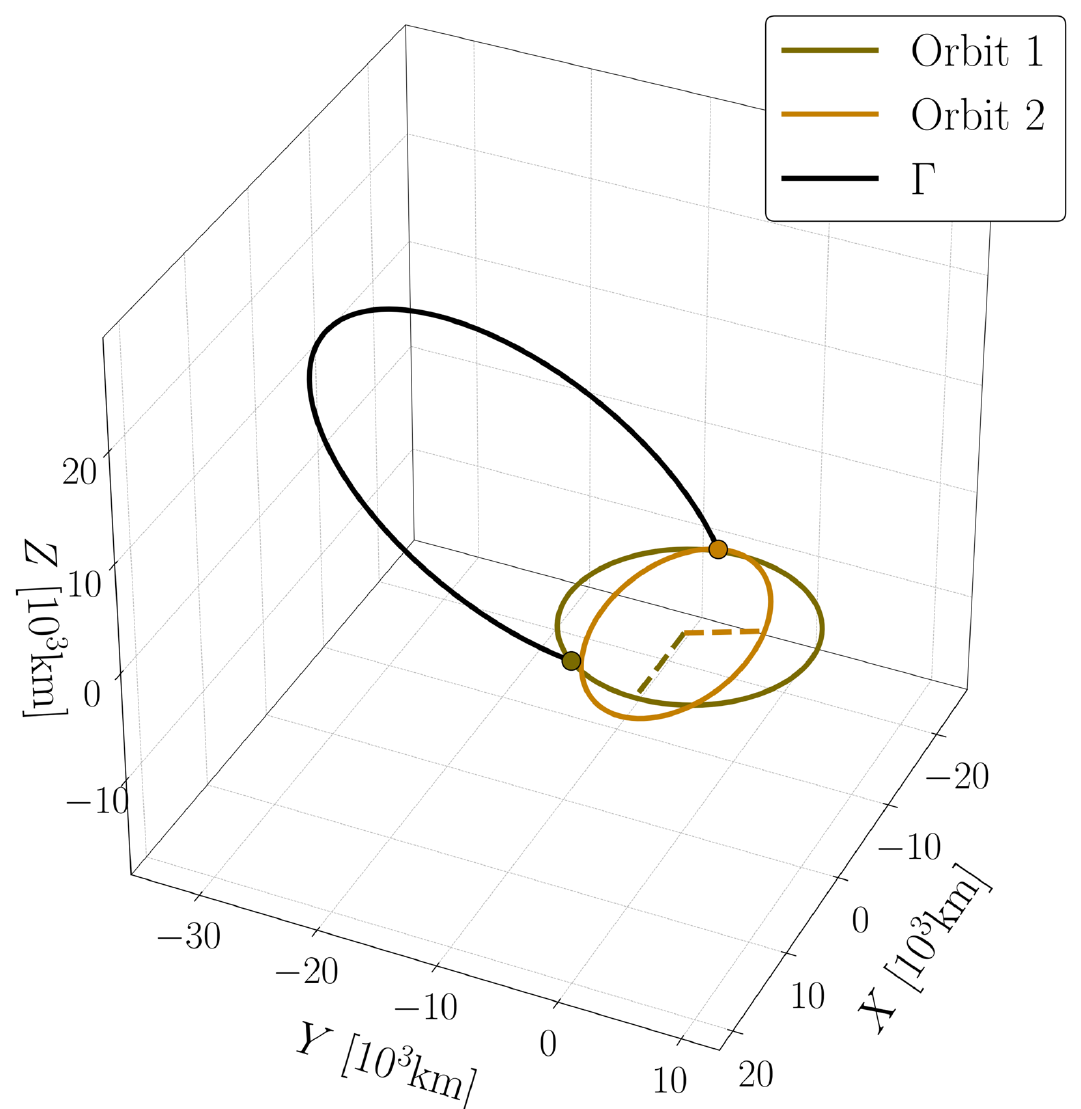}
        \caption{A sample \acrshort{tiot} ($d_1=\mathrm{s}$\taken{, short transfer}) near $\infty^\mathrm{s}_{\sigma,2}$.}
        \label{fig:tinf_short_lambert_arc.png}
    \end{subfigure}
    \begin{subfigure}[b]{0.48\textwidth}
        \centering
        \includegraphics[page=9, width=0.99\textwidth]{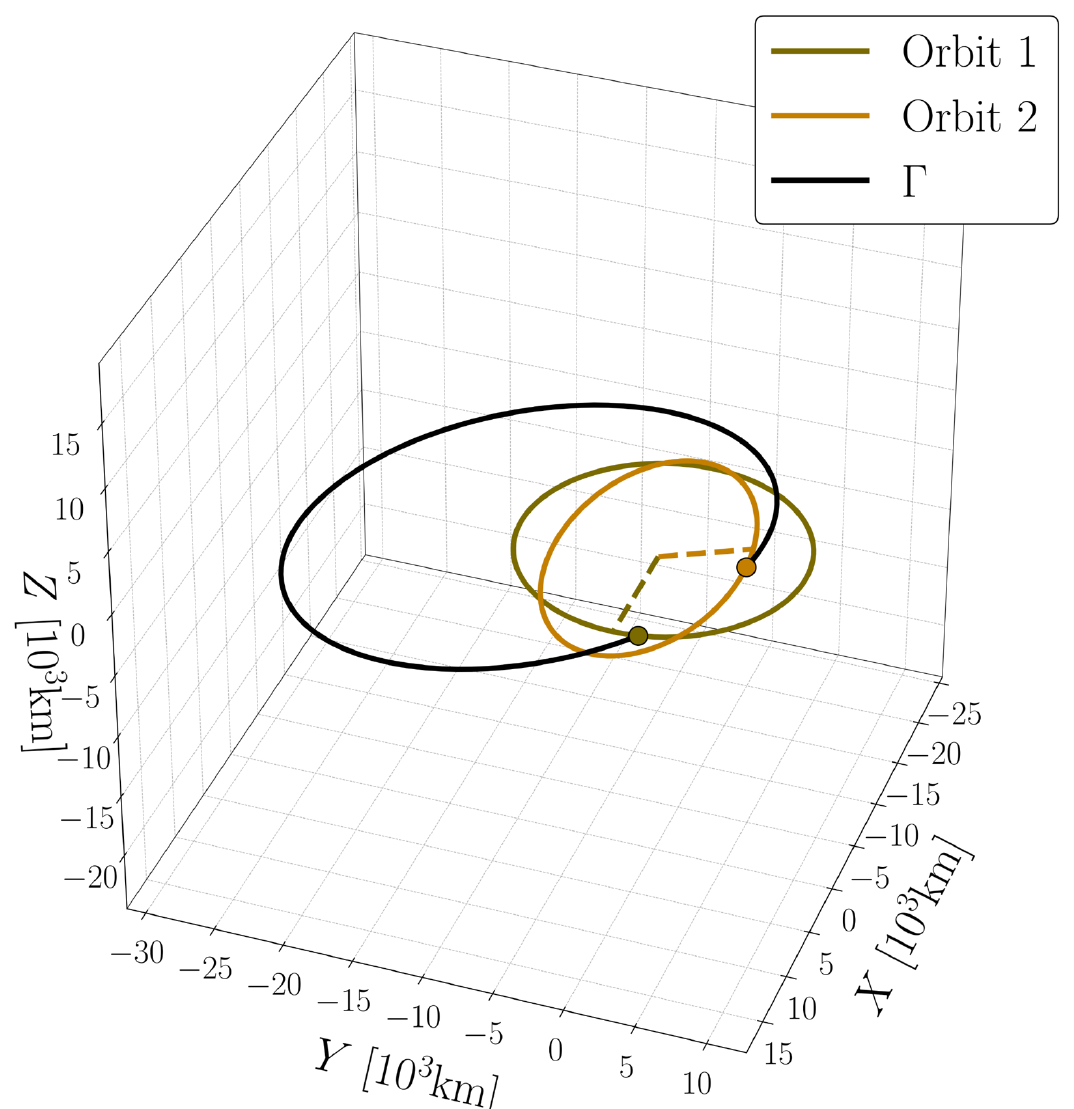}
        \caption{A sample \acrshort{tiot} ($d_1=\mathrm{l}$\taken{, long transfer}) near $\infty^\mathrm{l}_{M,1}$.}
        \label{fig:tinf_long_lambert_arc.png}
    \end{subfigure}
  \caption{Representative \acrshort{tiot} geometries in the near-asymptotic regime ($\tof = 25\cdot 10^3$ seconds). \taken{In both panels, the dashed lines indicate the eccentricity-vector direction of each orbit.}}
  \label{fig:tinf_lambert_arc}
\end{figure}

\subsection{Temporal domain (\texorpdfstring{$\t-\tof$}{})}

The temporal domain formulation serves as the practical counterpart to the angular domain, directly mapping to the coordinates for standard porkchop plots and mission design. Selecting $x_1 = \t$, $x_2 = \tof$ for Eq. \eqref{eq:c} results in the optimality constraint as,
\begin{align}
    \label{eq:time_optimality} \left[ \frac{\partial J}{\partial x_1} , \frac{\partial J}{\partial x_2} \right] = \left[ \frac{\partial J}{\partial \t}, \frac{\partial J}{\partial \tof} \right] = \bm{0},
\end{align}
with each element derived via the chain rule as,
\begin{align}
    \label{eq:dJ_dt} & \frac{\partial J}{\partial \t} = \frac{\partial J}{\partial \ma{1}}\mm{1} + \frac{\partial J}{\partial \ma{2}}\mm{2}, \\
    \label{eq:dJ_dtof} & \frac{\partial J}{\partial \tof} = \frac{\partial J}{\partial \ma{2}}\mm{2} + \frac{\delta J }{\delta \tof},
\end{align}
with the relevant sensitivities supplied in \hyperref[app:gradient]{Appendix~A}. Then, the null space of the Jacobian in Eq. \eqref{eq:jacobian} points towards the direction that is perpendicular to the planes in Fig. \ref{fig:porkchop_2d_3d_3d.png} \taken{(depicted by the black arrow in Fig.~\ref{fig:porkchop_2d_3d_2d.png}; this direction is not the cost gradient but the one that connects the seemingly isolated optima along a continuous family curve)}. Following this direction enables the exploration of optimal families by effectively decoupling the search from specific epoch constraints, thereby navigating the underlying angular space of $\ma{1}$ and $
\ma{2}$. Similar to the angular domain, investigating the asymptotic behavior at $\t \rightarrow \infty$ provides critical insights for family initiation. A fundamental question arises: do the analytical structures identified in the angular domain (i.e., parabolic limits) persist in the temporal domain? \taken{This correspondence is not immediate, since the two domains enforce distinct stationarity conditions: a configuration that is stationary in the angular domain is, in general, not stationary in the temporal domain at finite $\tof$, owing to the explicit time-of-flight sensitivity in Eq.~\eqref{eq:dJ_dtof}. Establishing the asymptotic equivalence is therefore what justifies adopting the parabolic asymptotes as seeds for the temporal-domain continuation.} The following proposition establishes their equivalence, 
\begin{prop}
    At an asymptotic case $\tof \rightarrow \infty$, the optimality constraints in the $\t-\tof$ domain (Eqs. \eqref{eq:dJ_dt}-\eqref{eq:dJ_dtof}) approach the optimality constraints in the $\ma{1}-\ma{2}$ domain (Eq. \eqref{eq:angular_optimality}). 
\end{prop}
\begin{proof}
    Refer to \hyperref[app:proof]{Appendix D}.    
\end{proof}
\noindent Consequently, the optimal parabolic transfers identified in the angular domain serve as robust analytical seeds for the temporal domain as well, allowing for the systematic initiation of optimal families from the asymptotic regime.

\section{\label{sec:framework}\taken{Framework for 
Exploring Families of \texorpdfstring{\acrshort{tiots}}{}}}
This section formalizes a computational framework for locating families of \acrshort{tiots}, building directly on the analytical developments introduced in the preceding sections. The framework consists of three main components: (1) seed initialization, (2) family continuation, and (3) solution analysis. The adoption of this framework enables a systematic exploration of multiple branches of locally optimal transfer solutions between arbitrary departure and arrival orbits. An example transfer scenario, summarized in Table \ref{table:sample_orbits_first}, is used throughout this section to illustrate the proposed framework for a representative general rendezvous problem. Among the two optimality domains introduced in Section \ref{sec:domain}, the temporal domain is the primary focus of the remainder of this work. The framework remains fully applicable to the angular domain; corresponding results are provided in \hyperref[app:angular_domain_family]{Appendix E}.

\taken{All numerical routines in the proposed framework are implemented in a Python environment. Computationally intensive routines that are repeatedly called during the continuation process are accelerated utilizing the just-in-time compiler package Numba\footnote{\taken{\url{https://numba.pydata.org/}}}. In particular, the Lambert solver is developed by adapting the MATLAB routine ``Robust solver for Lambert's orbital-boundary value problem''\footnote{\taken{\url{https://github.com/rodyo/FEX-Lambert}}} into a modified Python implementation compatible with Numba-based compilation. The solver primarily follows Izzo's algorithm~\cite{izzo2015revisiting}; when this solver fails to converge, a robust fallback adapted from the Lancaster--Blanchard--Gooding-type implementation in the original MATLAB routine is invoked~\cite{lancaster1970solution,gooding1990procedure}. Root finding and nonlinear least-squares calculations are performed utilizing routines from the SciPy package\footnote{\taken{\url{https://scipy.org/}}}. Table~\ref{table:numerical_tolerance} summarizes the numerical algorithms and procedures utilized in the framework, together with the corresponding tolerance values.}

\begin{table}[htpb]
\centering
\caption{\label{table:numerical_tolerance}\taken{Numerical algorithms, procedures, and tolerance settings utilized in the framework}}
\taken{\begin{tabular}{p{6.2cm} p{7.0cm}}
\toprule
Algorithm / Procedure & Tolerance / Acceptance criterion \\\midrule
Lambert solver (Izzo~\cite{izzo2015revisiting})
& Internal root iteration tolerance: $10^{-14}$ \\
Lambert fallback solver (Lancaster~\cite{lancaster1970solution})
& Internal root iteration tolerance: $10^{-12}$ \\
Asymptotic parabola stationary-point search 
& Root solver tolerance: $x_\mathrm{tol}=10^{-12}$; stationary points are accepted when $\|\nabla J\|\leq 10^{-6}$, with fallback cases relaxed up to $10^{-4}$ \\
Parabola seed correction 
& Least-squares tolerances: $f_\mathrm{tol}=x_\mathrm{tol}=g_\mathrm{tol}=10^{-12}$ \\
Porkchop seed generation 
& Least-squares tolerances: $f_\mathrm{tol}=x_\mathrm{tol}=g_\mathrm{tol}=10^{-12}$; candidate seeds are accepted when $\|\mathbf{r}\|<10^{-6}$ \\
Family continuation correction 
& Least-squares tolerances: $f_\mathrm{tol}=x_\mathrm{tol}=10^{-10}$; corrected family points are accepted when $\|\mathbf{r}\|\leq 10^{-4}$ \\
\bottomrule
\end{tabular}}

\vspace{2pt}
{\footnotesize\taken{\parbox{13.2cm}{Symbols $f_\mathrm{tol}$, $x_\mathrm{tol}$, and $g_\mathrm{tol}$ denote the termination tolerances on the cost-function reduction, the step in the independent variables, and the gradient norm, respectively, following the corresponding SciPy solver conventions. The acceptance criteria leverage $\|\nabla J\|$, the norm of the cost gradient, and $\|\mathbf{r}\|$, the norm of the residual of the solved system.}}}
\end{table}

\begin{figure}[htbp]
\centering
\includegraphics[page=10, width=0.6\textwidth]{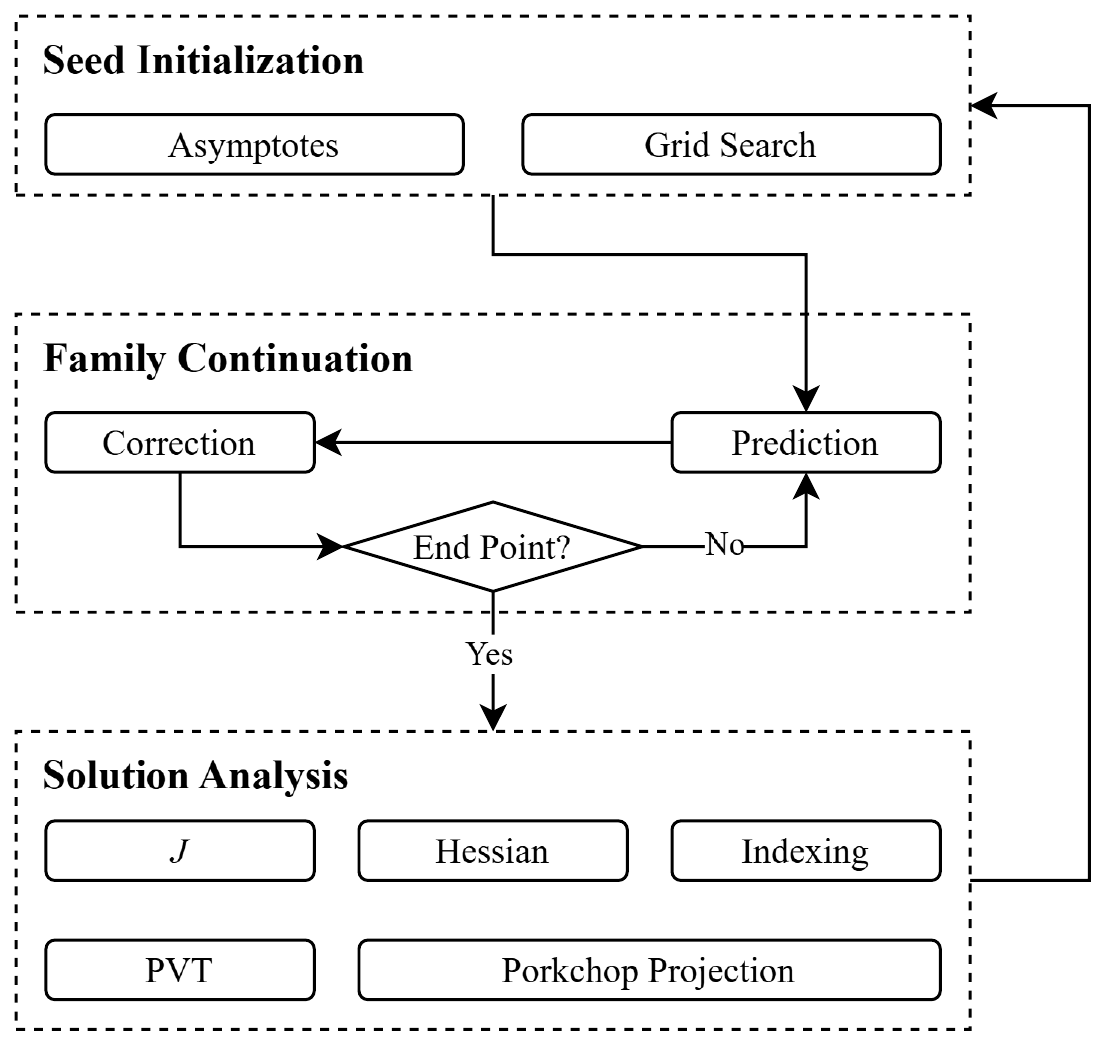}
\caption{Conceptual framework for constructing \acrshort{tiot} families.}
\label{fig:framework_overview}
\end{figure}

\subsection{Seed Initialization}
To initiate the continuation process described in Section \ref{sec:cont_scheme}, a set of starting seeds is required. Two complementary approaches are adopted in the present framework: (1) semi-analytic seeds derived from the asymptotic behavior as the time-of-flight approaches infinity, and (2) numerical seeds obtained via grid search. Both strategies are examined concurrently in this investigation. Unlike many previous studies that focus exclusively on minimum-fuel solutions, the present analysis considers all stationary transfers, regardless of their Hessian classification (Section \ref{sec:hessian}). This choice is deliberate: along a given family, second-order optimality properties may change, and limiting attention to minima alone may obscure global connectivity between solution branches. By tracking all stationary points and their evolving families, the framework enables a more complete characterization of the global structure of \acrshort{tiots}.

\subsubsection{Seeds from \texorpdfstring{$\tof \rightarrow \infty$}{} Asymptotes}
The asymptotic behaviors discussed in the preceding section provide a natural source of initial seeds. This approach is semi-analytic, computationally inexpensive, and can be applied without prior knowledge of the detailed structure of a given family. Its applicability, however, is limited to families that originate from or terminate at the asymptotic limit $\tof \rightarrow \infty$. Families that exhibit closed-loop behavior or remain entirely within finite time-of-flight ranges are not captured by this method.

The asymptotes summarized in Table \ref{table:tinf} correspond to parabolic Lambert limits. To generate seeds suitable for continuation at finite time-of-flight, the asymptotic angular configurations are re-converged in the selected optimality domain utilizing a large but finite transfer time\taken{, chosen to be a few orbital periods (Section~\ref{sec:domain})}. The resulting solutions serve as initial seeds for family continuation in the finite-time regime. For the baseline transfer scenario considered in this study, these seeds are located near $\tof = 25\cdot 10^3$ seconds and are listed in Table \ref{table:seeds_from_parabola}, with labels consistent with their corresponding asymptotic classifications.

\begin{table}[htpb]
\centering
\caption{\label{table:seeds_from_parabola}Re-converged numerical seeds derived from asymptotic stationary points ($t \rightarrow \infty$, Table \ref{table:tinf}) for family continuation. The seeds are obtained at a representative large time-of-flight ($\tof \approx 25 \cdot 10^3$ seconds) within the temporal domain.}
\begin{tabular}{ l c c c c}
\toprule
Root asymptote & $\ma{1}$ [rad] & $\ma{2}$ [rad] & $t$ [$10^3$ sec] & \taken{$J$ [km/s]}\\\midrule
$\infty^\mathrm{s}_{\sigma,1}$ & $1.61$  & $4.64$ & $25.0$ & \taken{25.10}\\
$\infty^\mathrm{s}_{\sigma,2}$ & $5.69$ & $1.39$ & $25.1$ & \taken{27.67}\\
$\infty^\mathrm{s}_{\sigma,3}$ & $2.46$ & $3.32$ & $25.1$ & \taken{13.42}\\
$\infty^\mathrm{s}_{\sigma,4}$ & $5.98$ & $0.24$ & $25.0$ & \taken{16.58}\\
$\infty^\mathrm{l}_{M,1}$ & $0.13$  & $6.12$ & $25.2$ & \taken{31.73}\\
$\infty^\mathrm{l}_{\sigma,1}$ & $3.18$ & $2.52$ & $25.4$ & \taken{27.21}\\
$\infty^\mathrm{l}_{\sigma,2}$ & $2.66$  & $5.87$ & $25.3$ & \taken{8.01}\\
$\infty^\mathrm{l}_{\sigma,3}$ & $5.78$ & $2.62$ & $25.3$ & \taken{7.82}\\
$\infty^\mathrm{l}_{m,1}$ & $0.82$ & $4.78$ & $25.2$ & \taken{6.03}\\
$\infty^\mathrm{l}_{m,2}$ & $4.21$ & $1.26$ & $25.4$ & \taken{6.75}\\\bottomrule
\end{tabular}
\end{table}

\subsubsection{\label{sec:seeds_grid}Seeds from Grid Search}
A grid search-based numerical approach offers a universal mechanism for seed generation, without restrictions on the types of admissible families. Although more computationally intensive than the asymptotic approach, its efficiency improves significantly when approximate regions of interest in the space are informed by prior knowledge of the optimality structure. The grid search is conducted on planar sections aligned with the selected optimality domains, as illustrated in Fig. \ref{fig:grid_search_explained}. In the $\ma{1}-\ma{2}$ domain, the search planes are parallel to the $\ma{1}-\ma{2}$ plane (i.e., $\tof = c$ with a constant $c$, Fig. \ref{fig:grid_search_explained_angle}). In the $\t-\tof$ domain, the search planes are aligned with progressions in $\tof$, expressed as $\ma{1}/n_1 - \mad{2}/n_2 = c$ (Fig. \ref{fig:grid_search_explained_time}). 
Stationary points on these planes are identified using a numerical solver that enforces the first-order optimality conditions in the respective domains. In particular, for the temporal ($\t-\tof$) domain, this procedure is analogous to a porkchop-based search for local minima; however, the objective here is to identify stationary points, forming a superset that includes minima, maxima, and saddle points. For the representative example depicted in Fig. \ref{fig:seeds_porkchop}, the porkchop plots are constructed over the ranges $0 \leq \t \leq 2P_{syn}$ and $ 0.2P_2 \leq \tof \leq 2P_2$, where $P_{syn} = 2\pi/(\mm{2}-\mm{1})$ and $P_2 = 2\pi/\mm{2}$. These bounds are selected to capture the dominant structures of the cost landscape for the illustrative case and do not represent a limitation of the proposed framework. In Fig. \ref{fig:seeds_porkchop}, the detected stationary points are marked according to the color scheme defined in Table \ref{table:colormap} and serve as seeds for the subsequent continuation process.

 \begin{figure}[htbp]
  \centering
  \begin{subfigure}[b]{0.49\textwidth}
    \centering
    \includegraphics[page=11, width=0.99\textwidth]{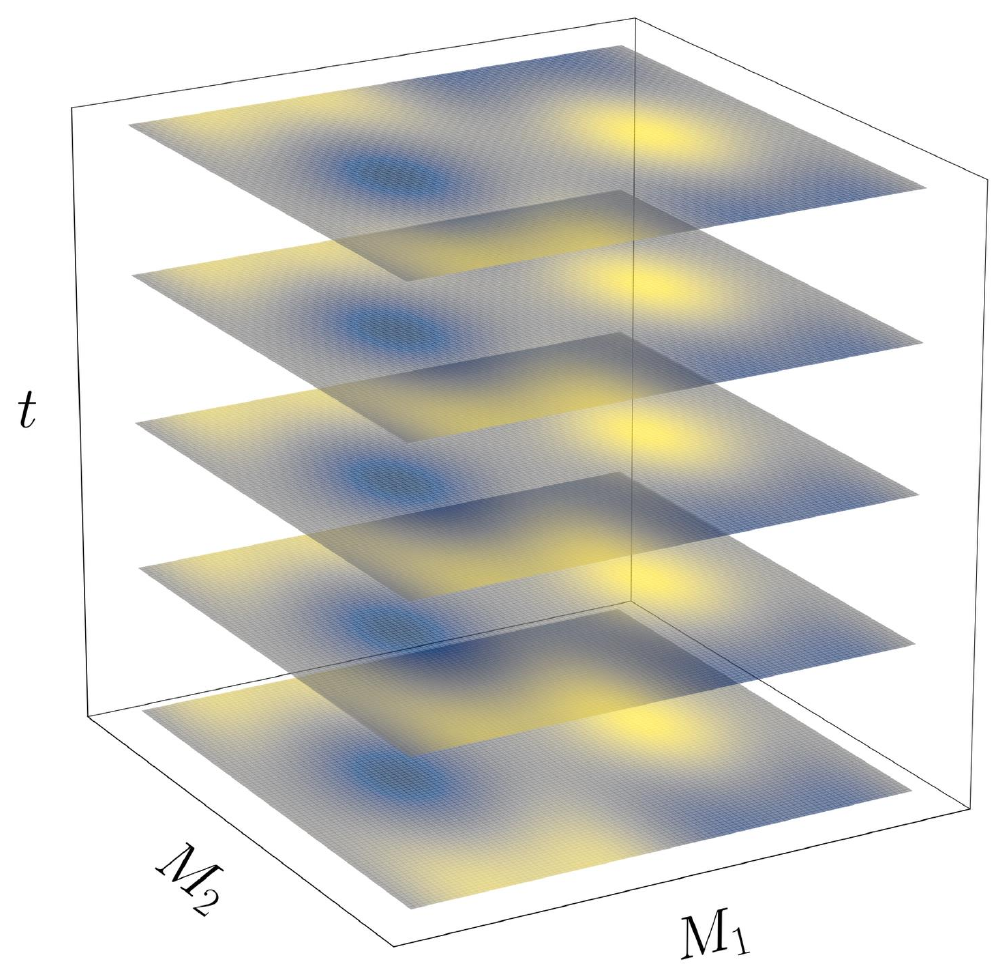}
    \caption{\label{fig:grid_search_explained_angle}Angular ($\ma{1}-\ma{2}$) domain.}
  \end{subfigure}
  \hfill
  \begin{subfigure}[b]{0.49\textwidth}
    \centering
    \includegraphics[page=12, width=0.99\textwidth]{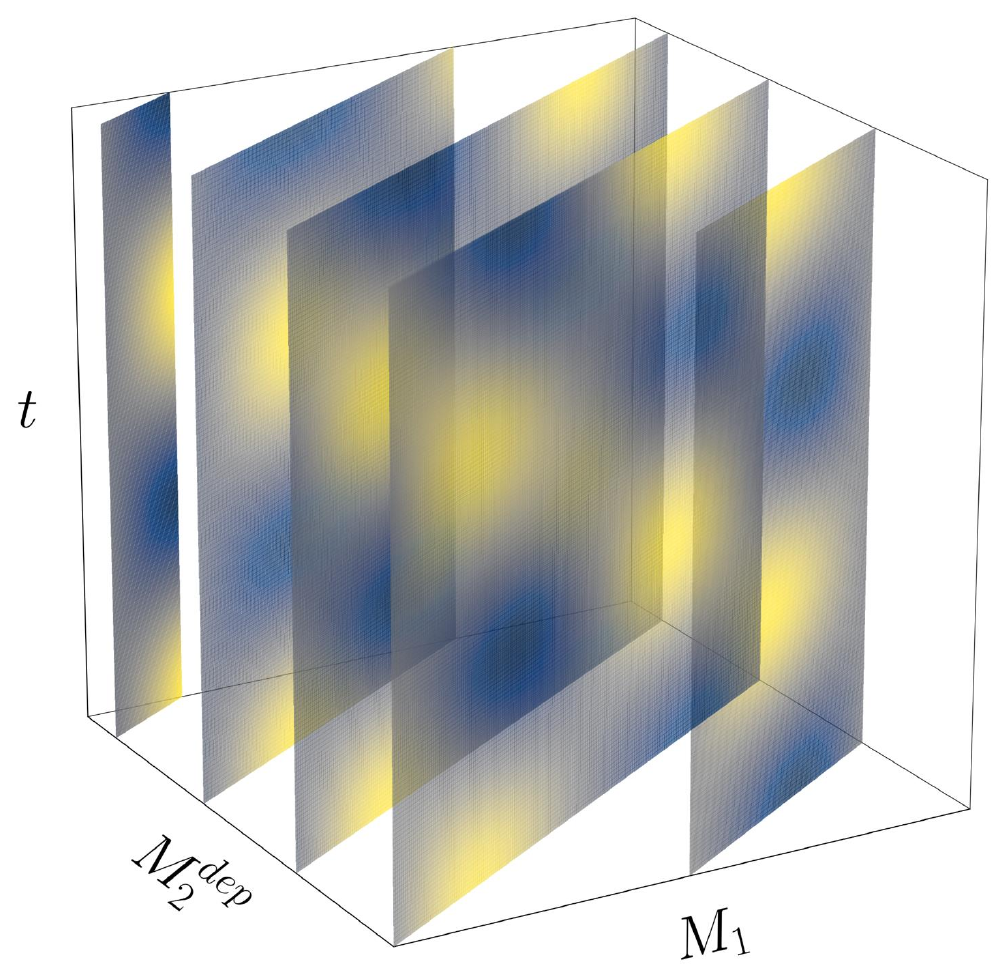}
    \caption{\label{fig:grid_search_explained_time}Temporal ($\t-\tof$) domain.}
  \end{subfigure}
  \caption{Visualization of grid search strategies within the two optimality domains.}
  \label{fig:grid_search_explained}
\end{figure}

\begin{figure}[htbp]
  \centering
  \begin{subfigure}[b]{0.49\textwidth}
    \centering
    \includegraphics[page=13, width=0.99\textwidth]{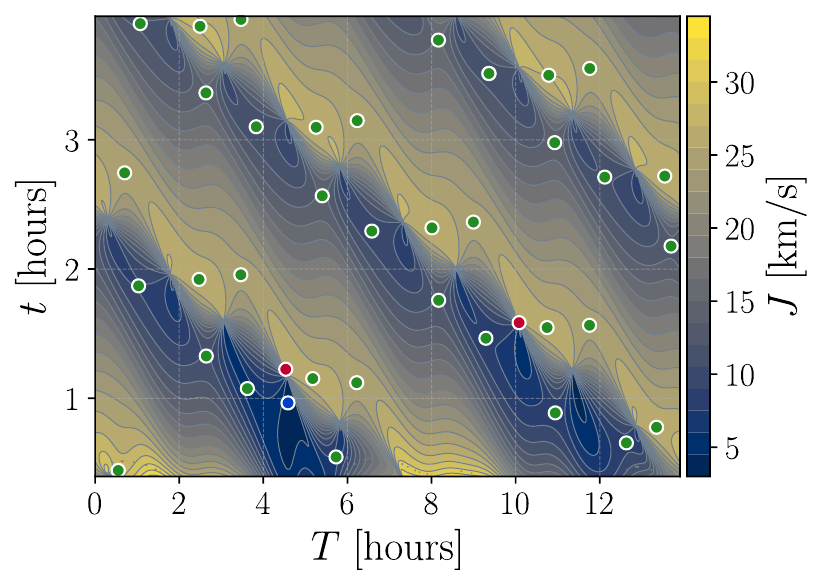}    \caption{\label{fig:seeds_porkchop_short}$d_1 = \mathrm{s}$ \taken{(short transfer)}: 1 minimum, 2 maxima, and 37 saddle points.}
  \end{subfigure}
  \hfill
  \begin{subfigure}[b]{0.49\textwidth}
    \centering
    \includegraphics[page=14, width=0.99\textwidth]{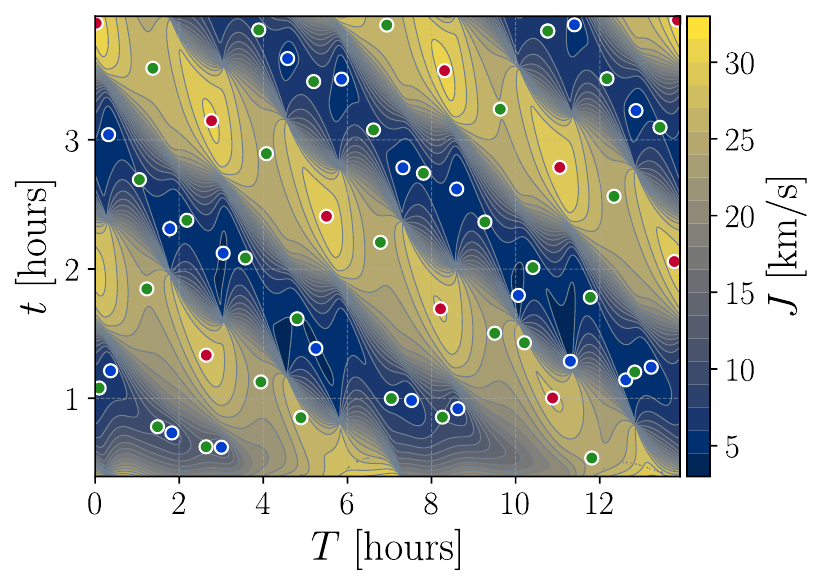}
    \caption{\label{fig:seeds_porkchop_long}$d_1 = \mathrm{l}$ \taken{(long transfer)}: 19 minima, 10 maxima, and 34 saddle points.}
  \end{subfigure}
  \caption{Identification of stationary points via grid search in the temporal domain.}
  \label{fig:seeds_porkchop}
\end{figure}

\subsubsection{Other Seed Sources}

In addition to asymptotic limits and grid-based searches, seeds may also be supplied using prior knowledge or diagnostic information obtained during the continuation process. For example, after completing continuation from all initially identified seeds, the absence of an expected connection between families may indicate missing branches. In such cases, additional seeds can be introduced to recover the complete family structure. \taken{An example of this situation occurs in the angular domain analysis, detailed in Appendix~E.}

\taken{It is emphasized, however, that this framework does not guarantee analytical completeness of the identified family set in either optimality domain. The difficulty is more pronounced in the temporal domain, where intricate connections that do not involve the asymptotic ($\tof\to0,\infty$) limits are more common (for instance, the cyclic families that connect to no asymptote), so the seeding cannot rely on the asymptotic configurations alone. Nevertheless, for non-resonant orbital configurations the time-lines become dense in the angular domain (Section~\ref{sec:porkchop_projection}), so a sufficiently long and fine epoch sampling intersects any family of finite angular extent; once a single point on a family is seeded, continuation recovers the entire branch, including narrow segments of interest. The residual risk is thus confined to families whose entire extent is vanishingly small, which are difficult to detect yet also of limited practical value, as their scarce recurrence is quantified by the perpendicular family extent in Section~\ref{sec:long_horizon}.}

\subsection{Family Continuation and End Points}
Starting from the identified seeds, families of \acrshort{tiots} are continued utilizing the numerical continuation scheme described in Section \ref{sec:cont_scheme}. In this framework, one-parameter families are traced by enforcing two optimality conditions while relaxing the remaining degree of freedom. From a given seed, continuation generally proceeds in two directions along the family. For convenience, any component of the state vector $\bm{X}$ may be used to define the initial continuation direction; in the present study, the $\pm \tof$ directions are adopted. A suitable step size along the family is required to avoid unintended jumps between neighboring branches. Although the step size $ds$, from Eq. \eqref{eq:guess}, is a user-defined parameter that may vary across applications, a value of $ds = 0.01$ is adopted for the case studies presented here\footnote{Since $ds$ measures arc length in a space of mixed units (radians for $\ma{1,2}$ and a nondimensional time for $\tof/1000$), it does not carry a single physical unit; it serves purely as a scaling parameter that controls the step size within the free-variable ($\bm{X}$) space.}, with the variables scaled as $(\ma{1}, \ma{2}, \tof/1000)$. From each seed, family members are generated sequentially in both directions until one of the following termination conditions (``end points'') is encountered:
\begin{enumerate}
    \item Upper time-of-flight bound ($\tof > \tof_{max}$): Continuation is terminated when the time-of-flight exceeds a prescribed maximum threshold. In the present analysis, this value is set to $\tof_{max} = 40\cdot10^3$ seconds. Beyond this range, the transfer trajectories gradually approach the parabolic Lambert solutions associated with the asymptotic limit $\tof \rightarrow \infty$, and further continuation provides limited additional insight.
    \item Lower time-of-flight bound ($\tof < \tof_{min}$): Continuation is also terminated when the time-of-flight drops below a minimum threshold, set to $\tof_{min} = 10$ seconds in this study. For shorter transfer times, the Lambert solutions typically become hyperbolic and exhibit increased sensitivity that destabilizes the continuation process and compromises numerical robustness. 
    \item Singularity at $\theta = 180^\circ$: A well-known singularity arises in the Lambert problem when the transfer angle satisfies $\theta = 180^\circ$. In this configuration, the orbital plane of the transfer arc becomes ambiguous, and the STM submatrix $\phi_{\bm{rv}}$, required to compute the sensitivities of the Lambert velocities with respect to the endpoint positions, becomes singular. For the reference case from orbits in Table \ref{table:sample_orbits_first}, this singular configuration is illustrated in Fig. \ref{fig:transfer_pi.png}, where the red line denotes the nodal line that intersects both orbital planes. At $\theta = 180^\circ = \pi$ radians, two distinct geometric configurations are possible (Fig. \ref{fig:transfer_pi.png}): (1) $\redfilledcircle\rightarrow \redfilledsquare$, and (2) $\redfilledtriangle \rightarrow \redfilledstar$. These configurations are denoted as $\pi_1$ and $\pi_2$, respectively. In the temporal domain, two distinct empirical behaviors are observed in the vicinity of this singularity for the cases examined in this study. In some instances, two family branches appear to connect smoothly across $\theta = 180^\circ$. This singular configuration is associated with the transition between long- and short-transfer Lambert arcs. Several strategies may be adopted to handle this behavior within a continuation framework. Purely numerical approaches may increase the continuation step size to effectively ``jump'' across the singularity while switching the transfer-type indicator $d_1$ (long or short). More analytical or geometric treatments are also possible. For example, Sun \cite{sun1969analytic} demonstrated that optimal transfer-plane orientations across $\theta = 180^\circ$ satisfy a Snell’s-law-like condition, a result later reproduced by Vinh et al. \cite{vinh1988optimal}. This suggests that the optimality constraints adopted here may transform into alternative geometric conditions at the singular configuration. In other instances, two distinct branches emerge that approach the singular configuration asymptotically as $\tof \rightarrow 0$, without exhibiting a smooth connection. A detailed analytical treatment of these behaviors near singularity, including a rigorous classification of the observed behaviors and their underlying causes, is beyond the scope of the present study. Nevertheless, empirical evidence indicates that smooth family continuation across $\theta = 180^\circ$ is possible in certain cases. Accordingly, the present work adopts a simple and robust strategy: continuation is halted as the solution approaches the singularity, and family connectivity across $\theta = 180^\circ$ is subsequently assessed \emph{a posteriori} using porkchop-based seeds on the opposite side of the singularity. If a branch is found to be missing, an additional seed is supplied from the porkchop plot to recover the full family.
    \begin{figure}
        \centering
        \includegraphics[page=72, width=0.5\textwidth]{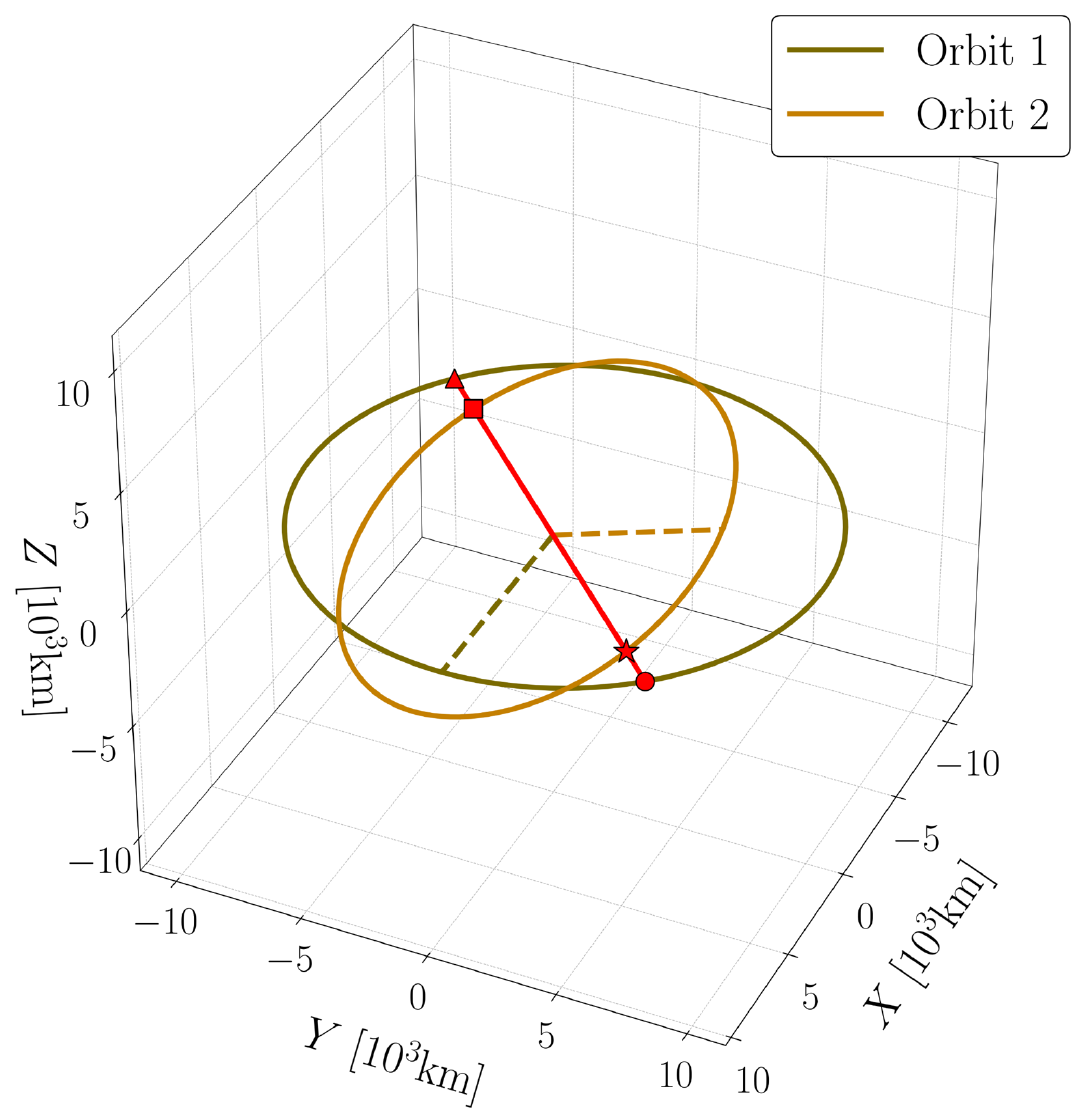}
        \caption{Geometric configuration of the common nodal line (red): intersection of the departure and arrival orbital planes.}
        \label{fig:transfer_pi.png}
    \end{figure}    
    \item Completed cycle: For families that exhibit cyclic behavior, continuation is terminated when the current family member returns to the original seed. In the present investigation, such cyclic families are observed only in the temporal domain; the angular domain formulation does not yield closed-loop behavior for the cases examined. \taken{In the angular domain, every identified family instead terminates at an asymptotic ($\tof\to0$ or $\infty$) or singular ($\theta=180^\circ$) limit (Appendix~\ref{app:angular_domain_family}); a broader characterization across orbital geometries remains a topic for future work.}
\end{enumerate}
For the examples presented in this work, these four termination conditions encompass all observed family end points. Depending on the orbital geometry, additional behaviors may arise that violate the implicit assumptions of the current continuation scheme. For example, the Jacobian $\partial \c / \partial \fv$ in Eq. \eqref{eq:dc_dfv} during the differential correction process may, in principle, admit a two-dimensional null space, indicating the intersection of multiple families at a single point. Although no such bifurcating configurations are observed in the present study, their potential existence motivates further investigation.

\taken{After all continuations from the generated seeds are completed, duplicate continuations are identified by treating each continuation as a curve in $(\ma{1},\ma{2},\tof)$ space and comparing these curves directly, with wrapped distances applied to the angular coordinates. For each pair of continuations, nearest-neighbor matching is performed in both directions: a point on one curve is regarded as covered when the nearest point on the other lies within the prescribed tolerances, taken here as $3.0\times10^{-2}$~rad in each angular coordinate and $3.0\times10^{-2}$ in the normalized time coordinate. Two continuations are regarded as duplicates of the same family when at least 95\% of the points on each curve are covered in this sense; equivalently, the 95th percentile of the nearest-neighbor discrepancies falls within these tolerances, tolerating mismatch over the remaining few percent typically near the curve endpoints. This bidirectional coverage requirement separates a genuine merger, with global overlap of the two curves, from a mere close approach, with only local intersection or proximity. When multiple continuations are grouped as duplicates, a single representative curve is retained, preferring one whose continuation terminated cleanly at both ends and, failing that, the one with more sampled points. These thresholds proved adequate for the configurations examined; more intricate family evolution, such as branches that remain close over an extended arc, may call for finer tolerances or a supplementary manual check, and other curve-matching or clustering criteria could equally be used, provided they distinguish global curve overlap from local close approaches.}

For the sample transfer scenario summarized in Table \ref{table:sample_orbits_first}, families are generated from both asymptotic and grid search seeds. The resulting families are visualized in the three-dimensional $\ma{1}-\ma{2}-\tof$ space and are plotted in \taken{Fig. \ref{fig:seeds_family_combine_3d}} for the asymptotic and grid search seeds, respectively. The families supplied from the grid search seeds fully encompass those generated from the asymptotic seeds. In contrast, cyclic families do not connect to the $\tof \rightarrow \infty$ asymptotes and therefore cannot be discovered leveraging asymptotic initialization alone. \taken{The two seeding strategies therefore play complementary roles: the grid search provides exhaustive numerical coverage, in the present cases yielding a superset of the asymptotically seeded families, whereas the asymptotic approach is parameter-free, requiring no choice of search domain or resolution, and analytically identifies the $\tof \rightarrow \infty$ origins of the families. These analytic origins, and their qualitative changes under variations in orbital geometry, underlie the reconnection and reorganization of families noted in Section~\ref{sec:conclusions_future_work}, structural information that is less apparent from purely numerical grid searches.} In \taken{Fig. \ref{fig:seeds_family_combine_3d}}, two black dotted lines correspond to $\pi_{1,2}$, singular configurations defined from Fig. \ref{fig:transfer_pi.png}. 



\begin{figure}[htbp]
  \centering
  \begin{subfigure}[b]{0.49\textwidth}
    \centering
    \includegraphics[page=15, width=0.99\textwidth]{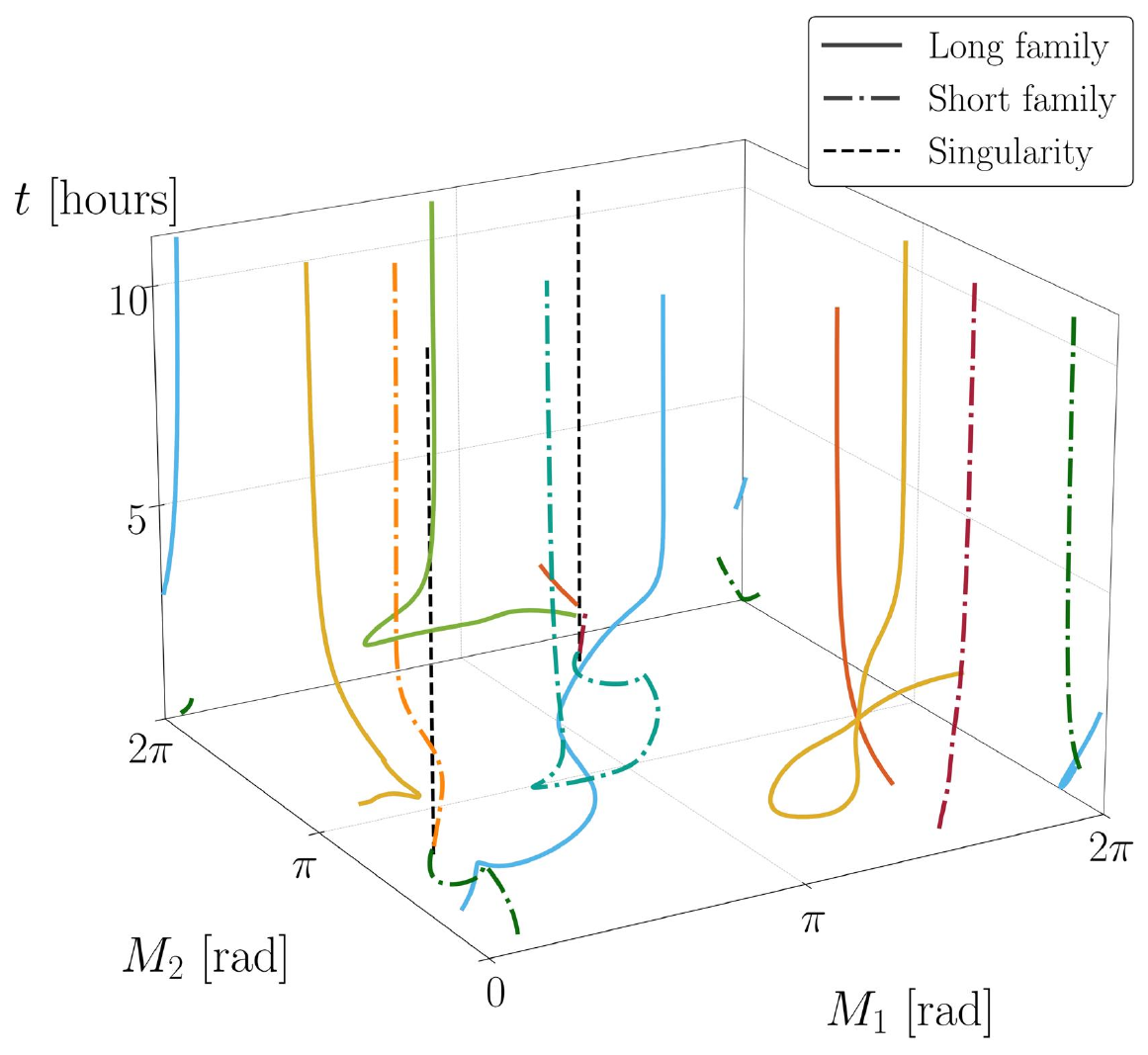}
    \caption{\label{fig:parabola_seeds_family}.}
  \end{subfigure}
  \hfill
  \begin{subfigure}[b]{0.49\textwidth}
    \centering
    \includegraphics[page=16, width=0.99\textwidth]{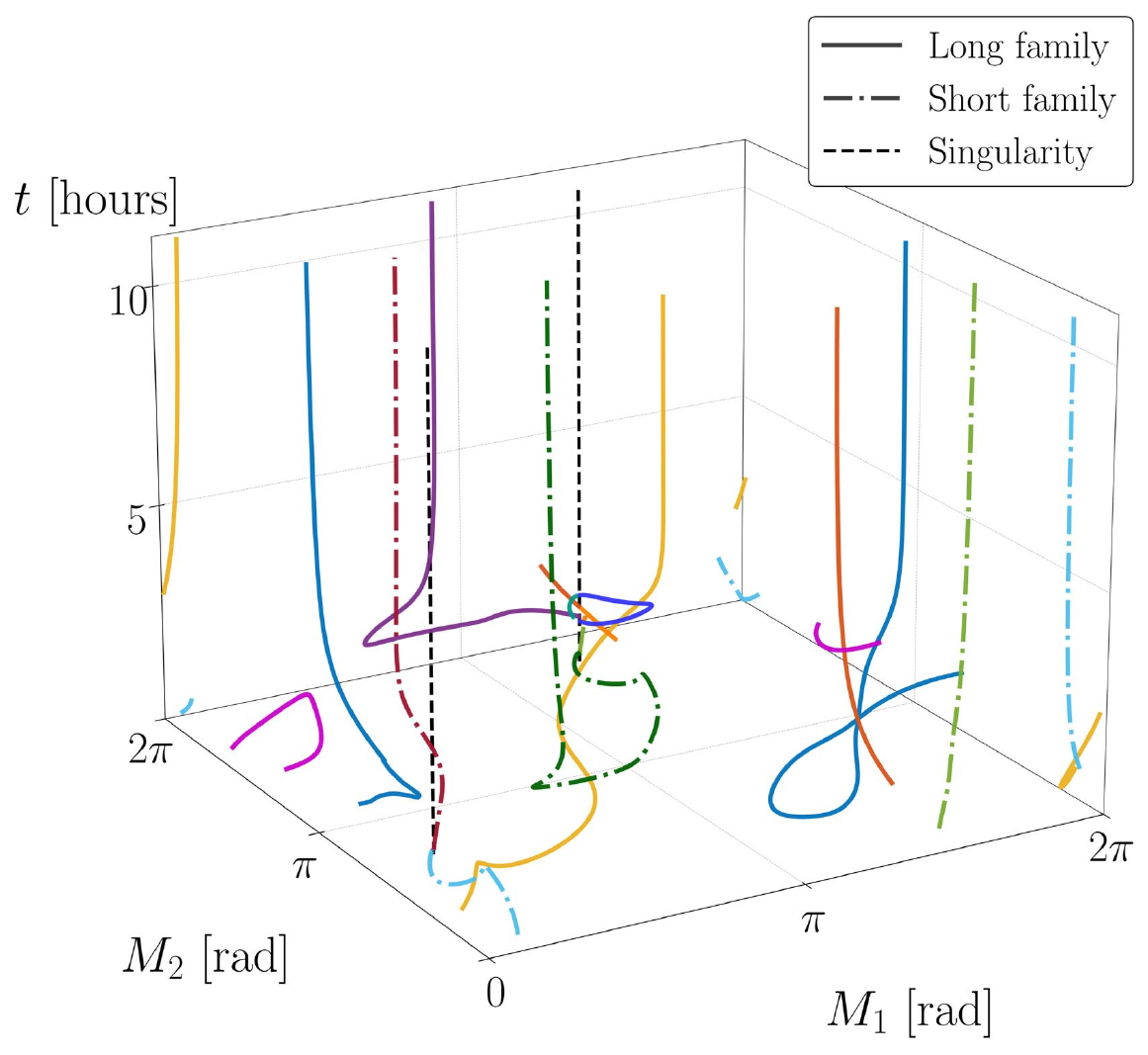}
    \caption{\label{fig:porkchop_seeds_family}.}
  \end{subfigure}
  \caption{\taken{Global structure of \acrshort{tiot} families initiated from asymptotes (Figs. \ref{fig:tinf_short} and \ref{fig:tinf_long}) and grid search seeds (Fig. \ref{fig:seeds_porkchop}).}}
  \label{fig:seeds_family_combine_3d}
\end{figure}

\subsection{Solution Analysis}
The families recorded during seed initialization and continuation are further analyzed to extract qualitative and quantitative insights into the structure of \acrshort{tiot} families. From this section onward, the analysis focuses on families generated from grid search seeds, as they supply the more comprehensive coverage of the solution space.

\taken{The analysis consists of several modules, as depicted in Fig.~\ref{fig:framework_overview}. These modules should not be interpreted as a single prescribed procedure or algorithm in which every step must always be executed. Rather, they are independent analysis components that can be selected and applied depending on the objective of the study. For example, if the problem is limited to a simple two-impulse mission design and only the required $\Delta V$ is of interest, PVT information is unnecessary. Conversely, such diagnostic information becomes useful when the objective is to classify stationary solutions or examine family connectivity. Thus, the ``Solution Analysis'' block in Fig.~\ref{fig:framework_overview} represents a modular set of optional analyses rather than a mandatory sequential workflow.}

\subsubsection{Total Cost \texorpdfstring{$J$}{}}

The total transfer cost $J$ is a primary quantity of interest. \taken{Figure \ref{fig:orbit1_porkchop_family_with_dv}} displays the families generated from grid search seeds, with the color indicating the magnitude of $J$. Lower-cost transfers naturally cluster along specific segments of the families, revealing geometries that are potentially attractive from a mission design perspective. To enhance visual discrimination among low-cost solutions, the color scale is \taken{plotted in log scale. This adjusted scale highlights the most practically relevant subset of the solution space and clarifies the relative desirability of neighboring family branches.}



\begin{figure}[htbp]
\centering
\includegraphics[page=17, width=0.6\textwidth]{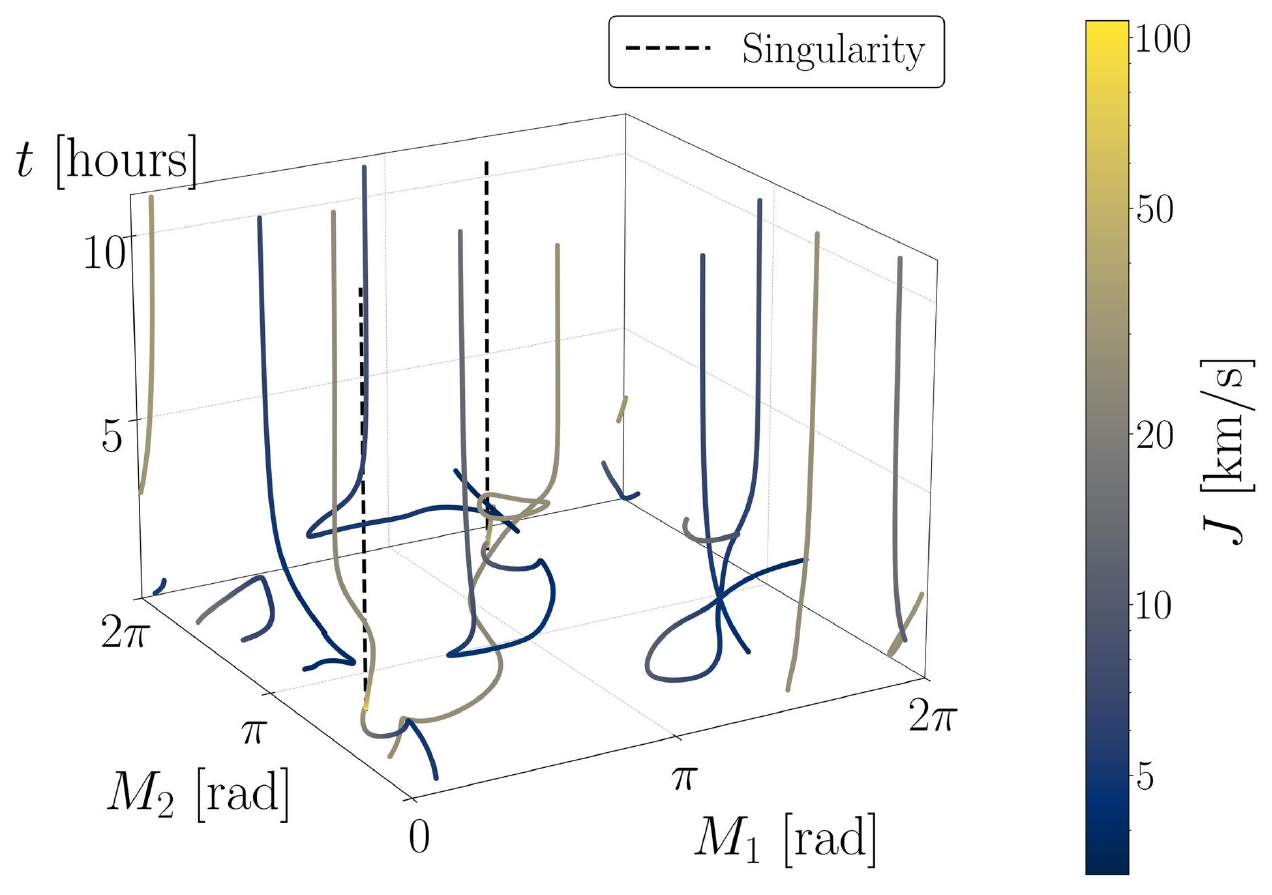}
\caption{\taken{Mapping of transfer cost $J$ along the identified solution families.}}
\label{fig:orbit1_porkchop_family_with_dv}
\end{figure}

\subsubsection{Hessian}
Although the continuation framework enforces first-order optimality conditions, the second-order optimality characteristics may vary along a given family. Accordingly, Hessian information is evaluated for all family members to distinguish between local minima, maxima, and saddle points. \taken{Figure \ref{fig:porkchop_seeds_family_hess}} illustrates the families colored according to their Hessian classification (Table \ref{table:colormap}). A notable observation is that the optimality type may change continuously along a single family, indicating transitions between minima, saddles, and maxima without discontinuities in the underlying transfer geometry. This behavior underscores the importance of treating families as continuous objects rather than collections of isolated optimal points.


\begin{figure}[htbp]
  \centering
  \begin{subfigure}[b]{0.49\textwidth}
    \centering
    \includegraphics[page=18, width=0.99\textwidth]{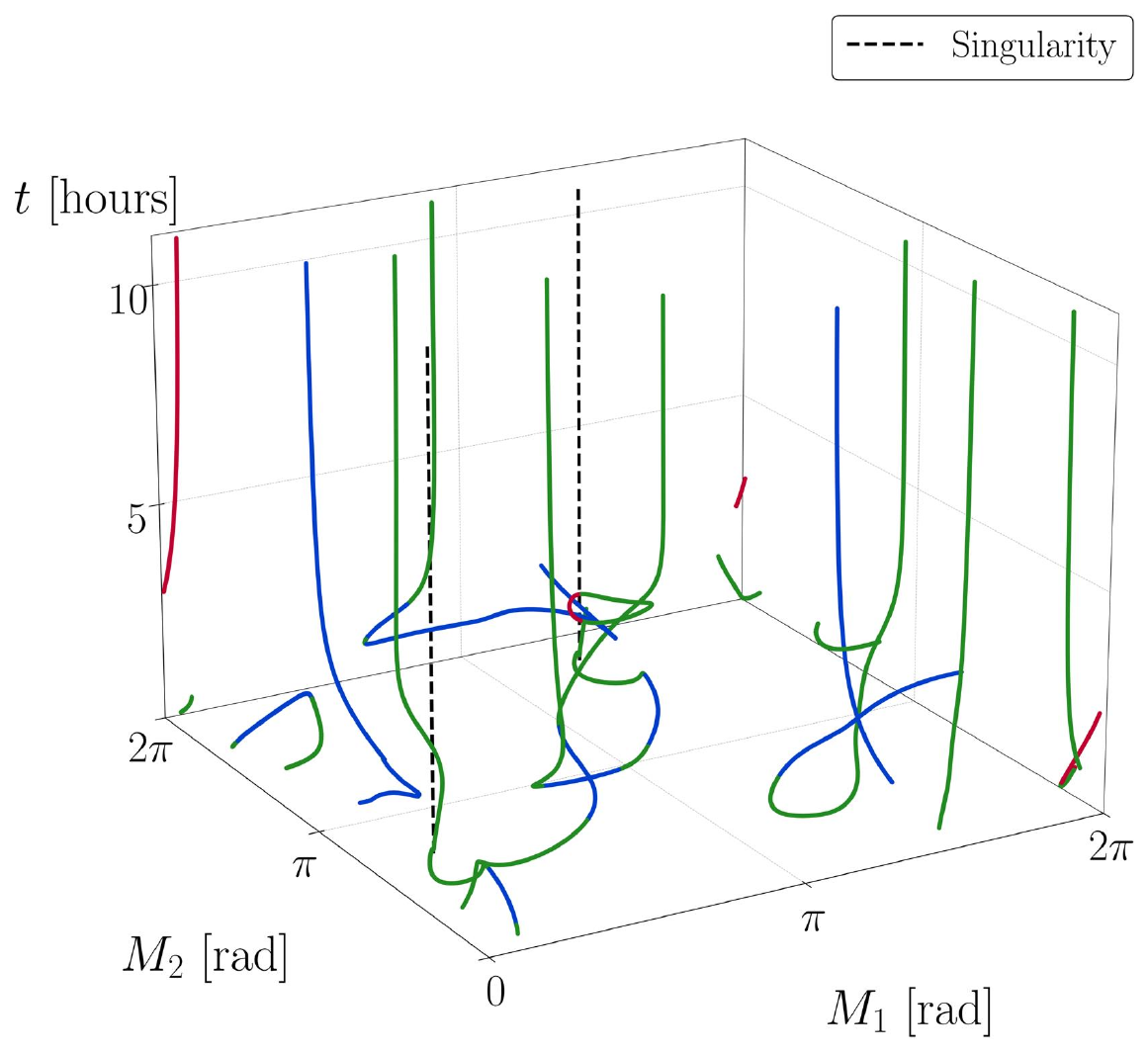}
    \caption{\label{fig:porkchop_seeds_family_hess}\taken{Evolutionary transition of optimality types along the \acrshort{tiot} families. The colors represent the Hessian classification defined in Table \ref{table:colormap} (blue: minimum, red: maximum, green: saddle point).}}
  \end{subfigure}
  \hfill
  \begin{subfigure}[b]{0.49\textwidth}
    \centering
    \includegraphics[page=19, width=0.99\textwidth]{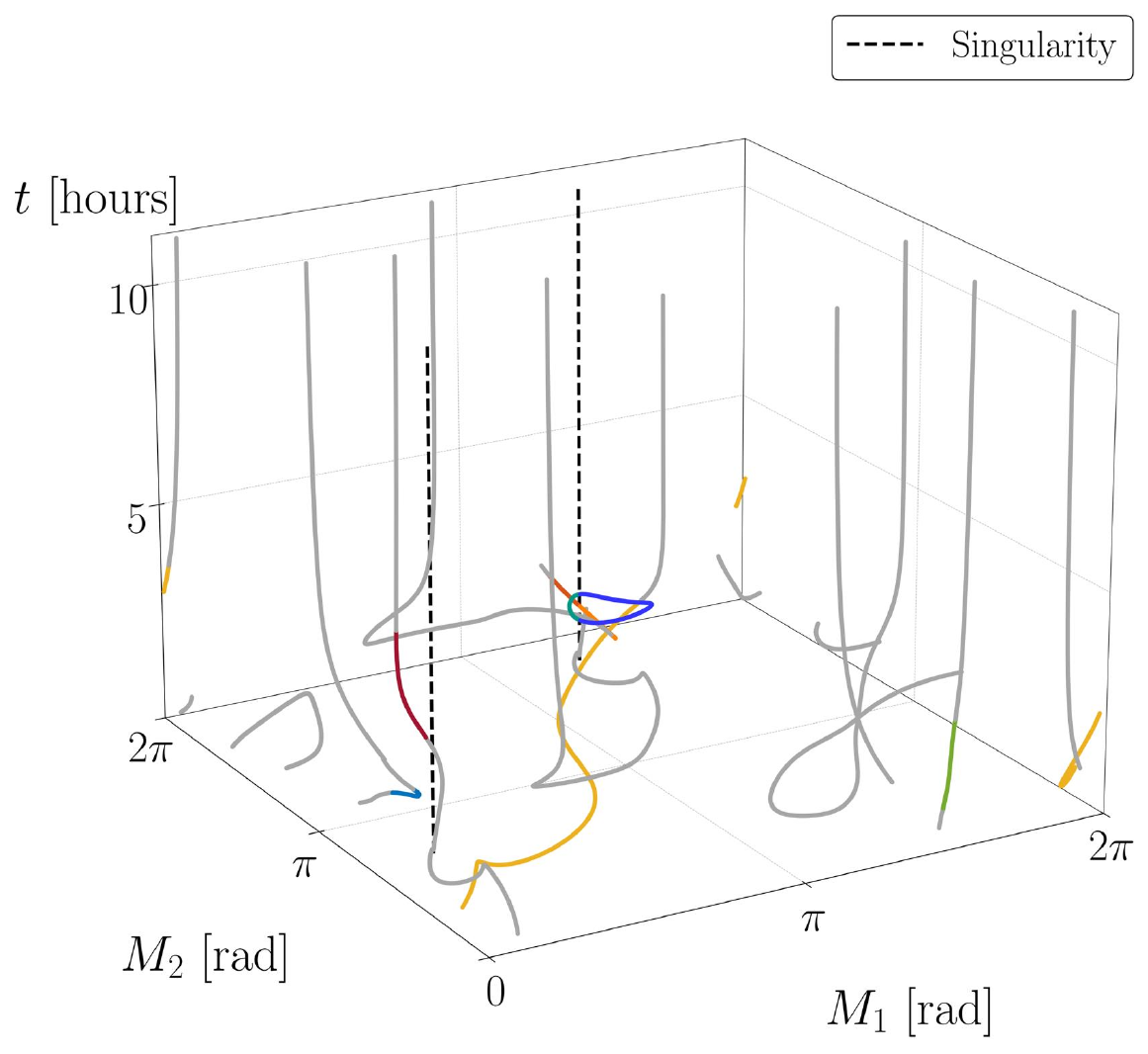}
    \caption{\label{fig:porkchop_seeds_family_pvt}\taken{Validation of impulsive optimality utilizing the \acrshort{pvt} criteria along the solution families. The color scheme follows Table \ref{table:family_index}, while segments that violate \acrshort{pvt} necessary conditions are marked in gray}.}
  \end{subfigure}
  \caption{\taken{Family curves with Hessian and \acrshort{pvt} information.}}
  \label{fig:porkchop_seeds_family_hess_pvt}
\end{figure}

\subsubsection{Assigning Family Index}

Since different seeds may converge to different segments of the same continuous family, or, conversely, to distinct families that appear or disappear across the domain, a consistent family indexing scheme is required. Assigning indices to entire families enables a systematic characterization of the global solution structure, revealing the connections, evolutions, and terminations of families. This indexing facilitates consistent reference across figures and analyses and supports comparative studies of cost trends, optimality transitions, and geometric behavior, particularly in the presence of singular configurations where families may merge, split, or terminate. After removing duplicate segments obtained from different seeds, each unique family is assigned a distinct index. Figure \ref{fig:family_connection1} illustrates the 12 indexed families identified for the representative case, while Table \ref{table:family_index} summarizes their end points and observed inter-family connections. The asymptotic end points at $\tof \rightarrow \infty$ are categorized following the labels from Table \ref{table:seeds_from_parabola} and Eq. \eqref{eq:asymptote_notation}. Singular end points $\pi_1$ and $\pi_{2}$ correspond to the configurations $\redfilledcircle\rightarrow \redfilledsquare$ and $\redfilledtriangle \rightarrow \redfilledstar$ in Fig. \ref{fig:transfer_pi.png}, respectively. At the singular end points, families may either connect smoothly across the singularity or terminate asymptotically, depending on the local family geometry. A zoomed-in view in Fig. \ref{fig:family_connection1} highlights several examples where apparently distinct families merge into a unified structure when examined in the full parameter space (e.g., Families 2, 4, and 9, as well as Families 10 and 12). In contrast, other branches originate from the same singularity but do not connect in a smooth fashion (e.g., Families 6 and 7, as well as Families 5 and 8). \taken{In these instances, the two families approach the singular configuration through different limiting geometries, so that the corresponding transfers remain geometrically distinct even where the families pass close to one another in the parameter space.} \taken{Because continuation terminates before reaching these singular configurations, such connections are assessed a posteriori rather than traced directly. A connection requires a short-transfer branch on one side and a long-transfer branch on the other, and is inferred from the similarity of the limiting transfer geometries and the smoothness of the solution curve across the singularity; in some cases it can be confirmed by bridging the singular point with an enlarged continuation step.} While a rigorous analytical characterization of these singular connections is beyond the scope of the present study \taken{and is left for future work}, the proposed indexing scheme provides a practical and unambiguous means of documenting the observed family connectivity and organizing the global solution structure.

\begin{table}[htpb]
\centering
\caption{\label{table:family_index}Summary of indexed \acrshort{tiot} families in the temporal domain for the baseline scenario (Table \ref{table:sample_orbits_first})}
\begin{tabular}{l c c c c }
\toprule
Family No. & End point 1 & End point 2 & Connections (if any) & Color \\\midrule
1  & $\infty^\mathrm{l}_{m,1}$      & $\infty^\mathrm{l}_{\sigma,3}$     & -                           & \linefamsolid{colFamOne} \\
2  & $\infty^\mathrm{l}_{m,2}$ & $\pi_2$               & Family 9 & \linefamsolid{colFamTwo} \\
3  & $\infty^\mathrm{l}_{M,1}$& $\infty^\mathrm{l}_{\sigma,1}$          & -                                         & \linefamsolid{colFamThree} \\
4  & $\infty^\mathrm{l}_{\sigma,2}$& $\pi_2$               & Family 9                            & \linefamsolid{colFamFour} \\
5  & $\infty^\mathrm{s}_{\sigma,2}$& $\pi_2$   & -         & \linefamsolid{colFamFive} \\
6  & $\infty^\mathrm{s}_{\sigma,4}$& $\pi_1$   & -         & \linefamsolid{colFamSix} \\
7  & $\infty^\mathrm{s}_{\sigma,1}$& $\pi_1$   & -         & \linefamsolid{colFamSeven} \\
8  & $\infty^\mathrm{s}_{\sigma,3}$& $\pi_2$   & -         & \linefamsolid{colFamEight} \\
9  & $\pi_2$     & $\pi_2$               & Family 2, 4                                         & \linefamsolid{colFamNine} \\
10 & $\pi_2$          & $\pi_2$         & Family 12                                        & \linefamsolid{colFamTen} \\
11 & -          & -          & -                                         & \linefamsolid{colFamEleven} \\
12 & $\pi_2$          & $\pi_2$       & Family 10                                         & \linefamsolid{colFamTwelve} \\\bottomrule
\end{tabular}
\end{table}

\begin{figure}[h!]
    \centering
    \includegraphics[page=20, width=0.99\textwidth]{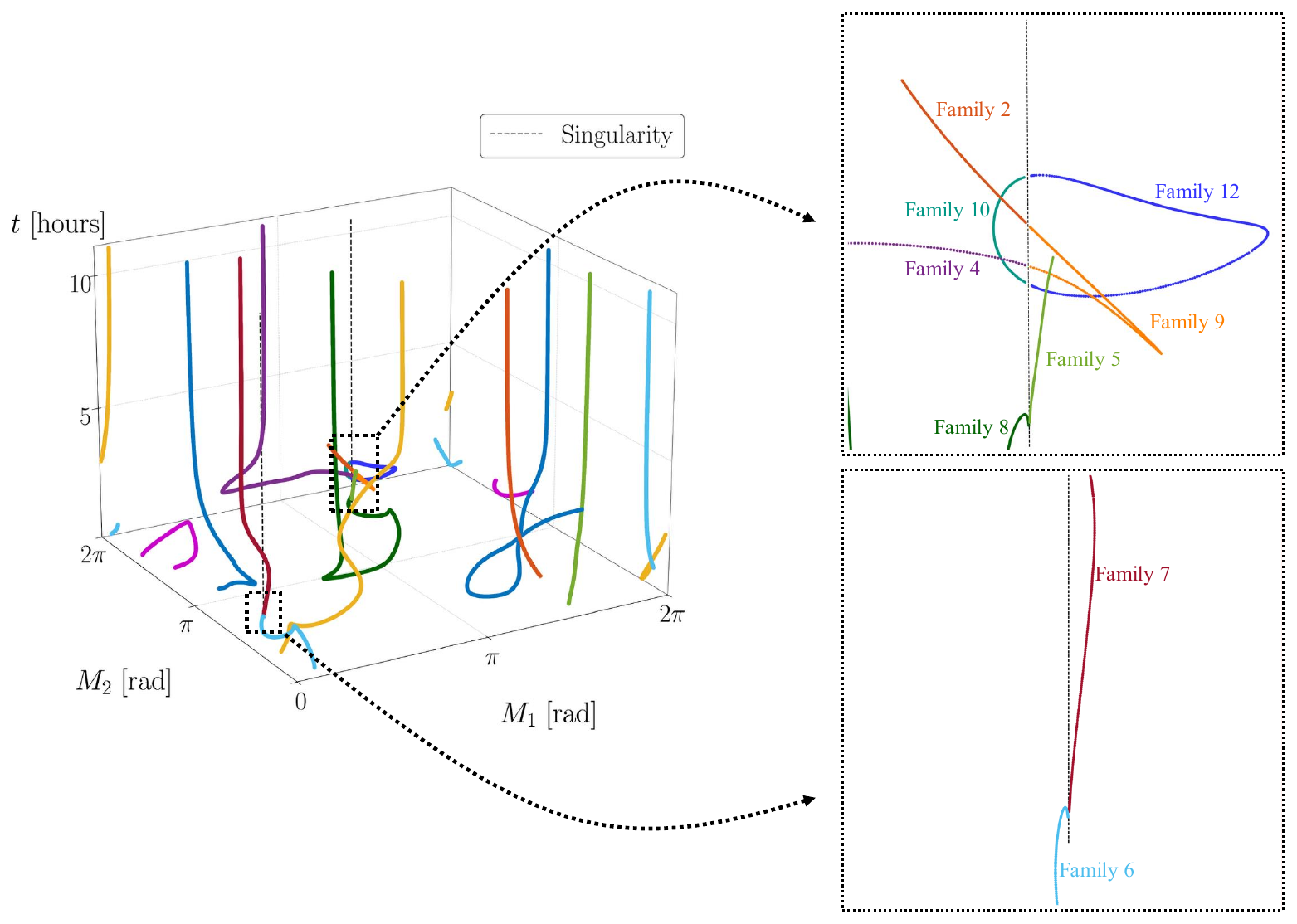}
    \caption{Global structure of \acrshort{tiot} families in the temporal domain colored by their indices (Table \ref{table:family_index}).}
    \label{fig:family_connection1}
\end{figure}

\subsubsection{\texorpdfstring{\acrshort{pvt}}{} (Specific to the \texorpdfstring{$\t-\tof$}{} Domain)}

Families generated by enforcing the temporal domain optimality constraints do not necessarily satisfy the necessary conditions from the \acrshort{pvt}; as such, each \acrshort{tiot} is evaluated for the \acrshort{pvt} conditions. \taken{The identified two-impulse solutions are locally optimal within the class of two-impulse transfers, as enforced by the stationary conditions, whereas the \acrshort{pvt} conditions provide a separate, stronger test of optimality over all impulsive trajectories.} For a given transfer, the primer vector is reconstructed by solving a two-point boundary value problem, ensuring alignment of the primer vector with the impulsive velocity changes at both burn locations, i.e., $\primer_{j} = \dv{j}/\norm{\dv{j}}$ for $j=1,2$. The resulting primer vector history is then checked against the remaining PVT conditions. Under the assumption of two impulsive burns and smooth dynamics, the key requirement is that the magnitude of the primer vector remains strictly less than unity along the interior of the transfer arc. Violation of this condition implies that an additional intermediate burn is required for local optimality.

\taken{Figure \ref{fig:porkchop_seeds_family_pvt}} displays the families colored according to both their \taken{family indices (Table \ref{table:family_index})} and \acrshort{pvt} satisfaction. Family members that do not satisfy the \acrshort{pvt} requirements are plotted in gray. Family members that satisfy the \acrshort{pvt} criteria are colored with the scheme following \taken{Table \ref{table:family_index}}. \taken{Specifically, Families~1, 2, 3, 5, 7, 9, 10, and~12 contain at least one member satisfying the \acrshort{pvt} requirements.} Note, as \acrshort{pvt} supplies necessary conditions, it is possible that local maxima and saddle solutions (red and green in \taken{Fig. \ref{fig:porkchop_seeds_family_hess}}, respectively) also satisfy the \acrshort{pvt} criteria. \taken{From a broader perspective, the gray segments delineate the regions where the two-impulse assumption ceases to be optimal and additional impulses become advantageous; the location of the peak primer magnitude further indicates the insertion point of an intermediate burn \cite{lion1968primer}. The present framework therefore maps where this transition occurs, although it does not compute the resulting multi-impulse solutions. Such fixed-revolution, multi-impulse optima with primer-vector optimality guarantees have been treated separately, e.g., \citet{arya2023generation}, and the two-impulse families identified here may serve as building blocks for these multi-impulse solutions \cite{saloglu2023existence, saloglu2025classification}. Extending the family-continuation framework itself to multi-impulse transfers remains a non-trivial direction for future work.}


\subsubsection{\label{sec:porkchop_projection}Projection to Porkchop Plots (Specific to the \texorpdfstring{$\t-\tof$}{} Domain)}

\begin{figure}[htbp]
  \centering
  \begin{subfigure}[b]{0.49\textwidth}
    \centering
    \includegraphics[page=21, width=0.99\textwidth]{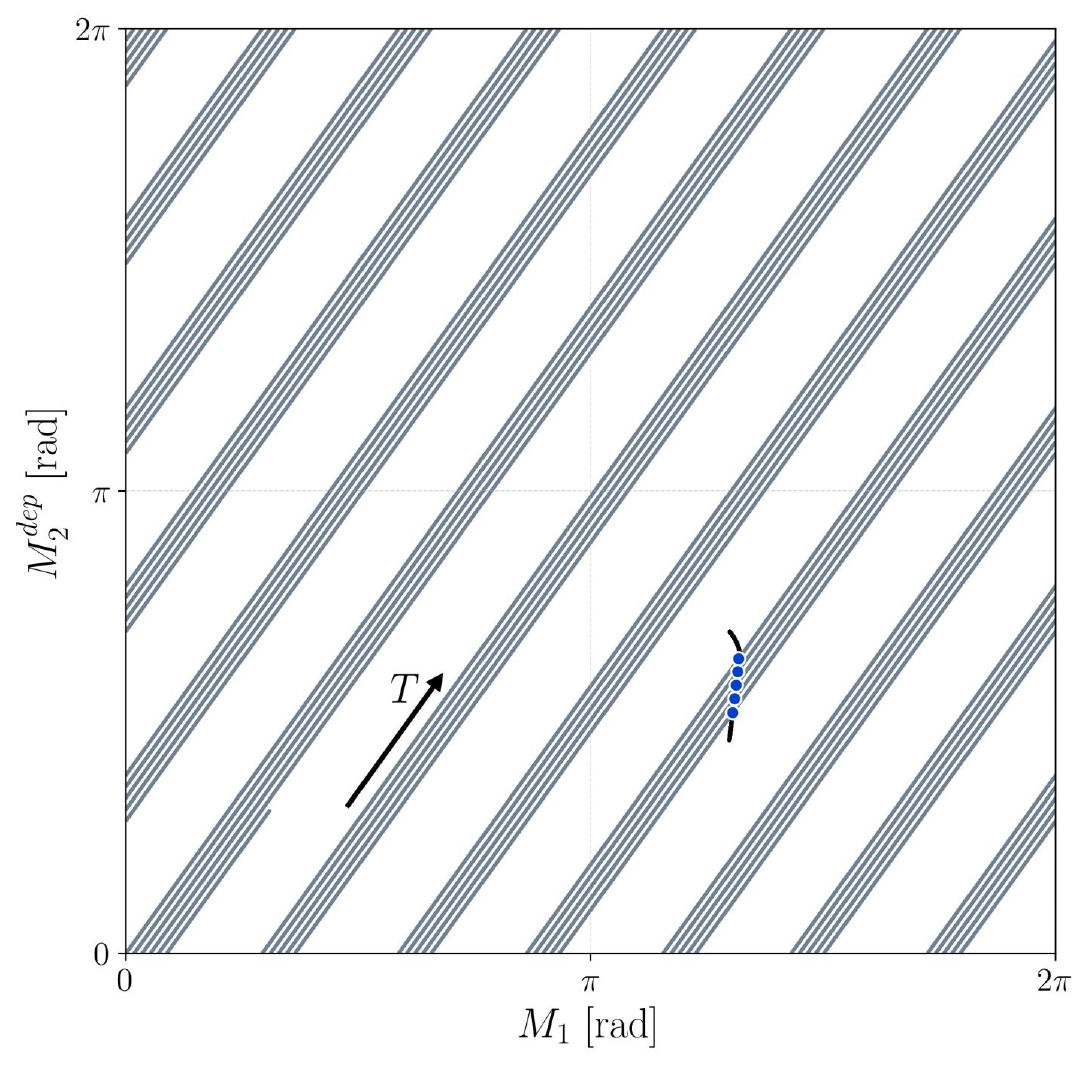}
    \caption{Two-dimensional view ($\ma{1}-\mad{2}$ domain).}
    \label{fig:parabola_seed0_porkchop_intersection_long_2d}
  \end{subfigure}
  \hfill
  \begin{subfigure}[b]{0.49\textwidth}
    \centering
    \includegraphics[page=22, width=0.99\textwidth]{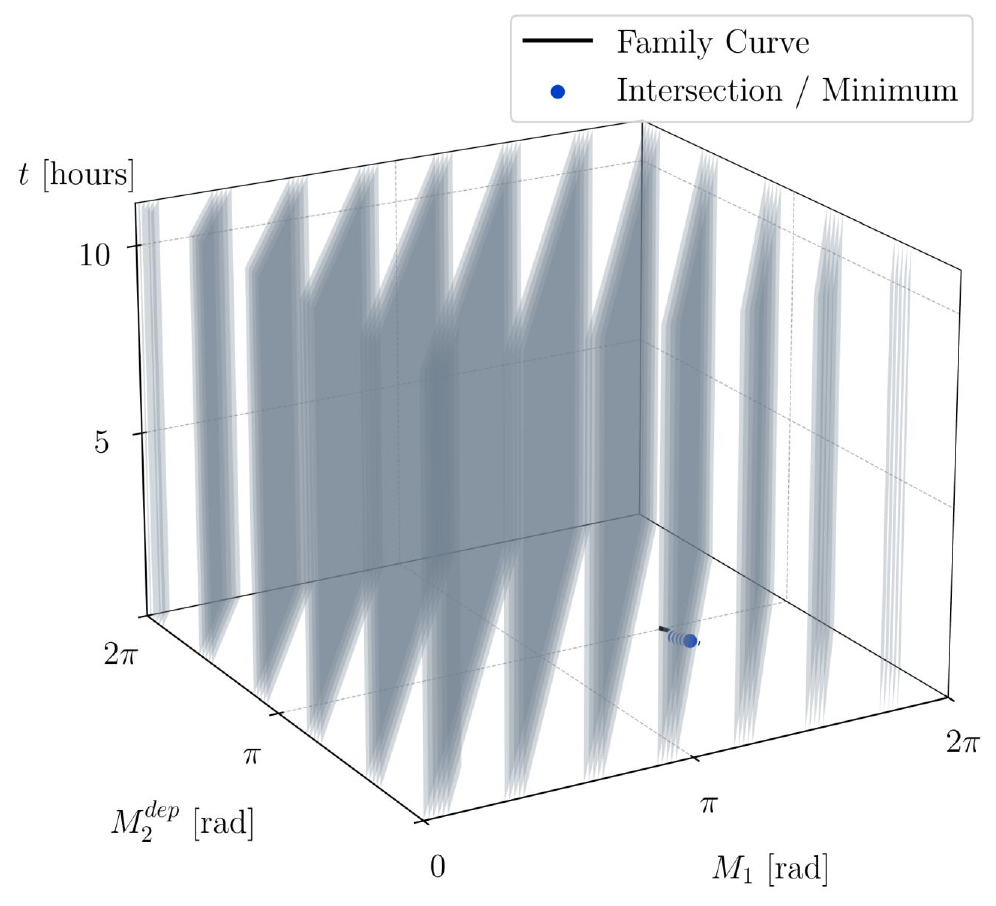}
    \caption{Three-dimensional view ($\ma{1}-\mad{2}-\tof$ domain).}
    \label{fig:parabola_seed0_porkchop_intersection_long_3d}
  \end{subfigure}
  \vspace{1em}
  \begin{subfigure}[b]{1.0\textwidth}
    \centering
    \includegraphics[page=23, width=0.99\textwidth]{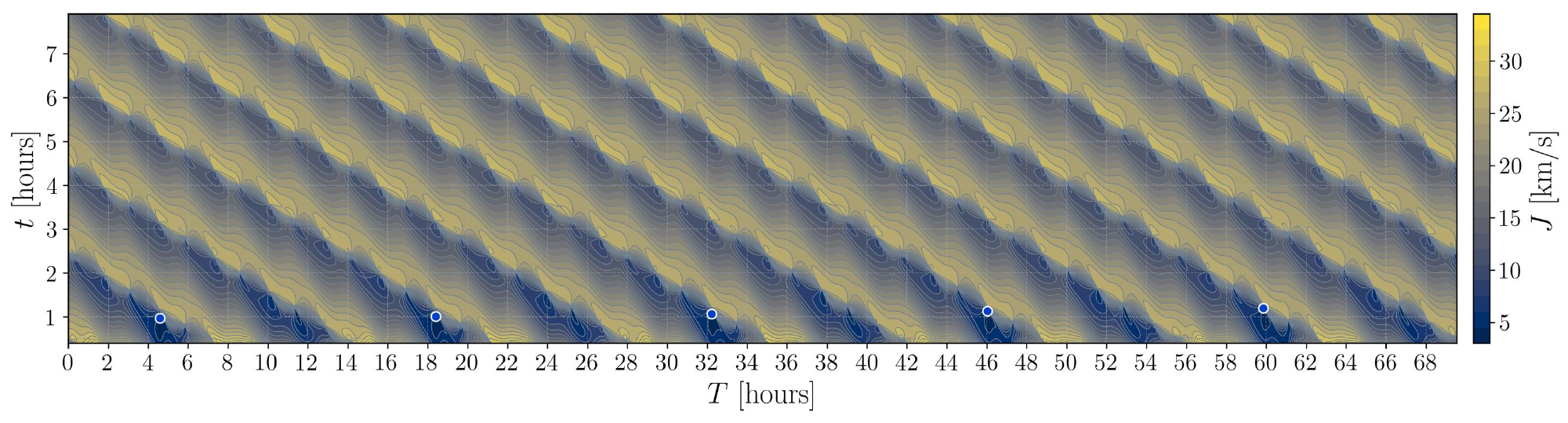}
    \caption{\acrshort{tiots} mapped onto a porkchop plot.}
    \label{fig:parabola_seed0_porkchop_intersection_long}
  \end{subfigure}
  \caption{Projection and intersection of Family 9 (Table \ref{table:family_index} and Fig. \ref{fig:family_connection1}) with a representative porkchop plot.}
  \label{fig:parabola_seed0_porkchop_intersection_long_combined}
\end{figure}

\begin{figure}[htbp]
  \centering
  \begin{subfigure}[b]{1.0\textwidth}
    \centering
    \includegraphics[page=25, width=0.99\textwidth]{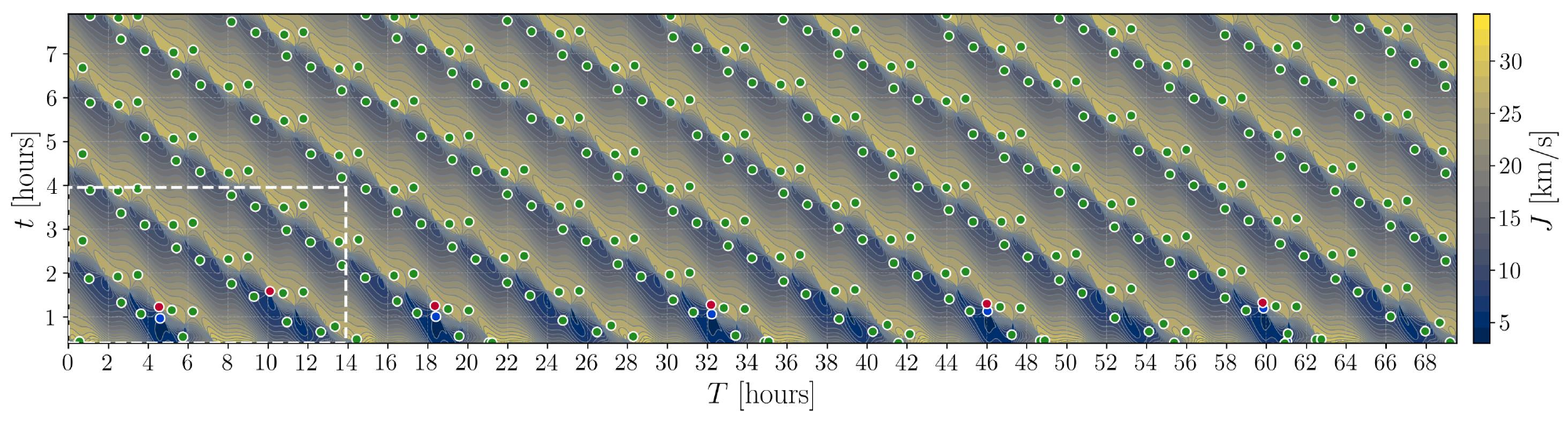}
    \caption{$d_1 = \mathrm{s}$ \taken{(short transfer)}.}
  \end{subfigure}
  \vspace{1em}
  \begin{subfigure}[b]{1.0\textwidth}
    \centering
    \includegraphics[page=24, width=0.99\textwidth]{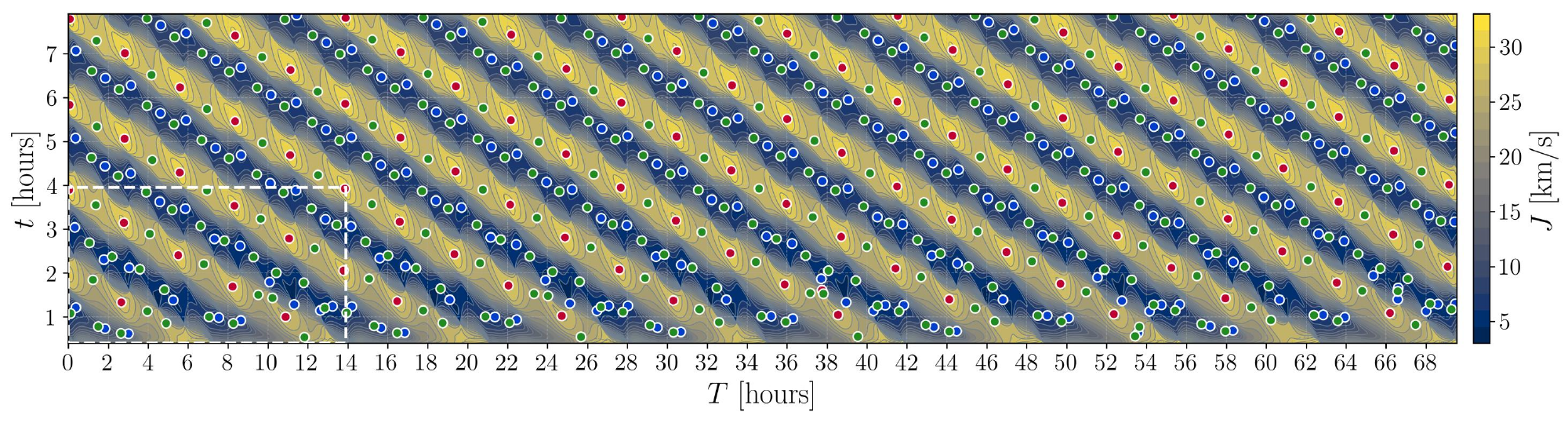}
    \caption{$d_1 = \mathrm{l}$ \taken{(long transfer)}.}
  \end{subfigure}
  \caption{Projection of all \acrshort{tiot} families (Table \ref{table:family_index} and Fig. \ref{fig:family_connection1}) onto a porkchop plot. The white dashed box indicates the limited range of the initial grid search used for seed generation (Fig. \ref{fig:seeds_porkchop}).}
  \label{fig:porkchop_seeds_porkchop_intersection}
\end{figure}

Families of \acrshort{tiots} provide a compact, global representation of solutions that would otherwise require exhaustive searches over large time windows. To connect this family-based perspective with conventional porkchop plots, family members are projected onto the $\t-\tof$ domain. The key idea is to reinterpret the ``time axis'' of a porkchop plot through the evolution of the angular variables. As time $\t$ progresses, the angles advance linearly, so that $\ma{1} = \mm{1}\t + \ma{1}^0$, $\mad{2} = \mm{2}\t + \ma{2}^0$, tracing straight ``time-lines'' in the $\ma{1}-\mad{2}$ domain\footnote{While the results thus far primarily visualize families in the $\ma{1}-\ma{2}$ domain, the $\ma{1}-\mad{2}$ domain is more effective for juxtaposing the $\t-\tof$ porkchop representation with the higher-dimensional family structure.}. In practice, these time-lines may be visualized as black straight trajectories in Fig. \ref{fig:parabola_seed0_porkchop_intersection_long_2d} initiating from $\ma{1}^0 = \ma{2}^0 = 0$. After one full revolution in either angle, the corresponding angle wraps around and the trajectory re-initializes from $0$ in that angular coordinate. When the mean motion ratio $\mm{1}/\mm{2}$ is non-resonant, the time-lines become dense over an infinite time span: as $\t\rightarrow \infty$, it visits arbitrarily close to every possible combination of $\ma{1}-\mad{2}$. In this sense, the base variable of a porkchop plot, $\t$, is decomposed into the angular variables that parametrize the family curves in $\ma{1}-\mad{2}-\tof$ space, tracing the time-lines or ``time-planes'' in Figs. \ref{fig:parabola_seed0_porkchop_intersection_long_2d} and \ref{fig:parabola_seed0_porkchop_intersection_long_3d}, respectively. With this interpretation, stationary points on porkchop plots arise naturally as intersection events. Specifically, intersections between the time-lines (or, time-planes) and an \acrshort{tiot} family curve correspond to transfer configurations that satisfy the stationary conditions in the $\t-\tof$ domain. Figure \ref{fig:parabola_seed0_porkchop_intersection_long_combined} illustrates this process for a local minimum fragment satisfying PVT of one representative family (e.g., Family 9 in Table \ref{table:family_index}); the family is illustrated in the $\ma{1}-\mad{2}-\t$ space together with the time-planes (Fig. \ref{fig:parabola_seed0_porkchop_intersection_long_3d}), and the intersection points map directly to stationary points in the porkchop plot (Fig. \ref{fig:parabola_seed0_porkchop_intersection_long}). In this case, the intersecting members are locally minimum \acrshort{tiots} and, thus, colored in blue. Repeating this process over all indexed families yields a systematic catalog of stationary points across the porkchop domain, as illustrated in Fig. \ref{fig:porkchop_seeds_porkchop_intersection}. The white dotted box corresponds to the initial bounds for the grid search supplying initial seeds within the framework (Section \ref{sec:seeds_grid}).

This projection-based approach complements conventional porkchop analysis. One may certainly inspect porkchop plots visually and/or apply gradient-based solvers to locate local extrema \cite{bang2018two}. However, porkchop plots are fundamentally constrained by grid resolution: when multiple extrema are closely spaced or visually indistinguishable, reliable detection typically requires a finer grid and additional numerical effort. Moreover, when the time window of interest is wide, a large fraction of the sampling effort becomes wasteful because only limited regions contain meaningful optima\taken{; that is, the locally optimal transfers are sparse within the broad $(\t,\tof)$ window, so a uniform grid expends most of its samples in regions far from any optimum}. Although the near-periodicity can sometimes be exploited to reduce the search domain \cite{bang2018two}, the Lambert geometry is not perfectly repetitive, and optima may appear or disappear between cycles. Without global structural insight, one may either miss relevant solutions or expend substantial computation exploring uninformative regions. In contrast, the proposed framework replaces repeated local optimization over time windows with the identification of intersection points between two curves, i.e, the time-lines and \acrshort{tiot} families. Once the families are computed, the time-line construction applies to an effectively unbounded range of $\t$ under the assumed Keplerian dynamics, enabling efficient and systematic detection of stationary points without relying on high-resolution porkchop sampling.

\section{\label{sec:cases} Case Studies}

This section demonstrates the proposed framework through a set of representative transfer scenarios. First, a baseline scenario is established that coincides with the illustrative example utilized in the preceding sections (Table \ref{table:sample_orbits_first}). This case is examined in detail to provide a comprehensive atlas of \acrshort{tiot} families, including representative transfer geometries and a classification of family branches. Within this baseline landscape, particular attention is paid to practically desirable configurations, specifically, locally minimal solutions that also satisfy the \acrshort{pvt} necessary conditions for two-impulse optimality. The burn characteristics and transfer orbit properties of these desirable configurations are analyzed to provide physical insight. Second, a controlled parametric study is conducted by varying the relative inclination while holding the remaining orbital elements fixed to the baseline configuration. This study tracks the evolution of the desirable configurations and their associated families as the inclination increases, revealing bifurcations as well as the creation or disappearance of solution branches. Finally, the broader implications and limitations of the current results are discussed, highlighting the role of the family-based viewpoint in complementing conventional porkchop-based searches.

\subsection{Baseline Transfer Scenario (Table \ref{table:sample_orbits_first})}

\subsubsection{Global Landscape and Family Classification}

Within the temporal domain, $12$ families are identified and indexed in Table \ref{table:family_index}, as plotted in Fig. \ref{fig:family_connection1}. Representative transfer geometries for each \acrshort{tiot} family are depicted in Fig. \ref{fig:orbit1_family_geometry}. Each family yields unique intersections on the porkchop plot, where detailed behaviors are recorded in \hyperref[app:family_porkchop]{Appendix F}. Naturally, not all solution geometries are of practical utility; solutions with undesirable characteristics, such as an excessively large $J$, may be screened out. 

\begin{figure}[b]
  \centering
  \begin{subfigure}[b]{0.39\textwidth}
    \centering
    \includegraphics[page=26, width=0.99\textwidth]{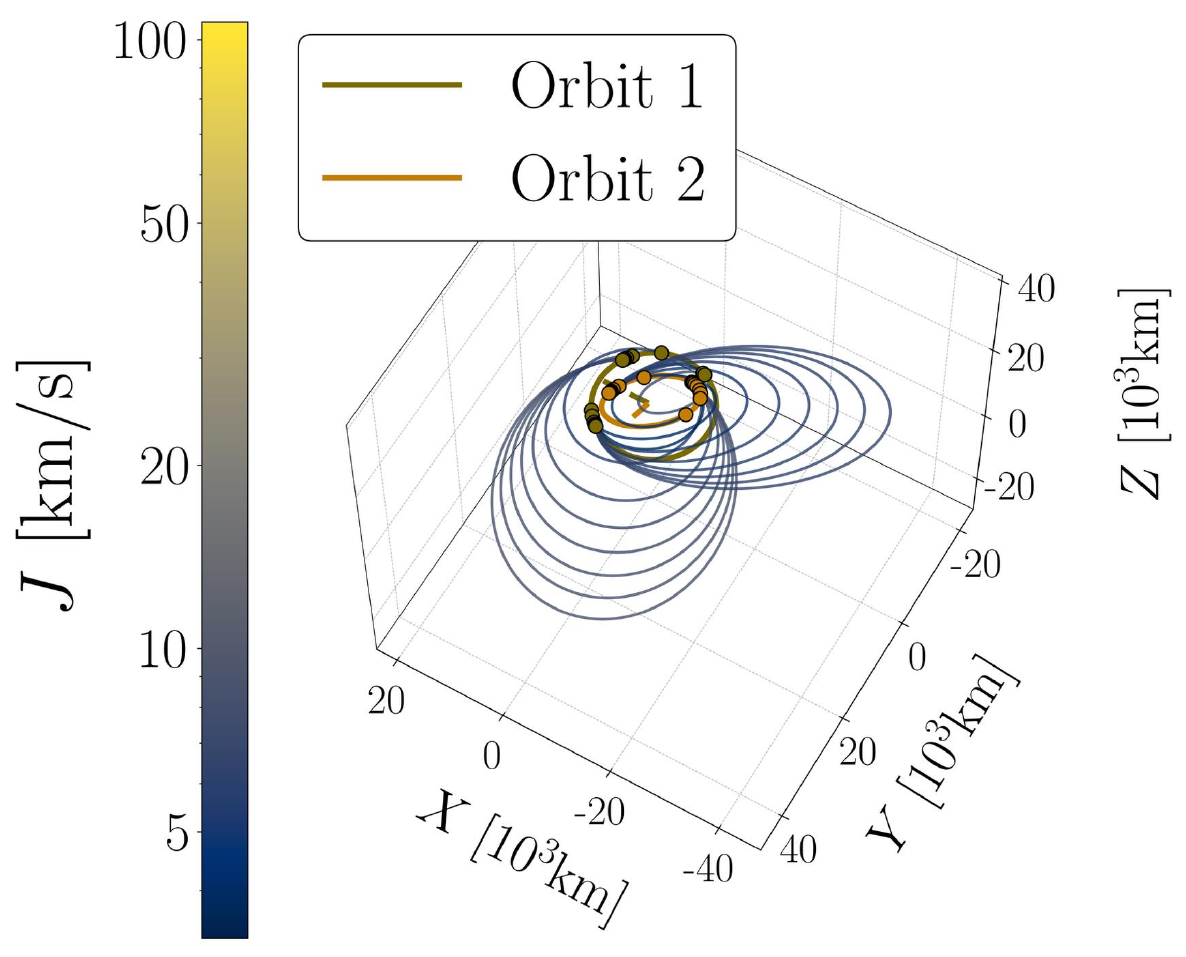}
    \caption{\label{fig:orbit1_family_geometry_seed7(family1)}Family 1.}
  \end{subfigure}
  \hfill
  \begin{subfigure}[b]{0.29\textwidth}
    \centering
    \includegraphics[page=27, width=0.99\textwidth]{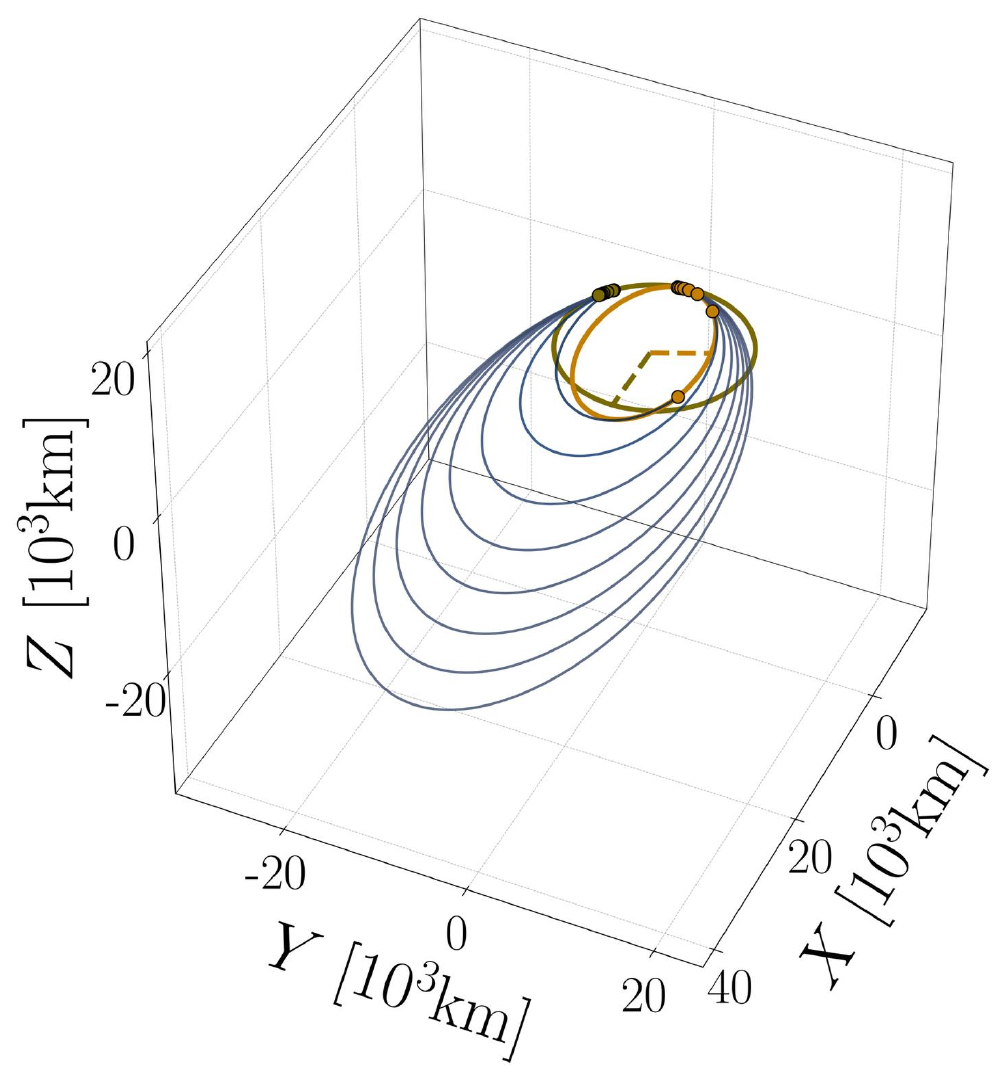}
    \caption{\label{fig:orbit1_family_geometry_seed8(family2)}Family 2.}
  \end{subfigure}
  \hfill
  \begin{subfigure}[b]{0.29\textwidth}
    \centering
    \includegraphics[page=28, width=0.99\textwidth]{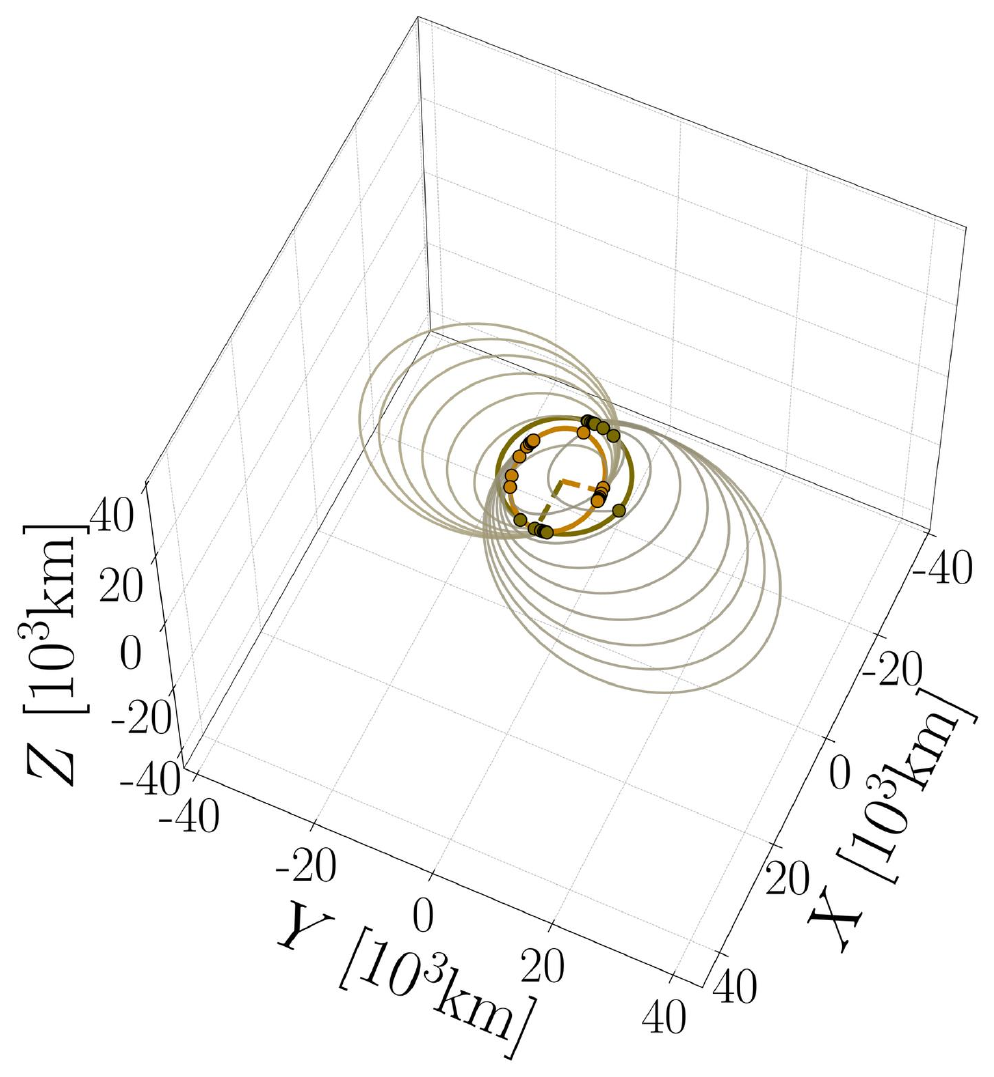}
    \caption{\label{fig:orbit1_family_geometry_seed10(family3)}Family 3.}
  \end{subfigure}
  \begin{subfigure}[b]{0.32\textwidth}
    \centering
    \includegraphics[page=29, width=0.99\textwidth]{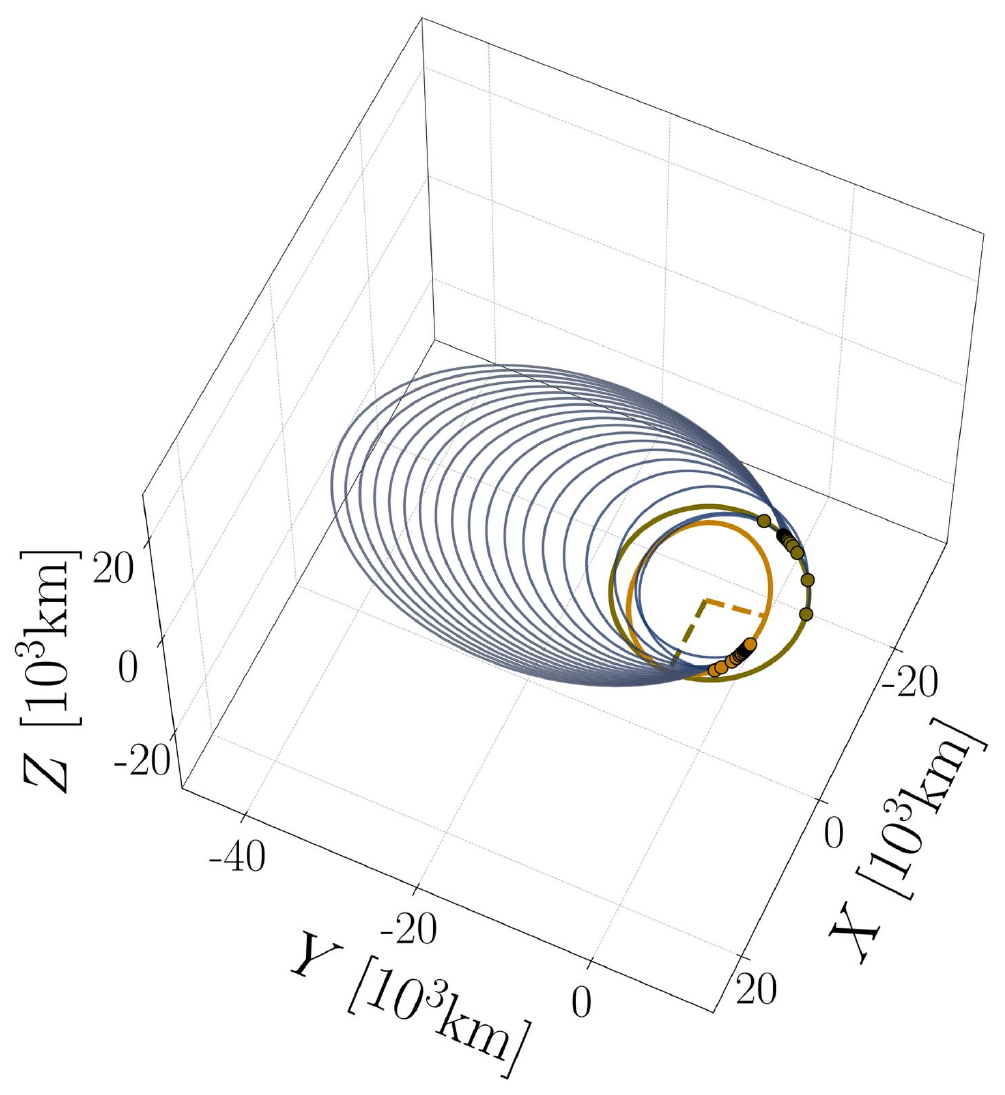}
    \caption{\label{fig:orbit1_family_geometry_seed9(family4)}Family 4.}
  \end{subfigure}
  \begin{subfigure}[b]{0.32\textwidth}
    \centering
    \includegraphics[page=30, width=0.99\textwidth]{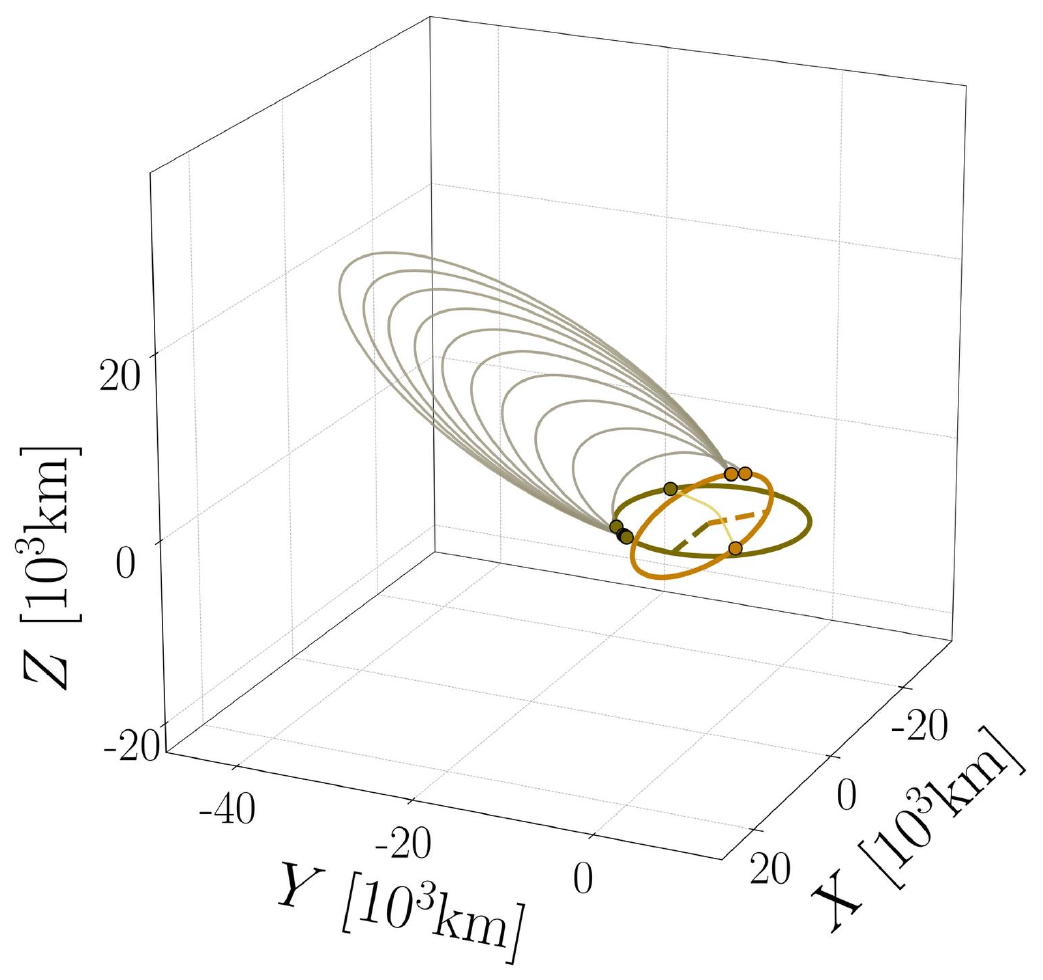}
    \caption{\label{fig:orbit1_family_geometry_seed4(family5)}Family 5.}
  \end{subfigure}
  \begin{subfigure}[b]{0.32\textwidth}
    \centering
    \includegraphics[page=31, width=0.99\textwidth]{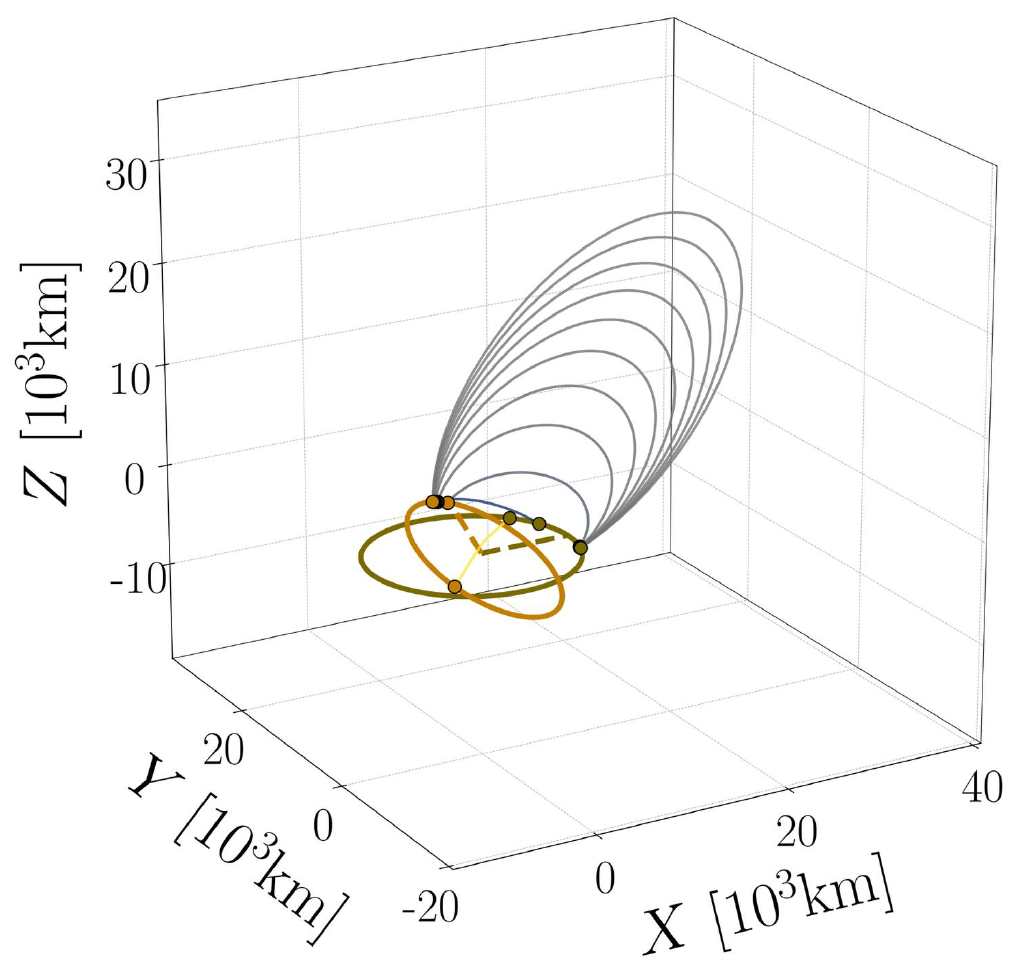}
    \caption{\label{fig:orbit1_family_geometry_seed5(family6)}Family 6.}
  \end{subfigure}
  \caption{Visual catalog of the 12 identified \acrshort{tiot} families
  (Fig. \ref{fig:family_connection1}) for the baseline scenario
  (Table \ref{table:family_index}). Continued on the next page.}
  \label{fig:orbit1_family_geometry}
\end{figure}

\begin{figure}[!tpb]\ContinuedFloat
  \centering
  \begin{subfigure}[b]{0.32\textwidth}
    \centering
    \includegraphics[page=32, width=0.99\textwidth]{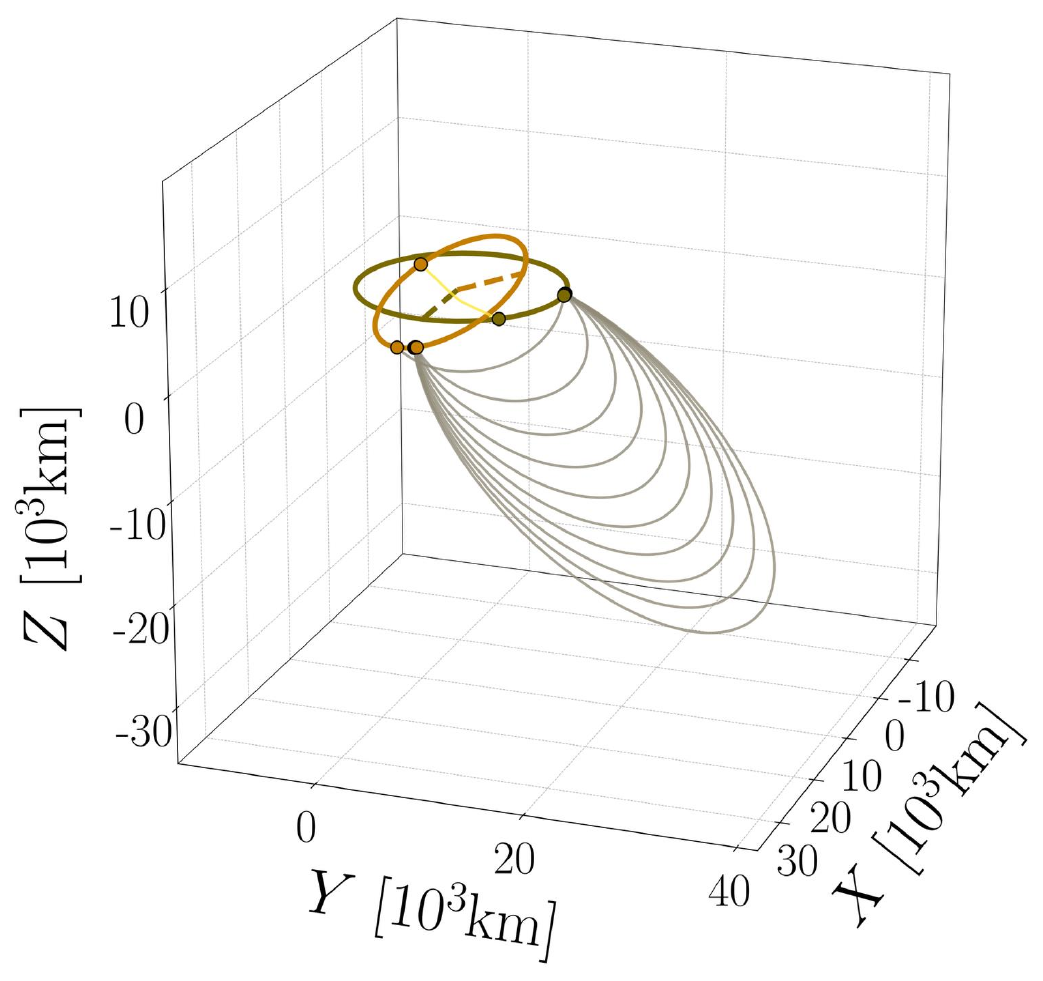}
    \caption{\label{fig:orbit1_family_geometry_seed2(family7)}Family 7.}
  \end{subfigure}
  \begin{subfigure}[b]{0.32\textwidth}
    \centering
    \includegraphics[page=33, width=0.99\textwidth]{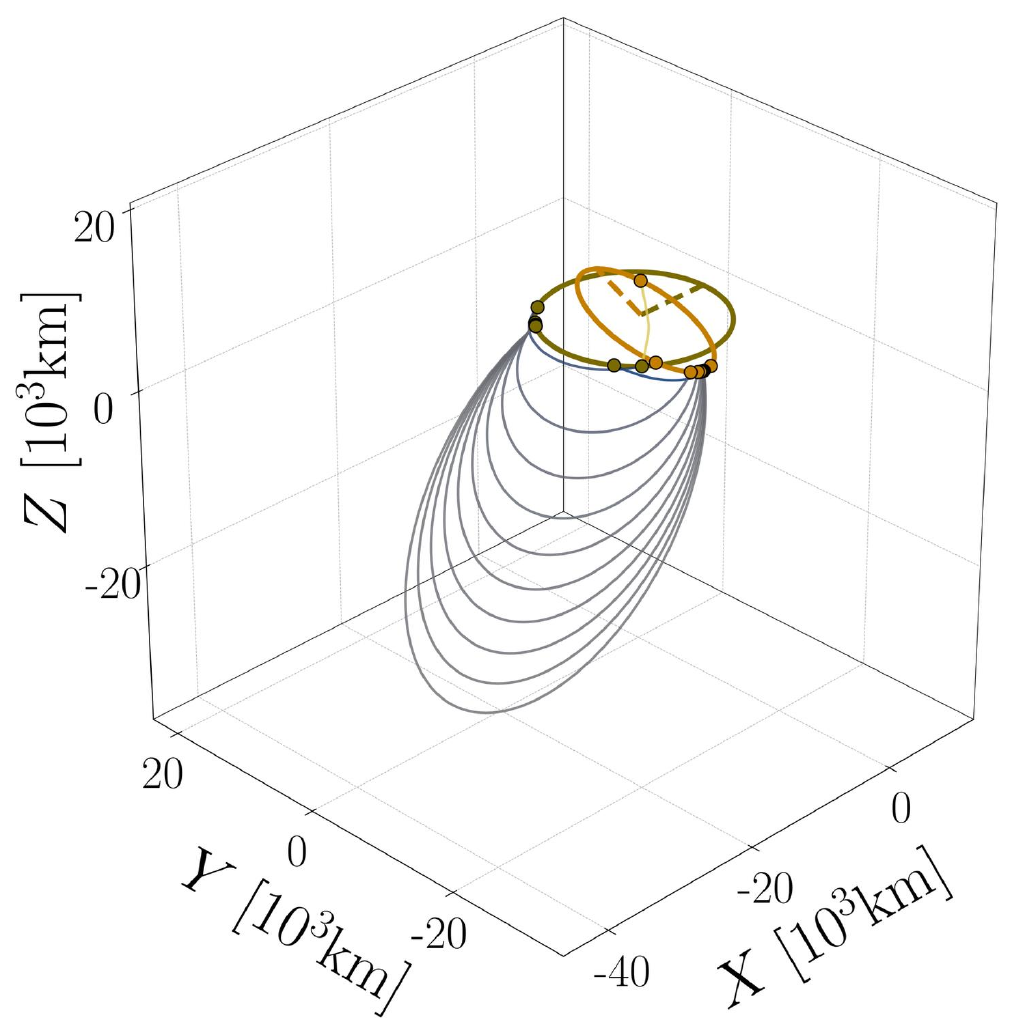}
    \caption{\label{fig:orbit1_family_geometry_seed3(family8)}Family 8.}
  \end{subfigure}
  \begin{subfigure}[b]{0.32\textwidth}
    \centering
    \includegraphics[page=34, width=0.99\textwidth]{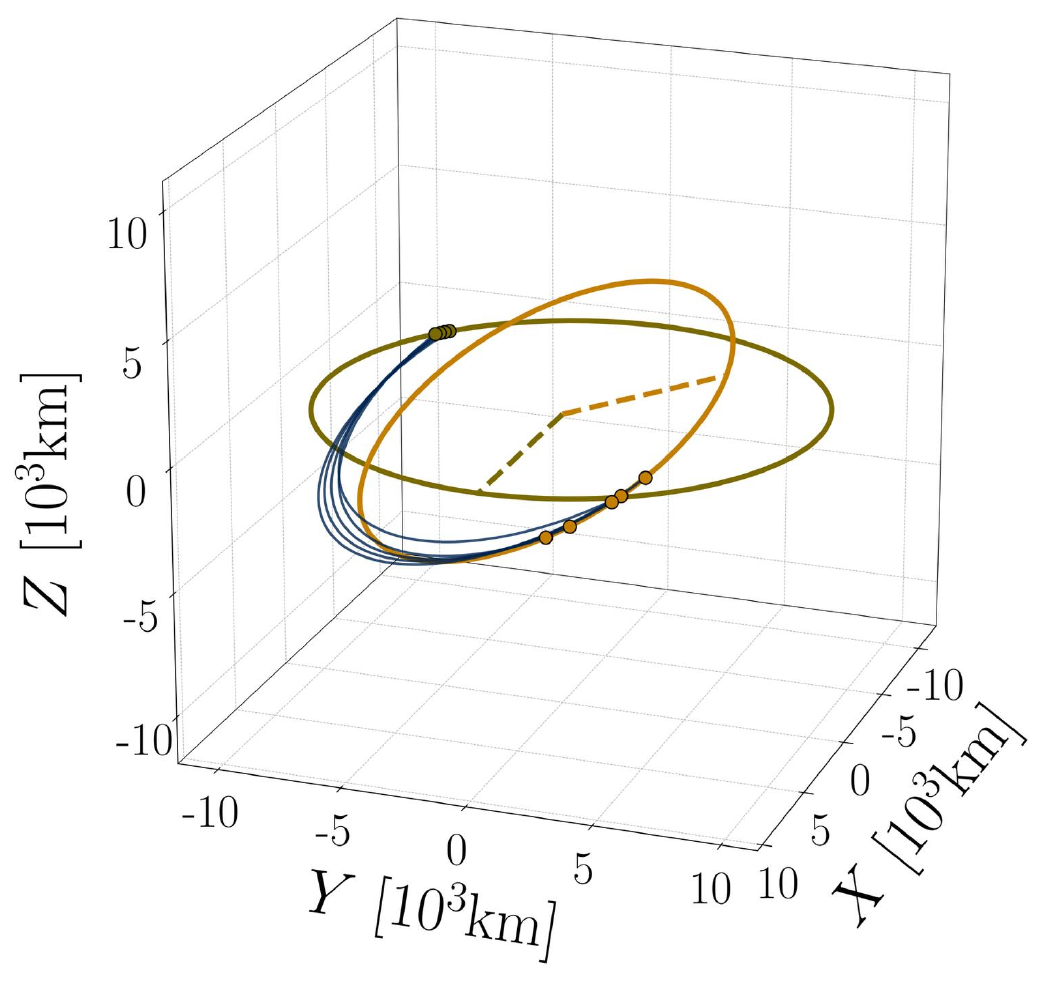}
    \caption{\label{fig:orbit1_family_geometry_seed0(family9)}Family 9.}
  \end{subfigure}
  \begin{subfigure}[b]{0.32\textwidth}
    \centering
    \includegraphics[page=35, width=0.99\textwidth]{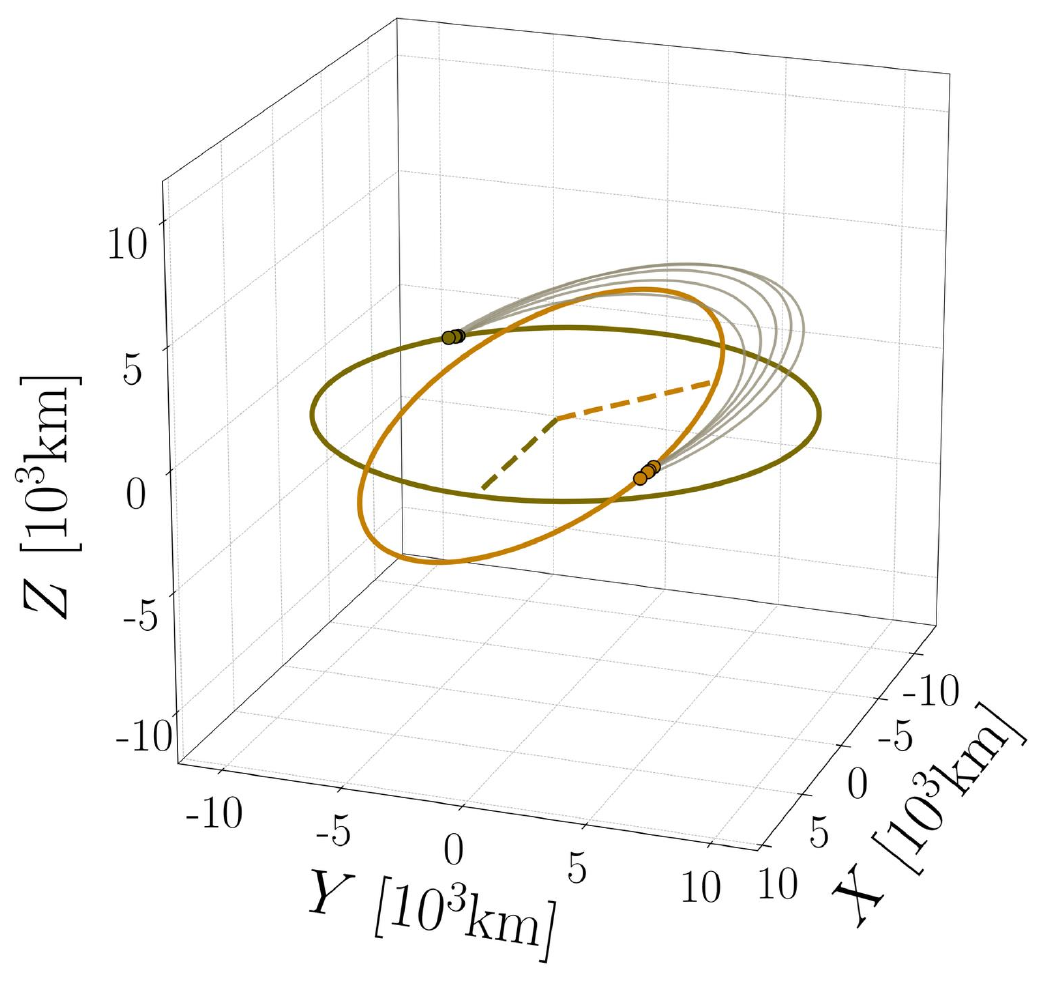}
    \caption{\label{fig:orbit1_family_geometry_seed1(family10)}Family 10.}
  \end{subfigure}
  \begin{subfigure}[b]{0.32\textwidth}
    \centering
    \includegraphics[page=36, width=0.99\textwidth]{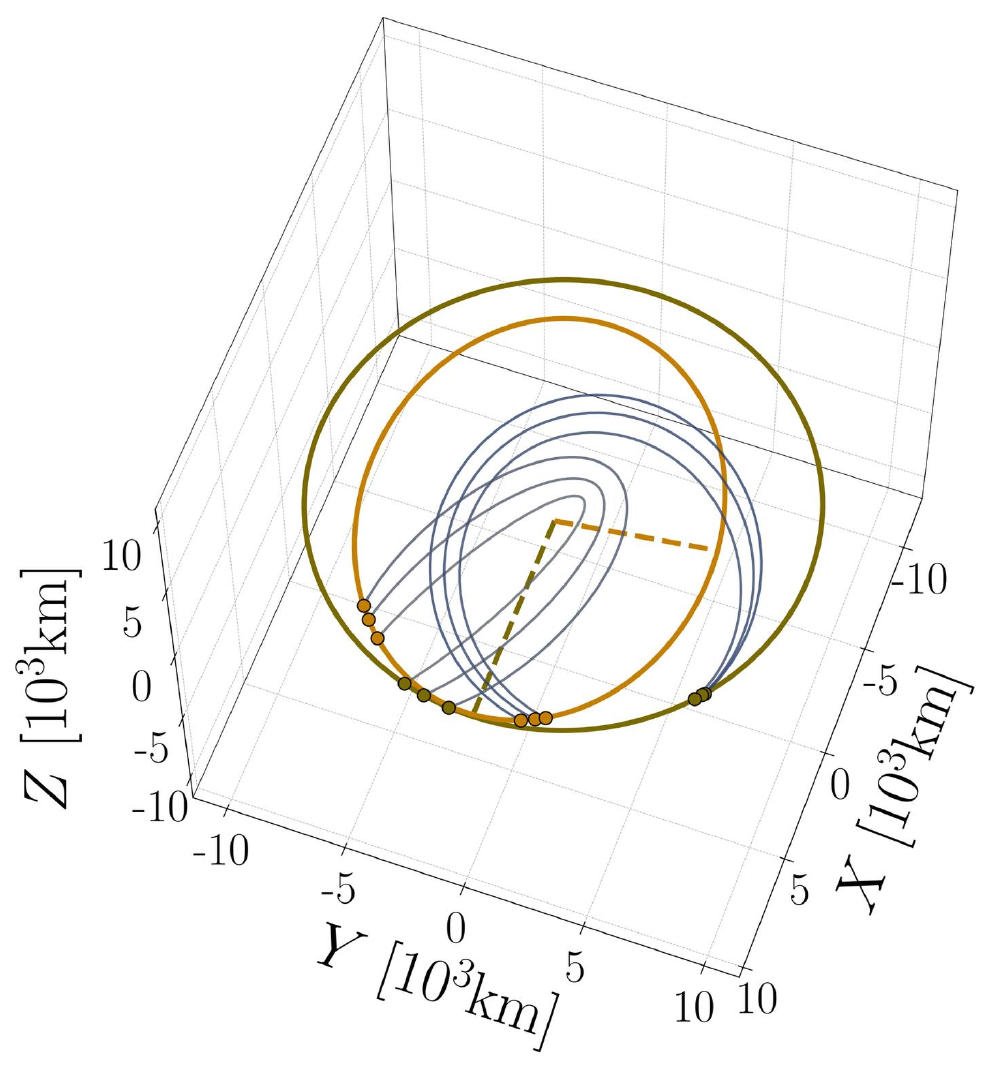}
    \caption{\label{fig:orbit1_family_geometry_seed6(family11)}Family 11.}
  \end{subfigure}
  \begin{subfigure}[b]{0.32\textwidth}
    \centering
    \includegraphics[page=37, width=0.99\textwidth]{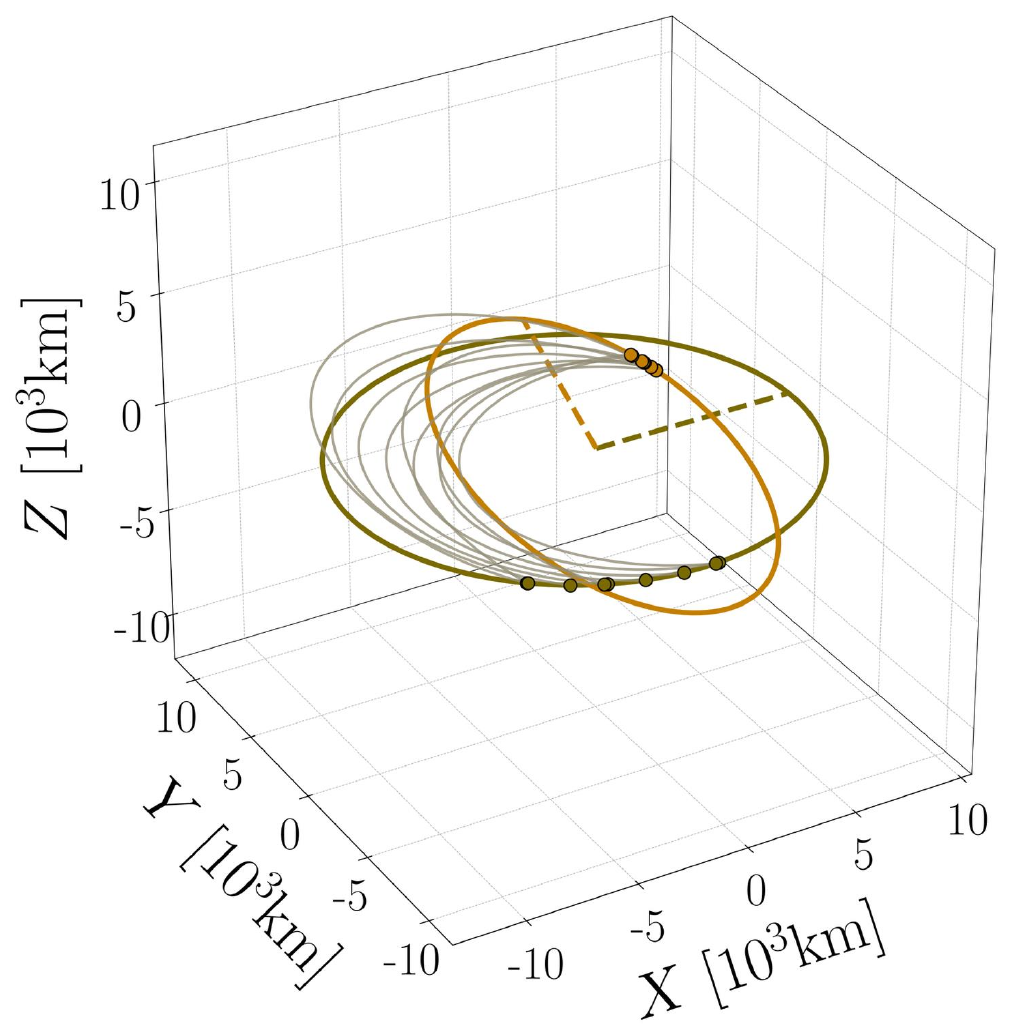}
    \caption{\label{fig:orbit1_family_geometry_seed11(family12)}Family 12.}
  \end{subfigure}
  \caption[]{Visual catalog of the 12 identified \acrshort{tiot} families
  (continued).}
\end{figure}
\FloatBarrier

\subsubsection{\label{sec:orbit1_optimal}Anatomy of Min-$J$ Geometries}

Among the various locally stationary families, those that contain local minima satisfying the \acrshort{pvt} conditions are examined further. While all members along a given family are locally stationary, the member achieving the minimum cost $J$ within the family is termed the ``min-$J$'' \acrshort{tiot}. Figures \ref{fig:orbit1_desirable_1}--\ref{fig:orbit1_desirable_3} illustrate three such configurations\footnote{While Family 3 also supplies members with locally minimal \acrshort{tiots} that satisfy the \acrshort{pvt} conditions, they represent hyperbolic Lambert arcs and are not considered desirable.}. Figures \ref{fig:orbit1_seed[0, 8]_local_optimal.png}, \ref{fig:orbit1_seed7_local_optimal.png}, and \ref{fig:orbit1_seed3_local_optimal.png} display the family branches in the solution space. The first optimal configuration spans across Families 2 and 9; thus, these two families are presented together in Fig. \ref{fig:orbit1_seed[0, 8]_local_optimal.png}. Additionally, the min-$J$ \acrshort{tiot} in each configuration is depicted in the Earth-centered inertial frame in Figs. \ref{fig:orbit1_seed[0, 8]_local_optimal_orbit.png}, \ref{fig:orbit1_seed7_local_optimal_orbit.png}, and \ref{fig:orbit1_seed3_local_optimal_orbit.png}. 

The first two geometries (Figs. \ref{fig:orbit1_desirable_1}--\ref{fig:orbit1_desirable_2}) share common characteristics. These transfers typically occur in the vicinity of the nodes, i.e., both the departure and arrival points lie near the common nodal line where the two orbital planes intersect. This configuration is also observable from the locations of the optimal members within the $\ma{1}-\ma{2}$ domain that lie close to the singular lines, $\pi_{1}$ and $\pi_{2}$. \taken{It is emphasized, however, that these optima lie in the vicinity of the nodes but not exactly on the nodal line; their transfer angles deviate slightly from $\theta = 180^\circ$. This deviation is consistent with \citet{vinh1988optimal}, who note that \emph{exact} nodal transfers ($\theta = 180^\circ$) typically do not render locally optimal solutions; the optima therefore settle just off the exact nodal configuration.} For the specific orbital combination in Table \ref{table:sample_orbits_first}, the departure orbit is associated with a larger radius from the central body (Earth). Consequently, the optimal transfers utilize a departure burn that primarily corrects for the inclination change. Table \ref{table:orbit1_desirable_burn_component} details the burn characteristics for the min-$J$ configurations. The components of the two burns, $\norm{\Delta \bm{v}_1}$ and $\norm{\Delta \bm{v}_2}$, are decomposed into the unit vectors of the \acrfull{lvlh} frame: the radial unit vector $\hat{r}$, the angular momentum direction $\hat{h}$, and the tangential direction $\hat{t}$. Burns in the $\hat{h}$-direction are responsible for changing the orbital plane ($i, \Omega$).

A key distinction regarding the plane change strategy is observed between the first two configurations. Note that the departure and arrival orbits are inclined at $i_1 = 0$ and $i_2 = 30^\circ$, respectively. Min-$J$ \acrshort{tiot} 1 performs nearly all necessary inclination change at departure ($\norm{\Delta \bm{v}_h} \approx 2.58$ km/s), resulting in a transfer orbit inclination of $i \approx 27.6^\circ$, close to the target's $30^\circ$. In contrast, min-$J$ \acrshort{tiot} 2, departing near the perigee, distributes the plane change maneuver more evenly; it reaches an intermediate inclination of only $i \approx 20.8^\circ$, leaving a significant correction ($\norm{\Delta \bm{v}_h} \approx 0.79$ km/s) for the arrival burn. This difference substantiates the underlying asymmetry between the two configurations. Given the non-circular geometry of the orbits ($e_1 = e_2 = 0.1$), the two configurations offer distinct transfer opportunities. Since the first geometry permits a departure near the apogee ($\ma{1} \approx 180^\circ$) and an arrival near the perigee ($\ma{2} \approx 0^\circ$), the optimal solution favors a larger inclination change at the first burn where the velocity is lower. In contrast, the second geometry is associated with a departure near the perigee ($\ma{1} \approx 0^\circ$) and an arrival near the apogee ($\ma{2} \approx 180^\circ$), rendering the inclination change at the first burn less efficient.

The third min-$J$ \acrshort{tiot} presents an intriguing case that contrasts with the previous examples. While the first two min-$J$ configurations naturally align with the expected geometries near the nodes, the third solution emerges as an additional, distinct local optimum. Although it bears a superficial resemblance to the first \acrshort{tiot}, it originates from a completely separate family branch, as evidenced by the disconnection between Families 8 and 9 in Fig. \ref{fig:family_connection1}. Notably, this transfer exhibits the highest cost, driven by a massive initial plane change ($\norm{\Delta \bm{v}_h} \approx 2.93$ km/s) that instantly matches the target inclination ($i \approx 30.2^\circ$). The appearance of this isolated optimal branch, existing independently from the primary nodal families, highlights the non-trivial topology of the solution space. While this analysis cannot be generalized to other orbital configurations, it supplies useful insights into transfers with non-negligible inclination change ($\Delta i = 30^\circ$).


\begin{figure}[htpb]
    \centering
    \begin{subfigure}[b]{0.48\textwidth}
        \centering
        \includegraphics[page=50, width=0.99\textwidth]{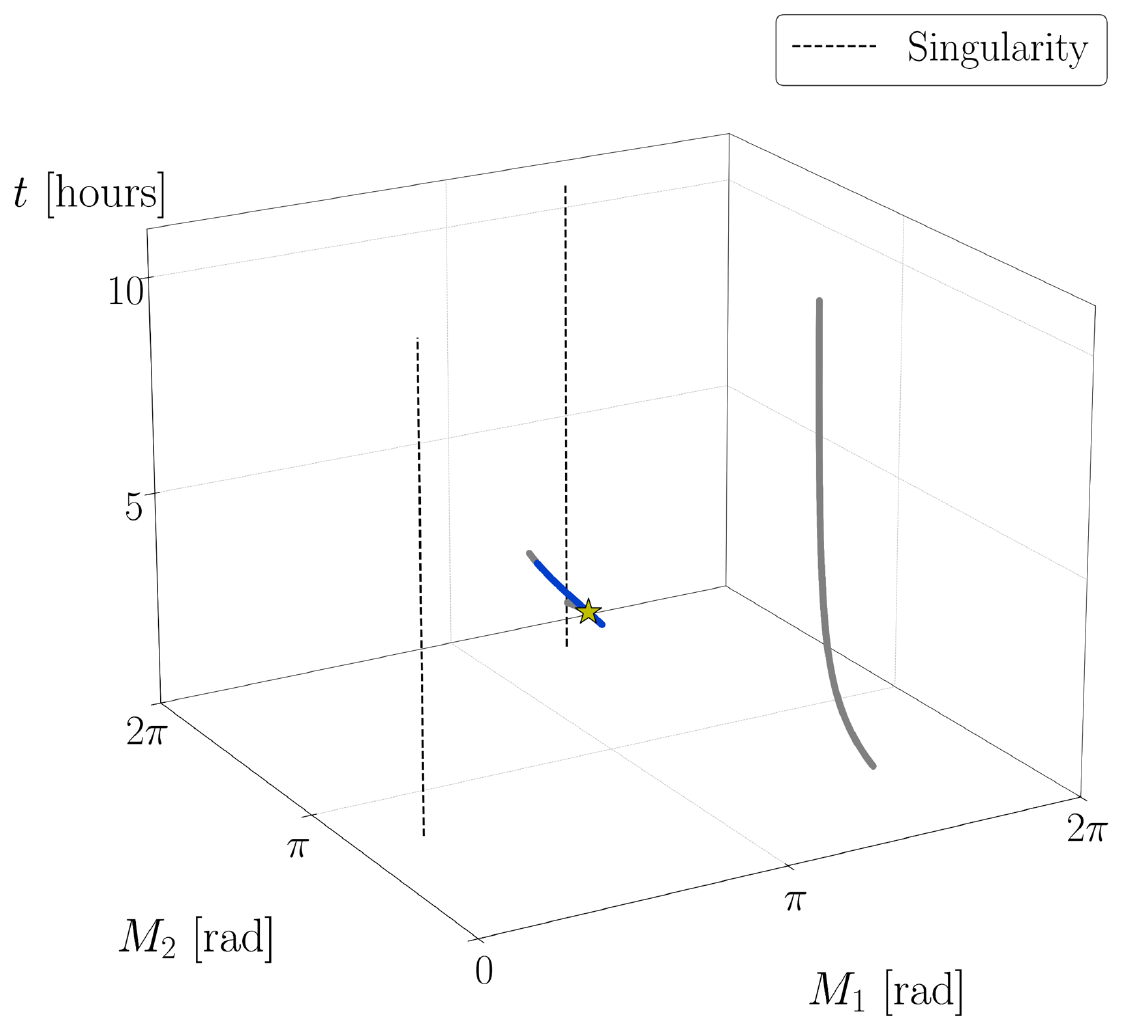}
        \caption{Families 2 and 9 in the $M_1-M_2-\tof$ space. \taken{The local-minimum solutions satisfying PVT are displayed in blue, whereas the others are plotted in gray.}}
        \label{fig:orbit1_seed[0, 8]_local_optimal.png}
    \end{subfigure}
    \begin{subfigure}[b]{0.48\textwidth}
        \centering
        \includegraphics[page=51, width=0.99\textwidth]{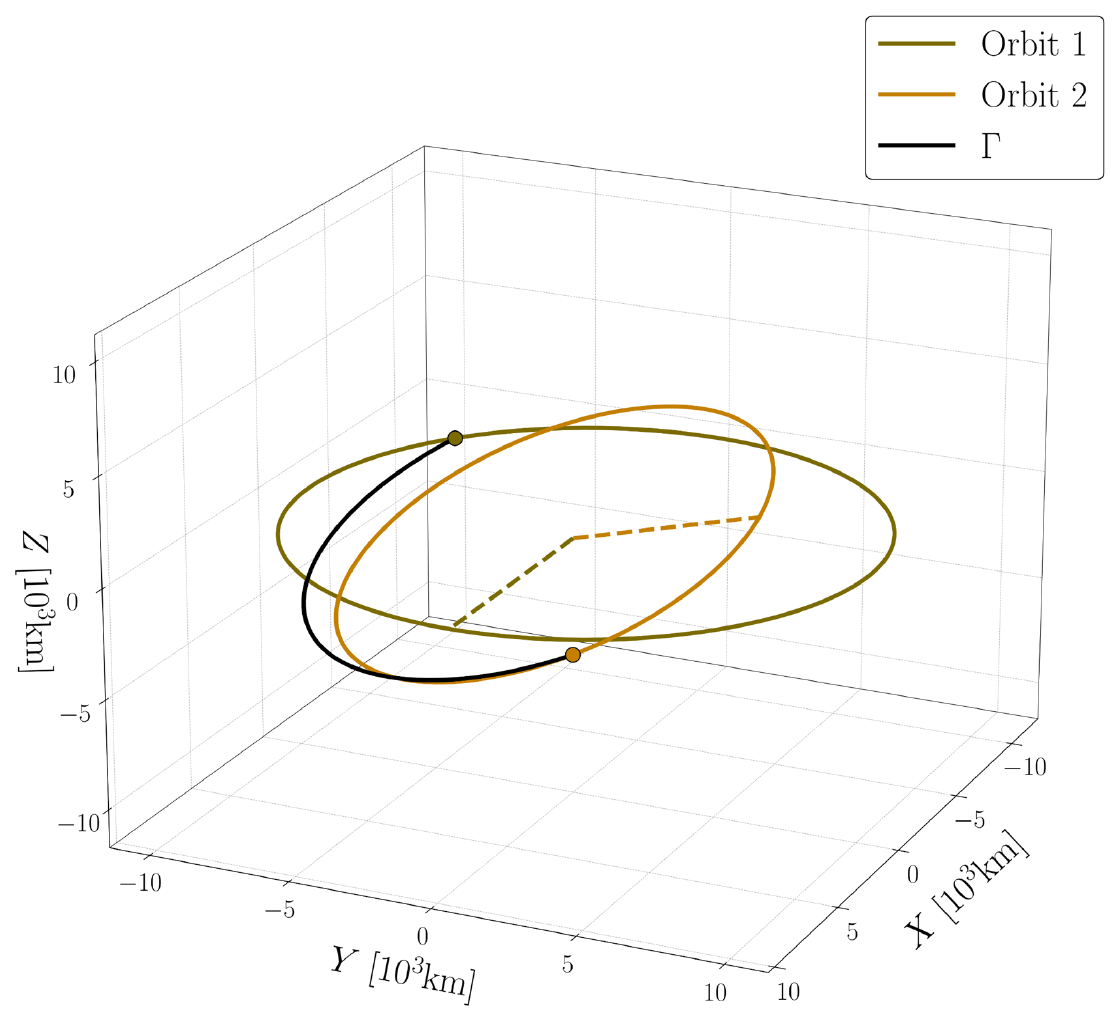}
        \caption{Min-$J$ \acrshort{tiot} 1 (yellow star in Fig. \ref{fig:orbit1_seed[0, 8]_local_optimal.png}).}
        \label{fig:orbit1_seed[0, 8]_local_optimal_orbit.png}
    \end{subfigure}
    \caption{Candidate optimal configuration 1 located in Families 2 and 9.}
    \label{fig:orbit1_desirable_1}
\end{figure}

\begin{figure}[htpb]
    \centering
    \begin{subfigure}[b]{0.48\textwidth}
        \centering
        \includegraphics[page=52, width=0.99\textwidth]{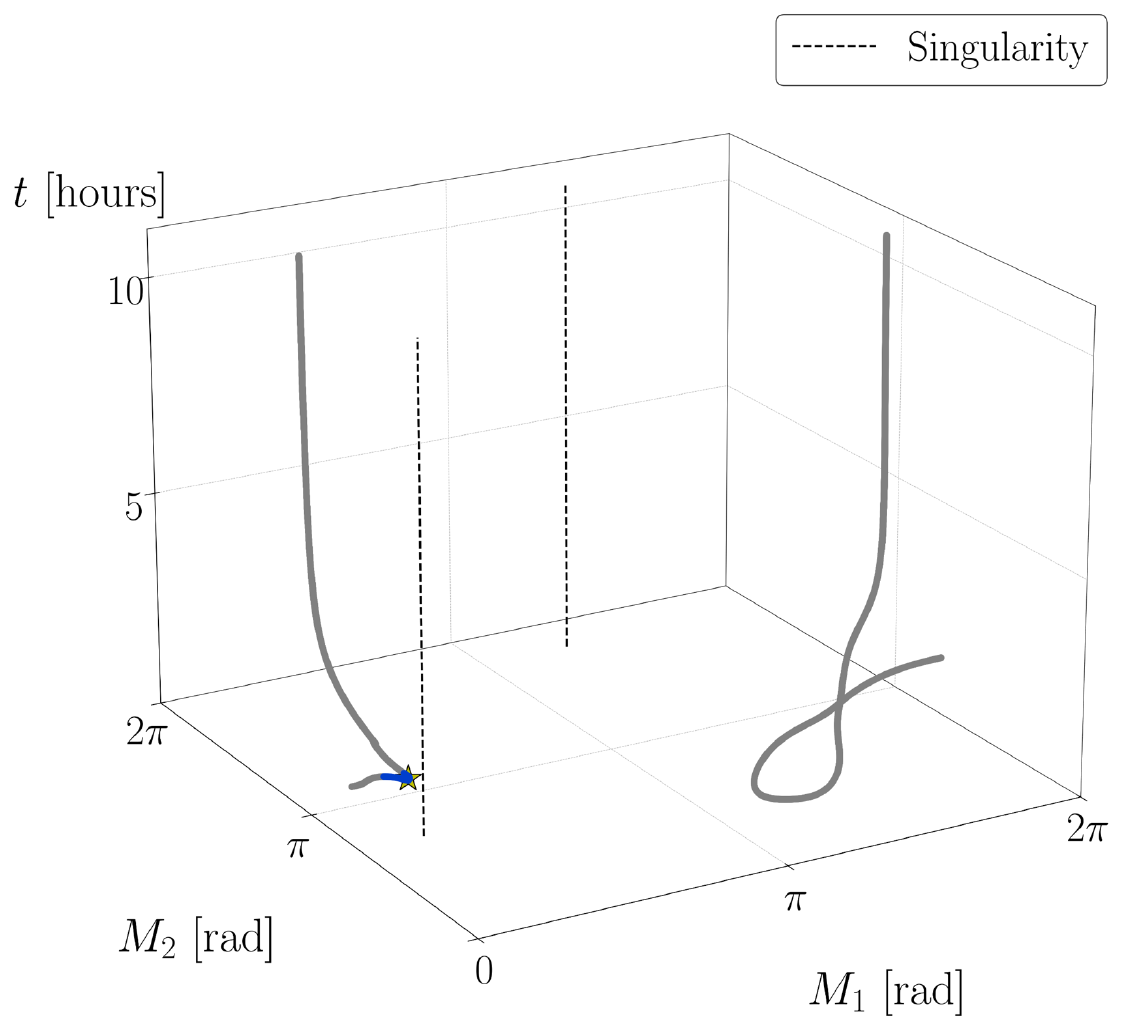}
        \caption{Family 1 in the $M_1-M_2-\tof$ space. \taken{The local-minimum solutions satisfying PVT are displayed in blue, whereas the others are plotted in gray.}}
        \label{fig:orbit1_seed7_local_optimal.png}
    \end{subfigure}
    \begin{subfigure}[b]{0.48\textwidth}
        \centering
        \includegraphics[page=53, width=0.99\textwidth]{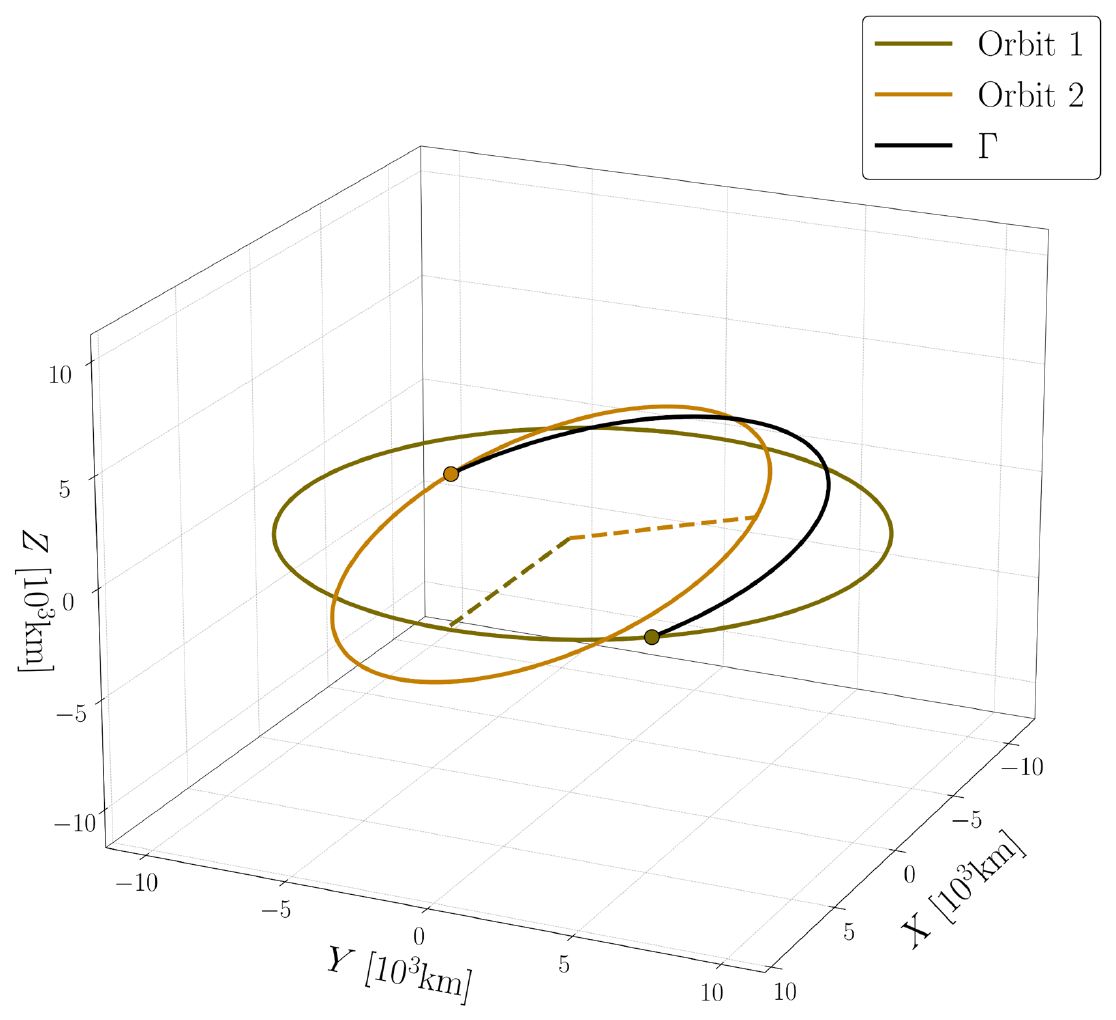}
        \caption{Min-$J$ \acrshort{tiot} 2 (yellow star in Fig. \ref{fig:orbit1_seed7_local_optimal.png}).}
        \label{fig:orbit1_seed7_local_optimal_orbit.png}
    \end{subfigure}
    \caption{Candidate optimal configuration 2 located in Family 1.}
    \label{fig:orbit1_desirable_2}
\end{figure}

\begin{figure}[htpb]
    \centering
    \begin{subfigure}[b]{0.48\textwidth}
        \centering
        \includegraphics[page=54, width=0.99\textwidth]{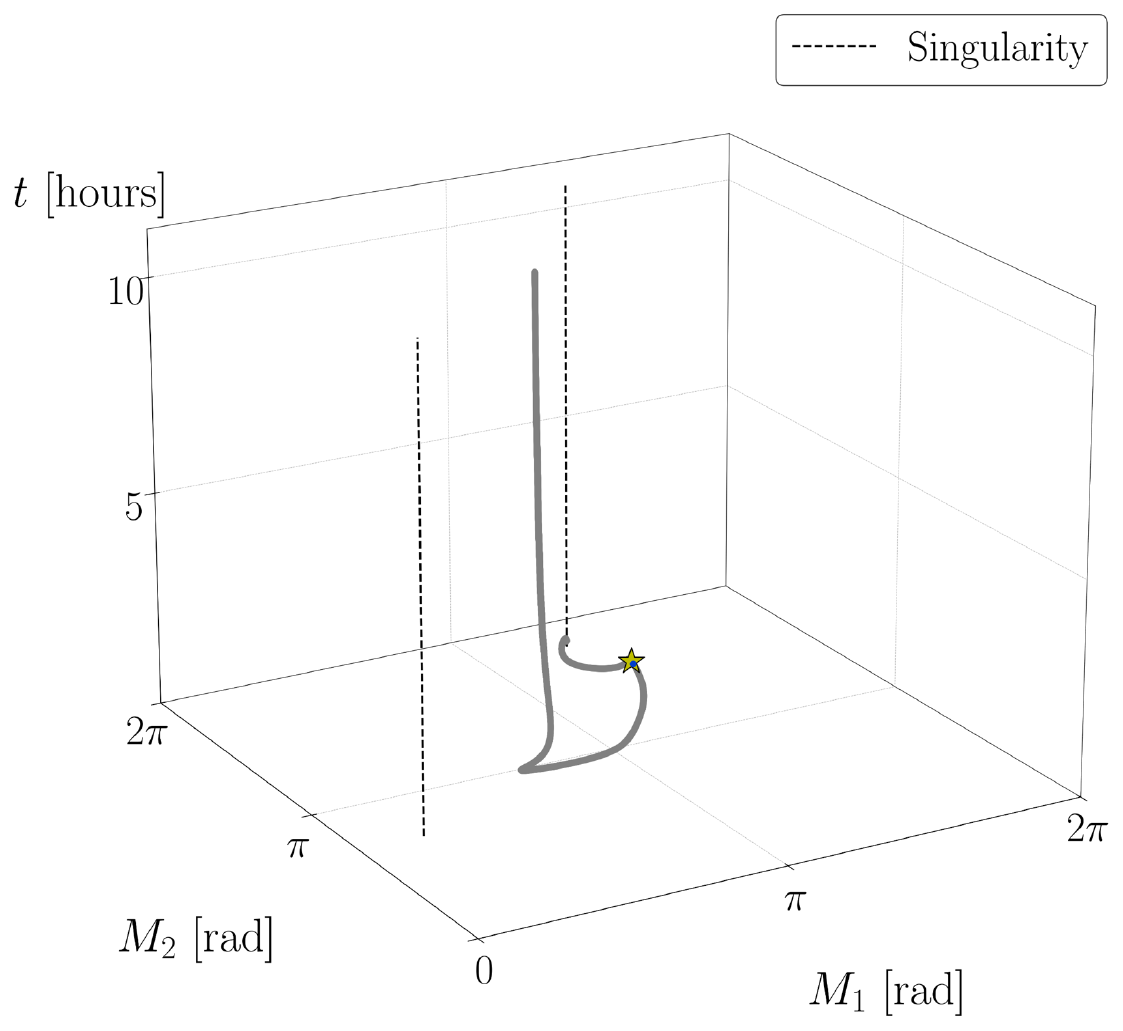}
        \caption{Family 8 in the $M_1-M_2-\tof$ space. \taken{The local-minimum solutions satisfying PVT are displayed in blue, whereas the others are plotted in gray.}}
        \label{fig:orbit1_seed3_local_optimal.png}
    \end{subfigure}
    \begin{subfigure}[b]{0.48\textwidth}
        \centering
        \includegraphics[page=55, width=0.99\textwidth]{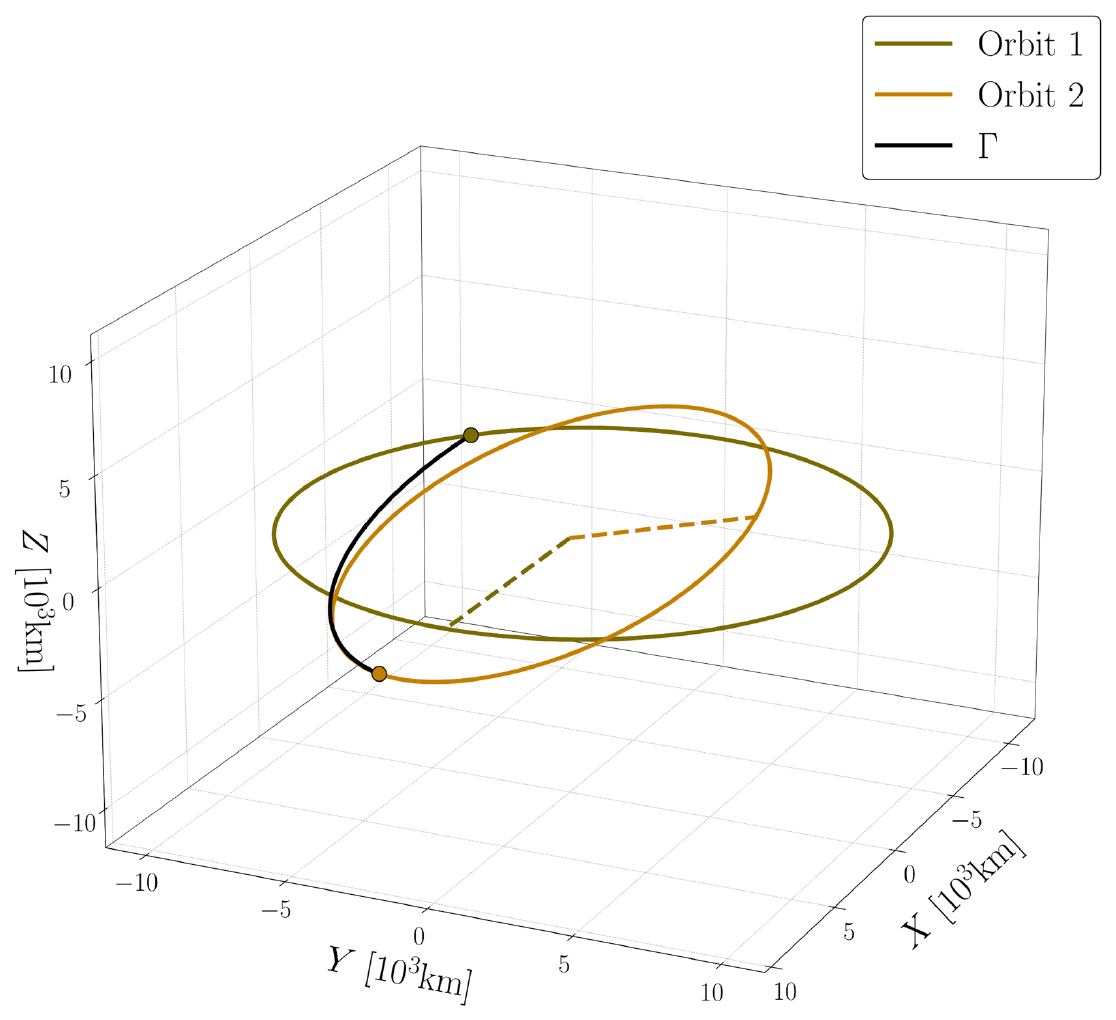}
        \caption{Min-$J$ \acrshort{tiot} 3 (yellow star in Fig. \ref{fig:orbit1_seed3_local_optimal.png}).}
        \label{fig:orbit1_seed3_local_optimal_orbit.png}
    \end{subfigure}
    \caption{Candidate optimal configuration 3 located in Family 8.}
    \label{fig:orbit1_desirable_3}
\end{figure}

\begin{table}[htpb]
\centering
\caption{\label{table:orbit1_desirable_burn_component}
Detailed burn characteristics and impulse components for the three min-$J$ configurations (Figs.~\ref{fig:orbit1_desirable_1}--\ref{fig:orbit1_desirable_3})}
\begin{tabular*}{\textwidth}{@{\extracolsep\fill}lcccccc}
\toprule
 & \multicolumn{2}{c}{Min-$J$ \acrshort{tiot} 1}
 & \multicolumn{2}{c}{Min-$J$ \acrshort{tiot} 2}
 & \multicolumn{2}{c}{Min-$J$ \acrshort{tiot} 3} \\
\cmidrule{2-3}\cmidrule{4-5}\cmidrule{6-7}
 & Burn 1 & Burn 2 & Burn 1 & Burn 2 & Burn 1 & Burn 2 \\
\cmidrule{2-3}\cmidrule{4-5}\cmidrule{6-7}
Total $J$ [km/s]        
                        & \multicolumn{2}{c}{$3.34$} 
                        & \multicolumn{2}{c}{$3.67$}
                        & \multicolumn{2}{c}{$3.86$} \\
\midrule
$\norm{\Delta \bm{v}}$ [km/s] 
                        & $2.83$ & $0.52$
                        & $2.49$ & $1.18$
                        & $3.09$ & $0.77$ \\
$\norm{\Delta \bm{v}_r}$ [km/s] 
                        & $0.78$ & $0.26$ 
                        & $0.63$ & $0.54$
                        & $0.67$ & $0.51$ \\
$\norm{\Delta \bm{v}_t}$ [km/s] 
                        & $0.84$ & $0.40$ 
                        & $0.72$ & $0.70$
                        & $0.70$ & $0.50$ \\
$\norm{\Delta \bm{v}_h}$ [km/s] 
                        & $2.58$ & $0.21$ 
                        & $2.30$ & $0.77$
                        & $2.93$ & $0.28$ \\
\bottomrule
\end{tabular*}
\footnotetext{Note: \acrshort{tiots} 1, 2, and 3 correspond to Figs.~\ref{fig:orbit1_seed[0, 8]_local_optimal_orbit.png}, \ref{fig:orbit1_seed7_local_optimal_orbit.png}, and \ref{fig:orbit1_seed3_local_optimal_orbit.png}, respectively.}
\end{table}


\subsubsection{\label{sec:long_horizon}Long-Horizon Planning via Family Projection}

Following the projection procedure described in Section~\ref{sec:porkchop_projection}, the portions of the \acrshort{tiot} families that are locally minimal and satisfy the \acrshort{pvt} conditions are projected onto the $\ma{1}-\mad{2}$ domain, as depicted in Fig.~\ref{fig:orbit1_desirable_planning}. The intersections of the time-lines and these optimal family segments indicate the availability of each optimal geometry as a function of wait time $T$.

\begin{figure}[htpb]
    \centering
    \includegraphics[page=74, width=0.8\textwidth]{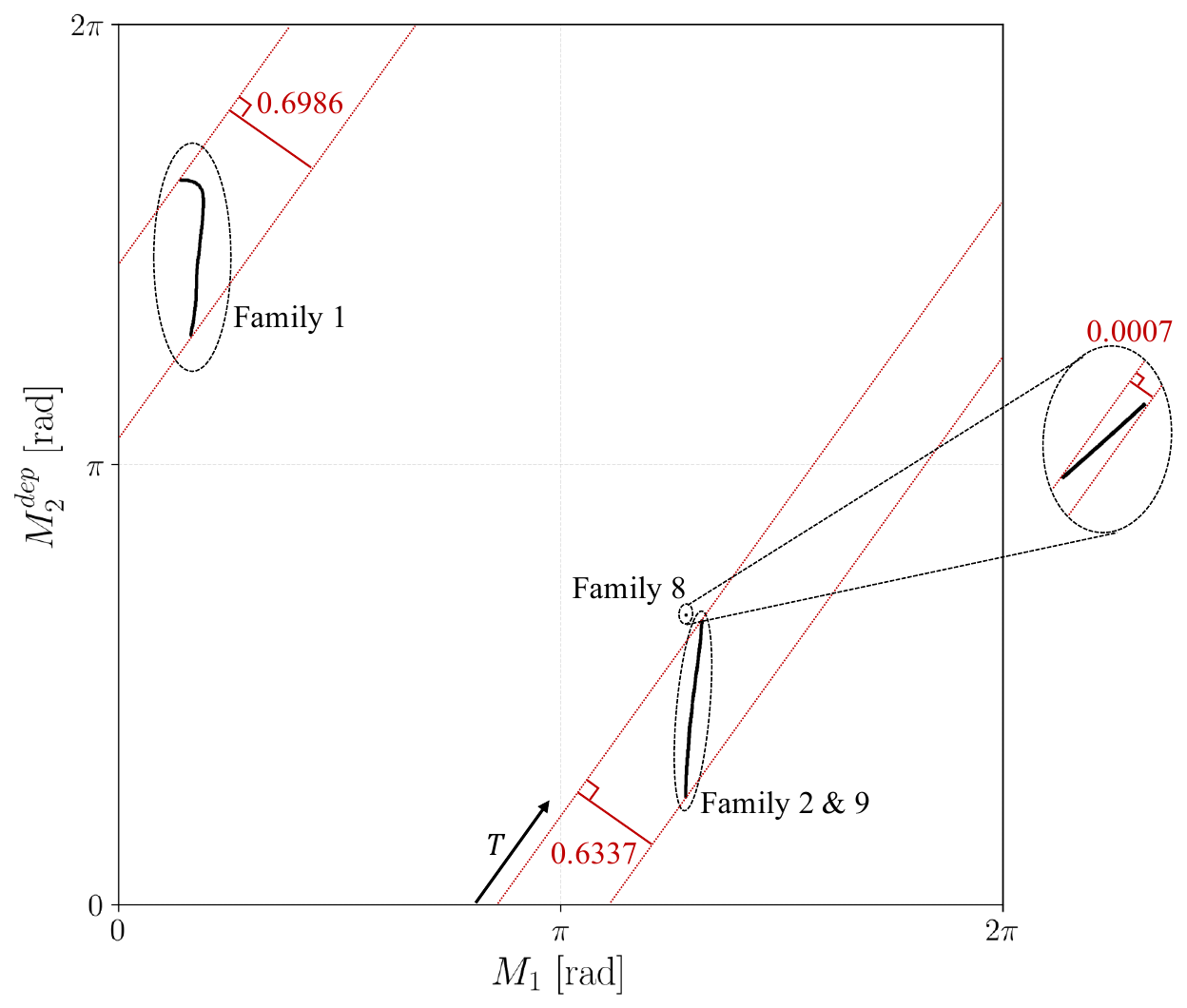}
    \caption{Identification of locally minimal transfer opportunities over $1$ year for the baseline scenario (Table~\ref{table:sample_orbits_first}). Black dots indicate crossings between the time-lines and locally minimal family segments satisfying the \acrshort{pvt} conditions.}
    \label{fig:orbit1_desirable_planning}
\end{figure}

For the baseline transfer scenario (Table~\ref{table:sample_orbits_first}), the orbital periods are not in resonance; consequently, the time-lines require infinite time to fill the full space of angular variable combinations, gradually covering the $\ma{1}-\mad{2}$ domain. In such generic, non-resonant configurations, the length of the \acrshort{tiot} family portions within the angular domain directly correlates to the ``robustness'' of the optimal configurations, i.e., how frequently a given optimal configuration is encountered by the advancing time-lines. A longer family segment perpendicular to the time-line direction implies more frequent crossings and, therefore, more frequent transfer opportunities.

To quantify this relationship, crossings between the time-lines and the locally optimal family segments are recorded over a representative time horizon of $1$ year. Table~\ref{table:crossing_stats} summarizes the number of crossings for each family, along with the corresponding perpendicular ($\perp$) length of the family within the angular domain. Among the three configurations, Family 8 (Fig.~\ref{fig:orbit1_desirable_3}) occupies only a narrow length within the angular domain and consequently generates only a single crossing over the entire $1$-year horizon. In contrast, the first two configurations exhibit substantially longer family segments and correspondingly more frequent transfer opportunities, directly confirming the correlation between angular-domain family length and temporal recurrence. 

\begin{table}[htpb]
\centering
\caption{\label{table:crossing_stats}
Number of crossings and perpendicular family length for each min-$J$ configuration over $1$ year.}
\begin{tabular*}{\textwidth}{@{\extracolsep\fill}llcc}
\toprule
Family No. & Representative Configuration & $\perp$ length [rad] & Crossings ($1$ yr) \\
\midrule
Family 2 \& 9 & Min-$J$ \acrshort{tiot} 1 (Fig.~\ref{fig:orbit1_desirable_1}) & 0.6337 & 544 \\
Family 1 & Min-$J$ \acrshort{tiot} 2 (Fig.~\ref{fig:orbit1_desirable_2}) & 0.6986 & 598 \\
Family 8 & Min-$J$ \acrshort{tiot} 3 (Fig.~\ref{fig:orbit1_desirable_3}) & 0.0007 & 1 \\
\bottomrule
\end{tabular*}
\end{table}


The results demonstrate the utility of the family-based framework for long-horizon mission planning. As established in Section~\ref{sec:porkchop_projection}, once the optimal families are constructed, detecting crossings with the time-lines reduces to a geometric intersection computation whose cost does not scale with the length of the planning horizon. For the present baseline scenario, all optimal transfer opportunities over the $1$-year horizon are enumerated almost instantly, confirming that the family-based approach enables efficient and systematic long-range planning that would be impractical with conventional epoch-by-epoch grid searches.

\subsection{\taken{Parametric Study: Topological Evolution under Inclination Change}}

The generalized orbital transfer problem involves a high-dimensional parameter space governed by the orbital elements of the departure and arrival orbits. While the proposed framework is fully capable of exploring variations across all these dimensions, attempting to analyze simultaneous variations would obscure the fundamental mechanisms driving the solution topology. Therefore, this study adopts a controlled parametric approach to isolate the impact of the plane change requirement as an exemplary analysis. While fixing other orbital parameters in Table \ref{table:sample_orbits_first}, multiple values for inclination are examined on the arrival side. A total of $15$ cases are examined, tracking the families from $i_2 = 0^\circ$ to $10^\circ$ with $1^\circ$ incremental steps, followed by larger intervals at $i_2 = 15^\circ, 20^\circ, 25^\circ, 30^\circ$. Special attention is directed toward the emergence and deformation of the practically desirable configurations identified previously with increasing inclination gap. By assessing the topology at incremental levels of inclination, it is possible to identify critical transitions, such as bifurcations, reconfigurations, or the disappearance of solution branches, that connect the familiar coplanar solutions to the complex structures observed for the baseline scenario discussed in Section \ref{sec:orbit1_optimal}.

\subsubsection{Coplanar Case (\texorpdfstring{$i_2 = 0^\circ$}{})}

The three desirable configurations identified in the preceding section (Figs \ref{fig:orbit1_desirable_1}-\ref{fig:orbit1_desirable_3}) originate from two distinct family branches in the coplanar transfer scenario. Denoting these families as ``root'' families, their behaviors in the three-dimensional solution space and the respective min-$J$\footnote{As defined in Section~\ref{sec:orbit1_optimal}, the ``min-$J$'' \acrshort{tiot} denotes the member achieving the minimum cost $J$ within the family.} \acrshort{tiots} along each family are supplied in Fig. \ref{fig:orbit7_desirable}. Note that Root family 1 (Fig. \ref{fig:orbit7_desirable_optimal}) displays a closed-loop behavior, hosting a wide range of \acrshort{tiot} solutions that are locally minimal and simultaneously satisfy the \acrshort{pvt} requirements. In contrast, Root family 2 (Fig. \ref{fig:orbit7_desirable_suboptimal}) originates from asymptotes at $\tof \rightarrow \infty$ and notably, contains no member that satisfies the \acrshort{pvt} necessary conditions while being locally minimal.

\begin{figure}[htpb]
    \centering
    \begin{subfigure}[b]{0.48\textwidth}
        \centering
        \includegraphics[page=57, width=0.99\textwidth]{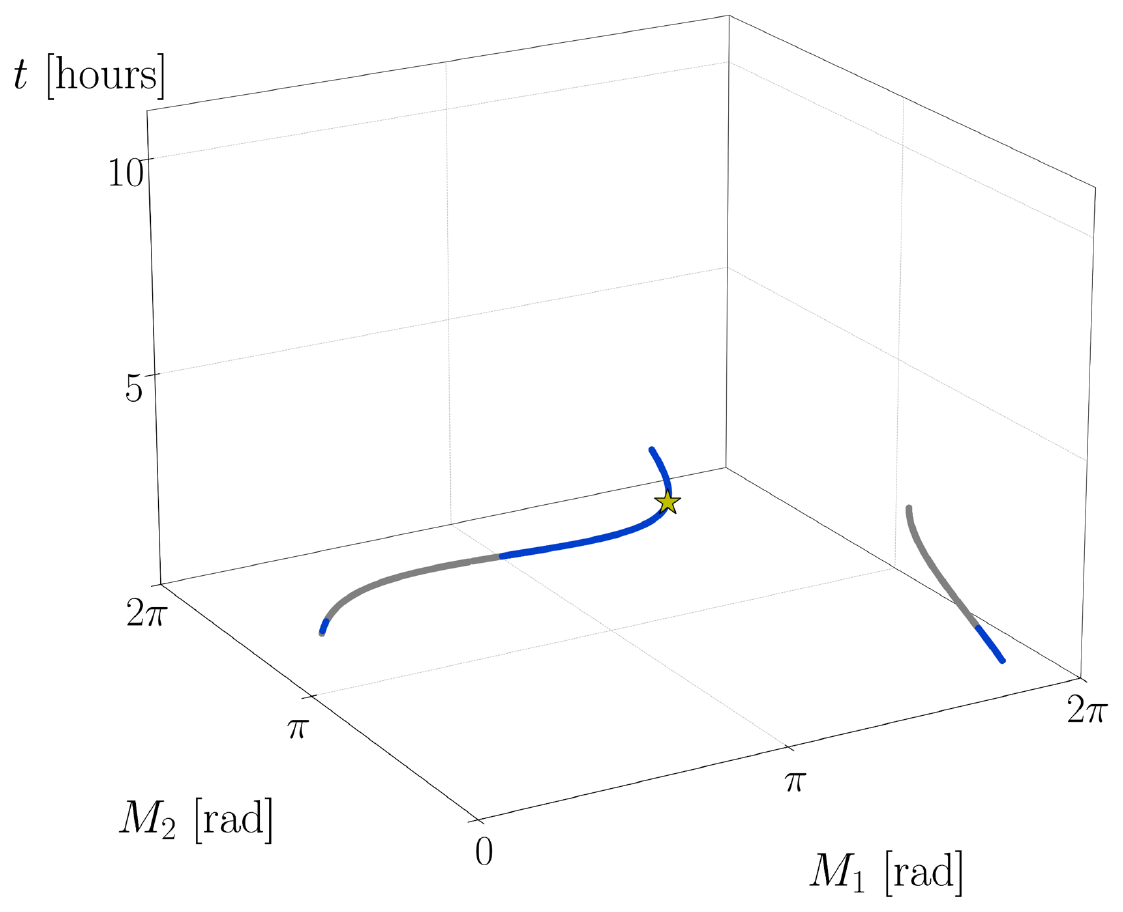}
        \caption{Root family 1. \taken{The local-minimum solutions satisfying PVT are displayed in blue, whereas the others are plotted in gray.}}
        \label{fig:orbit7_desirable_optimal}
    \end{subfigure}
    \begin{subfigure}[b]{0.48\textwidth}
        \centering
        \includegraphics[page=58, width=0.99\textwidth]{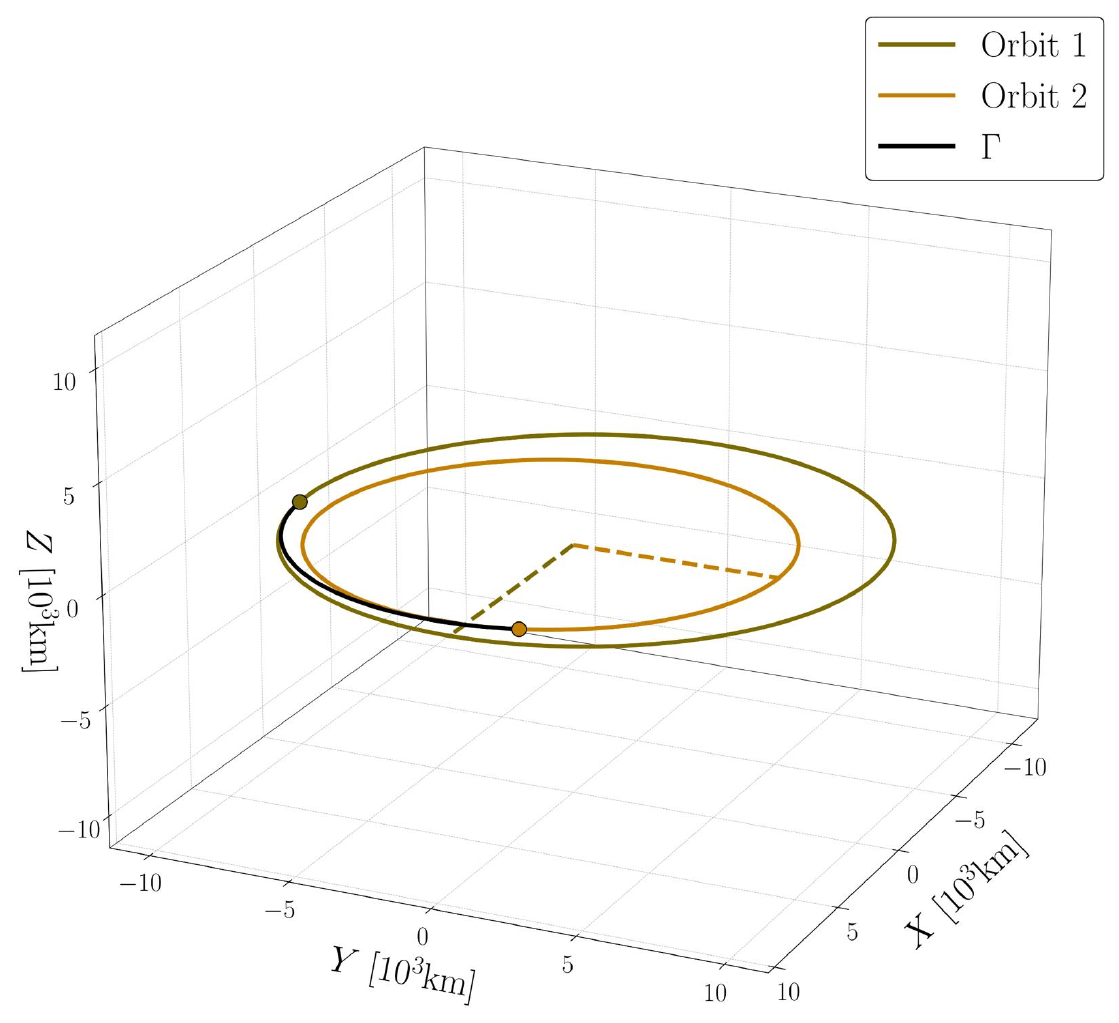}
        \caption{Min-$J$ \acrshort{tiot} along the family.}
        \label{fig:orbit7_desirable_optimal_transfer}
    \end{subfigure}
    \begin{subfigure}[b]{0.48\textwidth}
        \centering
        \includegraphics[page=59, width=0.99\textwidth]{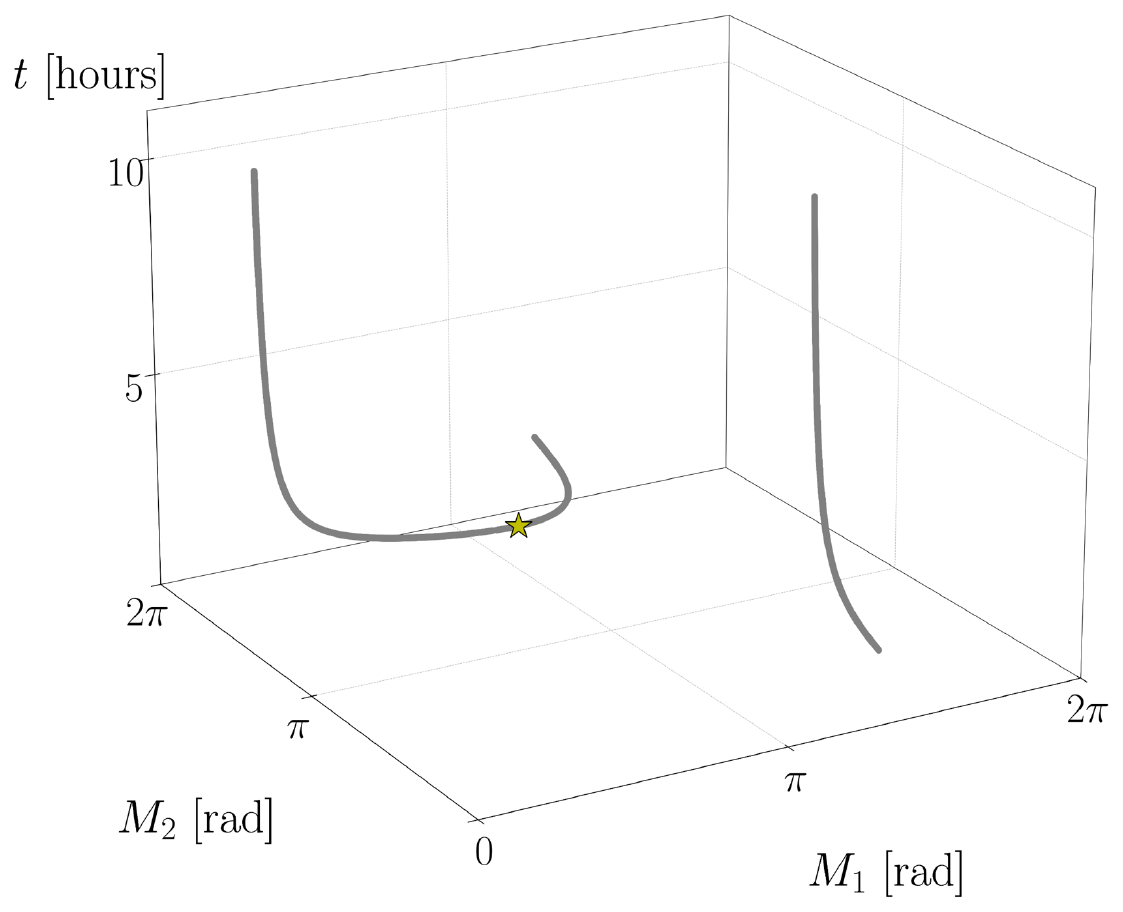}
        \caption{Root family 2. \taken{The local-minimum solutions satisfying PVT are displayed in blue, whereas the others are plotted in gray.}}
        \label{fig:orbit7_desirable_suboptimal}
    \end{subfigure}
    \begin{subfigure}[b]{0.48\textwidth}
        \centering
        \includegraphics[page=60, width=0.99\textwidth]{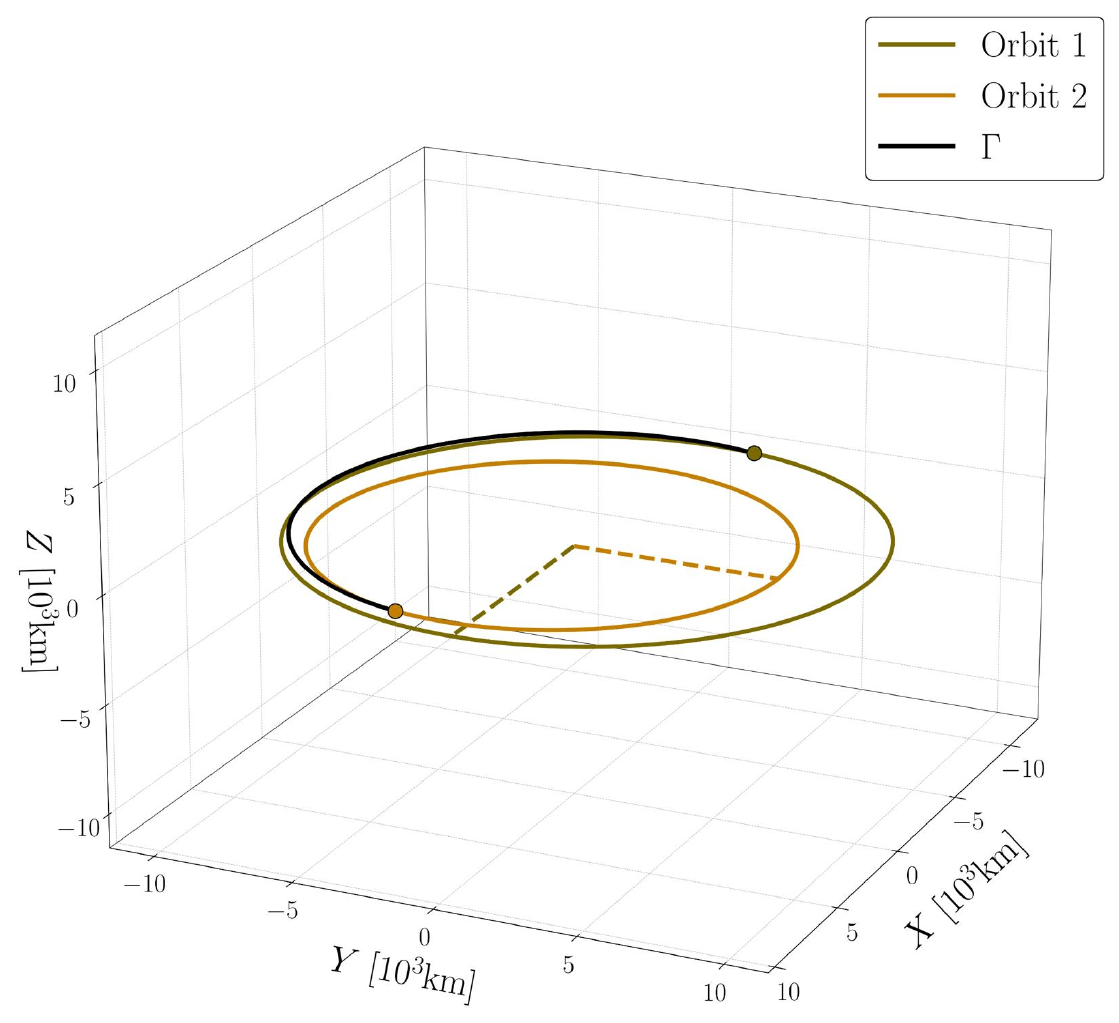}
        \caption{Min-$J$ \acrshort{tiot} along the family.}
        \label{fig:orbit7_desirable_suboptimal_transfer}
    \end{subfigure}
    \caption{Behaviors of ``root'' families in the coplanar transfer scenario.}
    \label{fig:orbit7_desirable}
\end{figure}

\subsubsection{Bifurcation between \texorpdfstring{$i_2 = 0^\circ$}{} and \texorpdfstring{$i_2 = 1^\circ$}{}}

A non-zero inclination fundamentally alters the admissible geometries and branch structure of the \acrshort{tiot} families. Figure \ref{fig:inc1} illustrates the evolution of the root families as $i_2$ increases from $0^\circ$ to $1^\circ$. The emergence of a singular line introduces a pathway linking to Root family 1. While this family forms a closed loop in the coplanar case, the inclined configuration leads to a merging behavior near the singular configuration $\pi_1$. The resulting branch no longer closes, but instead exhibits asymptotic behavior, with its two ends approaching $\tof \rightarrow \infty$ and $\tof \rightarrow 0$, respectively, while remaining close to $\pi_1$. This topological transition reflects a bifurcation in the family structure triggered by the introduction of a small but non-zero inclination. Additionally, the introduction of a non-zero inclination results in a noticeable shrinkage in the subset of family members that satisfy the \acrshort{pvt} conditions (compare Figs.~\ref{fig:orbit7_desirable_optimal} and \ref{fig:inc1_root1}). For Root family 2, the positive inclination enables the emergence of an optimal configuration that satisfies the \acrshort{pvt} condition, a behavior not observed in the coplanar case (see Fig. \ref{fig:orbit7_desirable_suboptimal}). Although this family evolves along a trajectory distinct from the coplanar case, its asymptotic behavior and the locations of its asymptotes remain nearly unchanged, indicating that the global structure of the family is preserved while its local optimality properties are modified.

\begin{figure}[htpb]
    \centering
    \begin{subfigure}[b]{0.48\textwidth}
        \centering
        \includegraphics[page=61, width=0.99\textwidth]{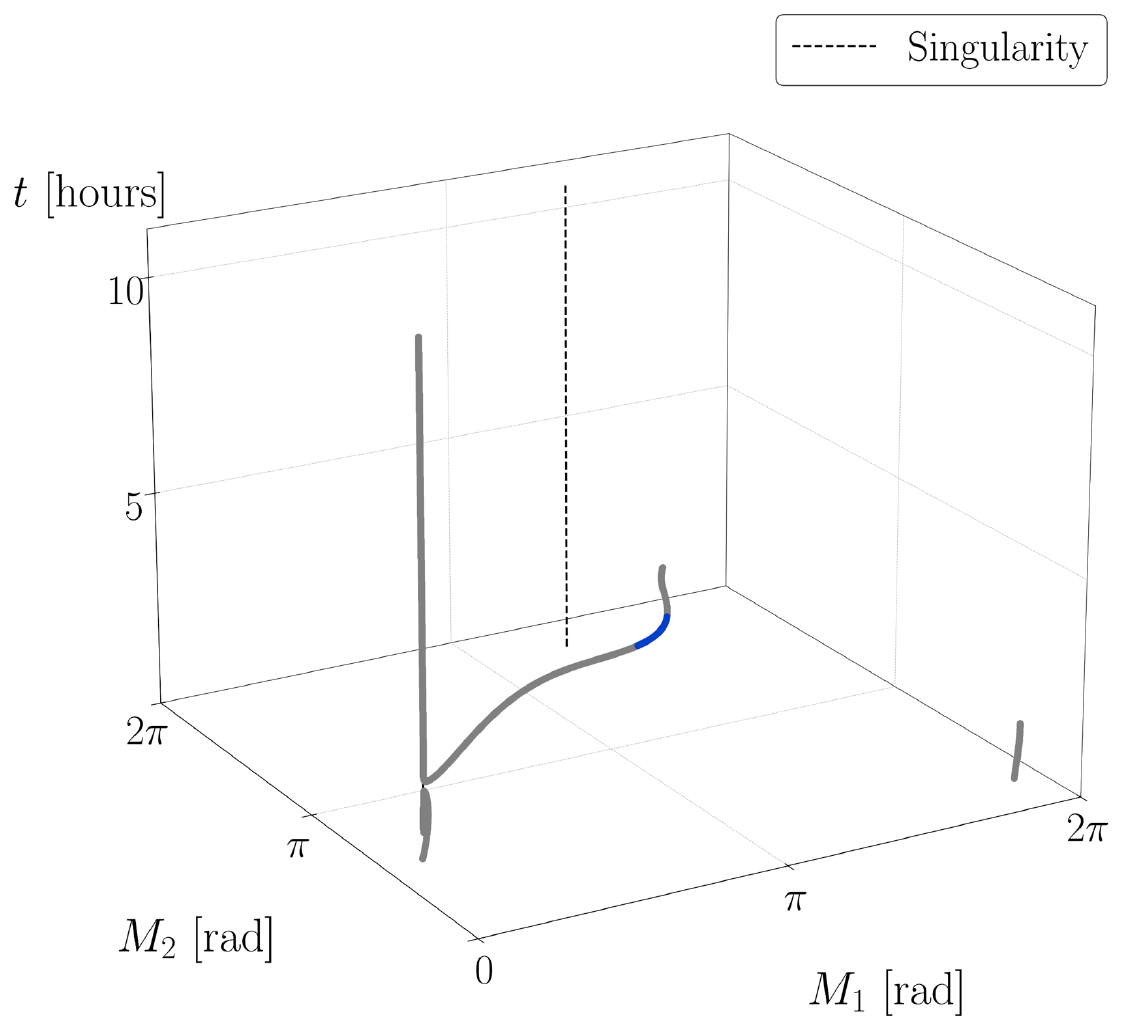}
        \caption{Root family 1.}
        \label{fig:inc1_root1}
    \end{subfigure}
    \begin{subfigure}[b]{0.48\textwidth}
        \centering
        \includegraphics[page=62, width=0.99\textwidth]{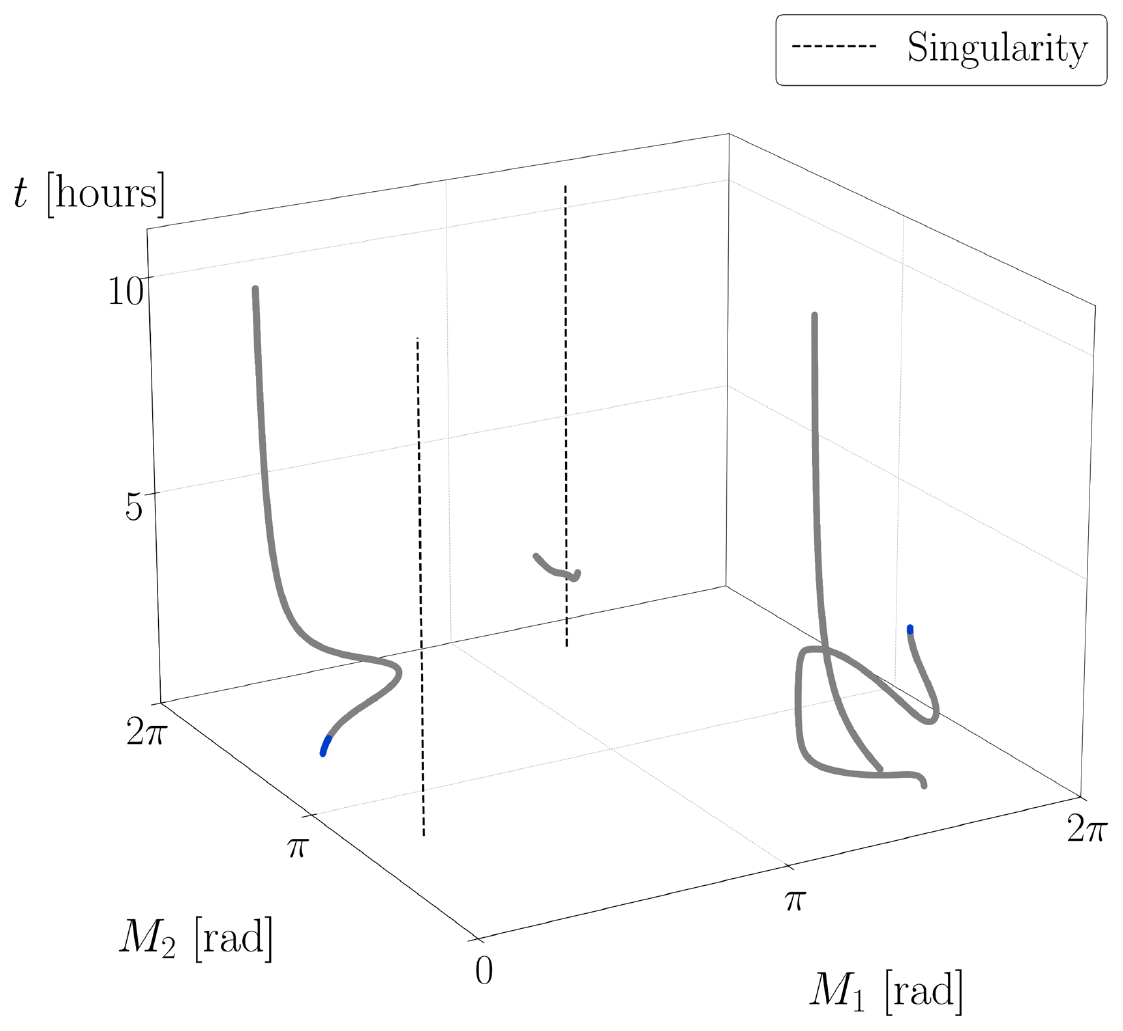}
        \caption{Root family 2.}
        \label{fig:inc1_root2}
    \end{subfigure}
    \caption{Structural bifurcation induced at $i_2 = 1^\circ$. \taken{The local-minimum solutions satisfying PVT are displayed in blue, whereas the others are plotted in gray.}}
    \label{fig:inc1}
\end{figure}

\subsubsection{Bifurcation between \texorpdfstring{$i_2 = 1^\circ$}{} and \texorpdfstring{$i_2 = 2^\circ$}{}}

The first root family undergoes a subsequent bifurcation between $i_2 = 1^\circ$ and $i_2 = 2^\circ$. At $i_2 = 1^\circ$, a nearby branch exists that originates from the singular line $\pi_2$ and traverses in the vicinity of Root family 1. As $i_2$ increases, a change in connectivity occurs; Root family 1 turns and connects to the singular line. After this reconfiguration, the two branches drift further apart as the inclination increases (Fig. \ref{fig:inc5_root2_bif}).

\begin{figure}[htpb]
    \centering
    \begin{subfigure}[b]{0.48\textwidth}
        \centering
        \includegraphics[page=63, width=0.99\textwidth]{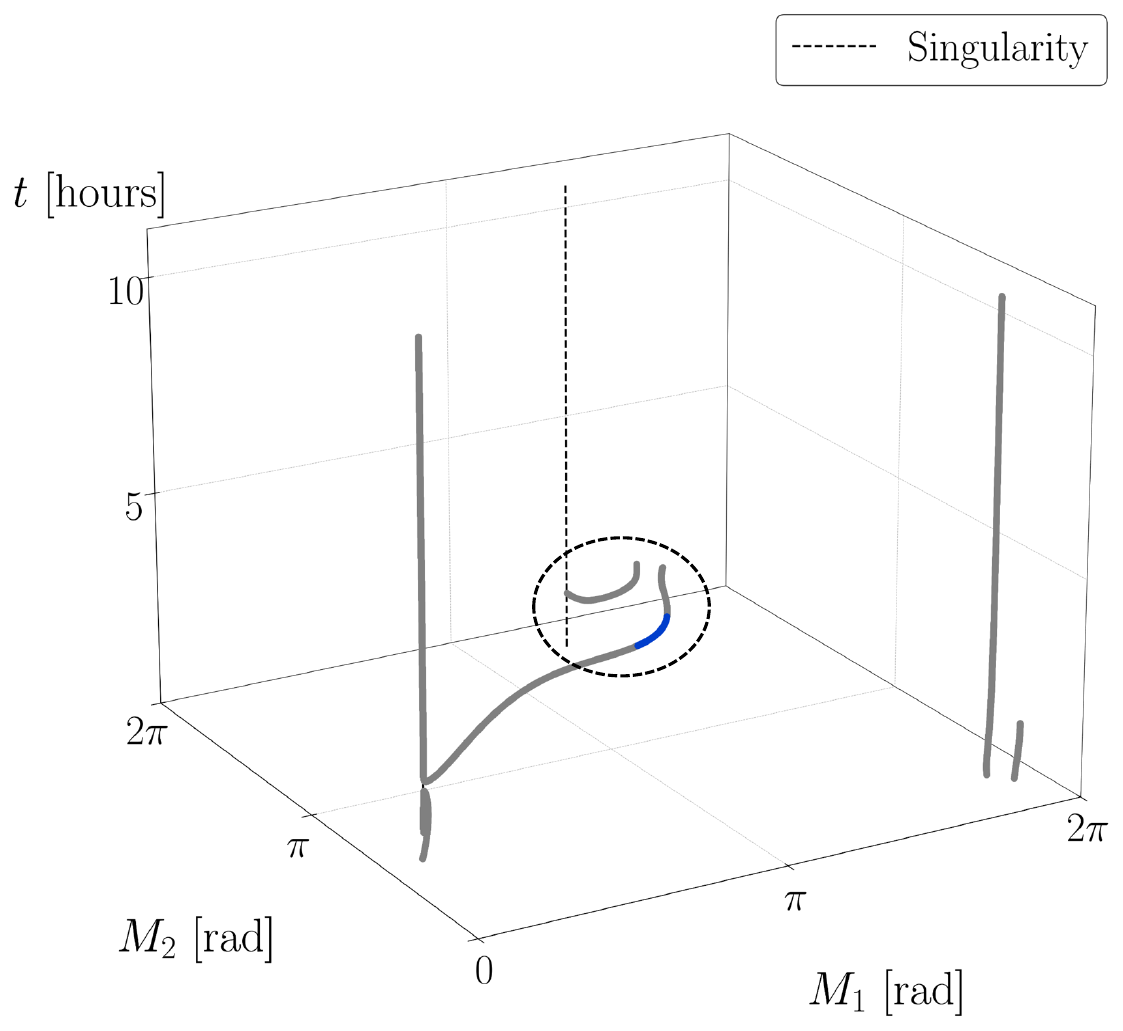}
        \caption{Root family 1 at $i_2 = 1^\circ$ and a nearby branch.}
        \label{fig:inc1_root1_bif}
    \end{subfigure}
    \begin{subfigure}[b]{0.48\textwidth}
        \centering
        \includegraphics[page=64, width=0.99\textwidth]{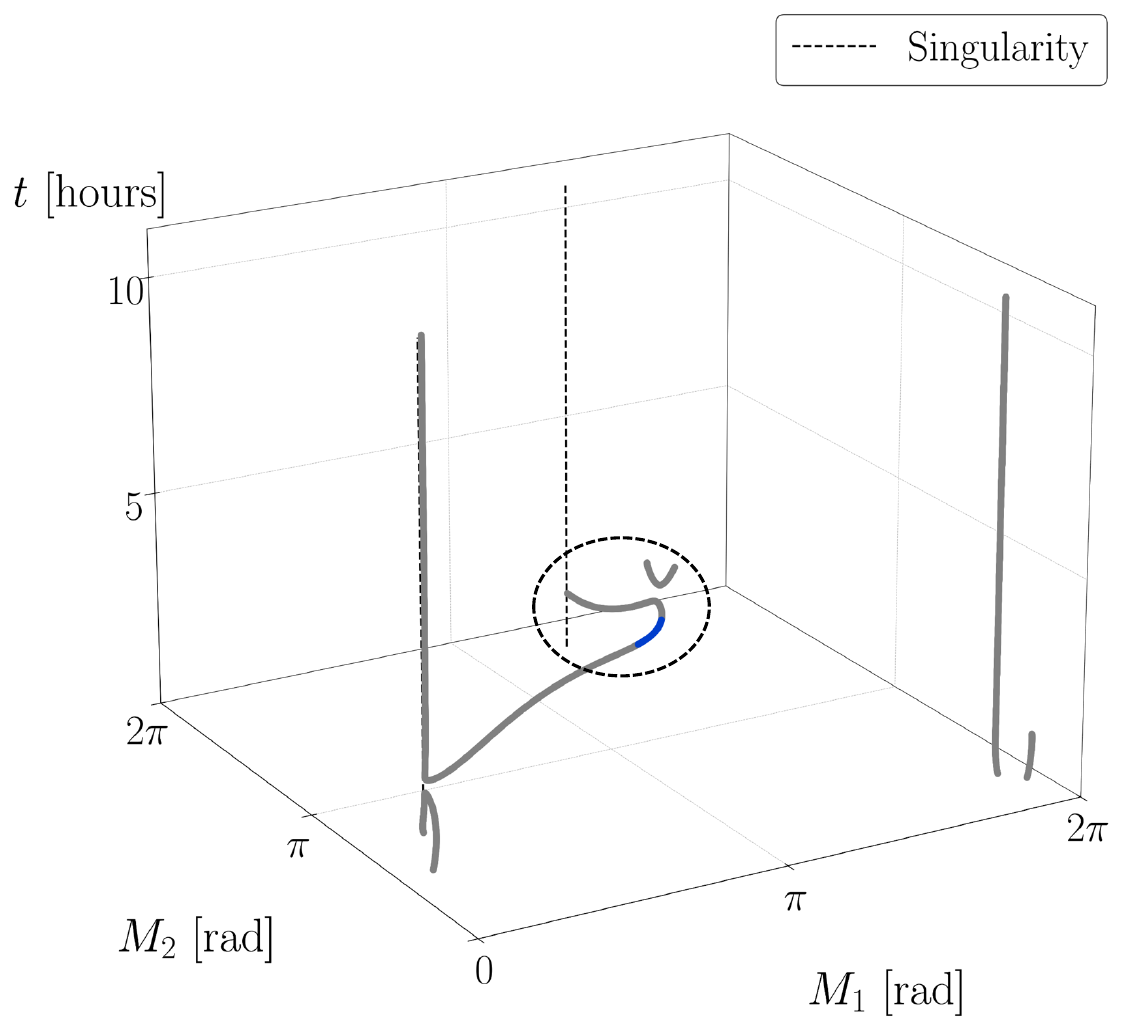}
        \caption{Root family 1 at $i_2 = 2^\circ$ and a reconfiguration.}
        \label{fig:inc2_root1_bif}
    \end{subfigure}
    \begin{subfigure}[b]{0.48\textwidth}
        \centering
        \includegraphics[page=65, width=0.99\textwidth]{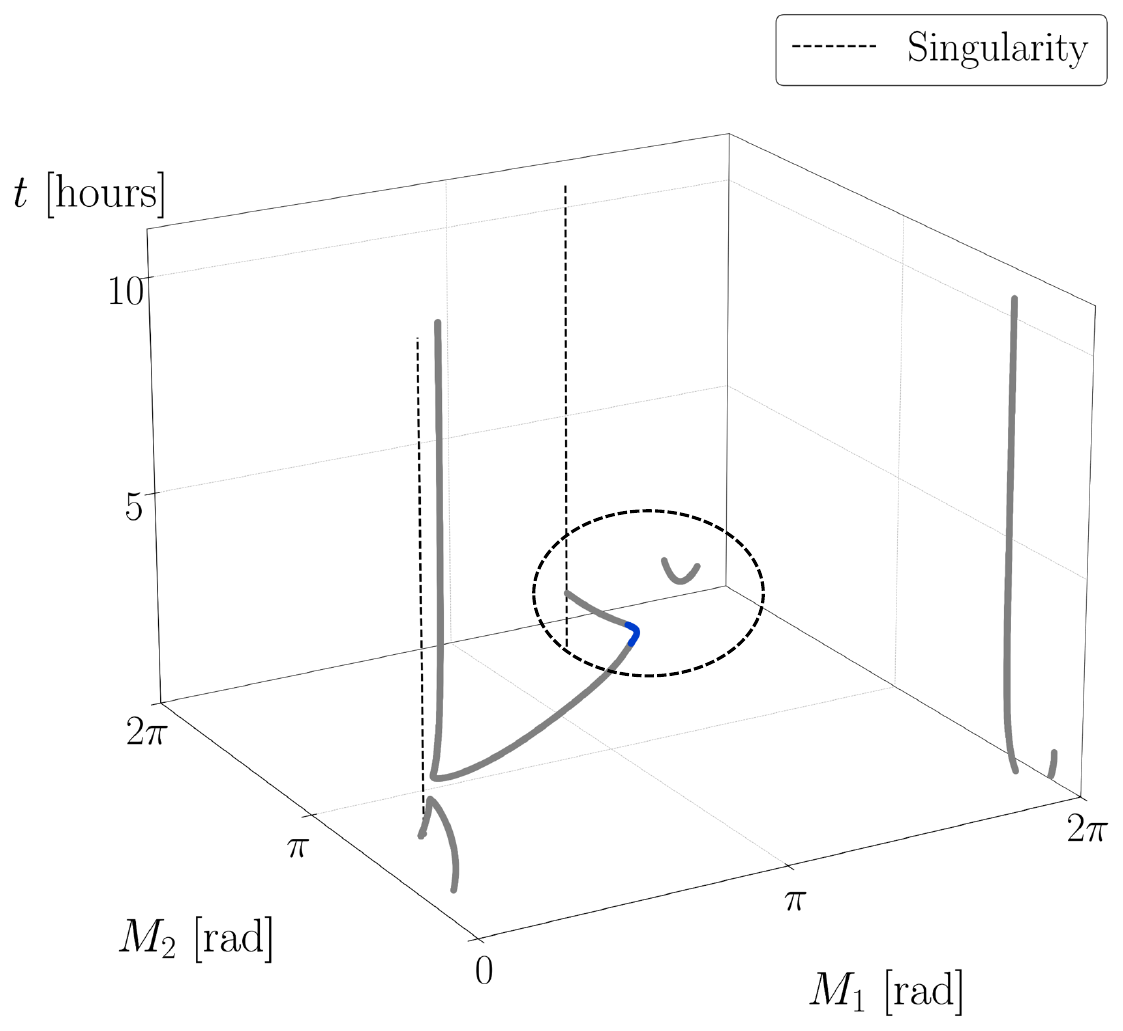}
        \caption{Root family 1 at $i_2 = 5^\circ$ with a progressive drift.}
        \label{fig:inc5_root2_bif}
    \end{subfigure}
    \caption{Bifurcation of Root family 1 between $i_2 = 1^\circ$ and $i_2 = 5^\circ$. \taken{The local-minimum solutions satisfying PVT are displayed in blue, whereas the others are plotted in gray.}}
    \label{fig:inc1_bif}
\end{figure}

\subsubsection{Bifurcation between \texorpdfstring{$i_2 = 8^\circ$}{} and \texorpdfstring{$9^\circ$}{}}

Another significant bifurcation occurs between $i_2 = 8^\circ$ and $9^\circ$. In this regime, a branch with endpoints at $\theta \rightarrow 180^\circ$ traverses close to the Root family 1 (Fig. \ref{fig:inc8_root1_bif}). Further increasing the inclination essentially splits the branches into new configurations, as illustrated in Fig. \ref{fig:inc9_root1_bif}. These two split branches subsequently formulate the basis for the desirable configurations observed in the baseline case ($i_2 = 30^\circ$), corresponding to min-$J$ \acrshort{tiots} 1 (Fig. \ref{fig:orbit1_desirable_1}) and 3 (Fig. \ref{fig:orbit1_desirable_3}), respectively. 

Thus, the origin of the three optimal configurations for the baseline scenario can be traced back as follows: two branches originate from the bifurcation of Root family 1 in the coplanar case, and one branch continues from Root family 2. It is worth emphasizing that the geometry associated with Root family 2 does not satisfy the \acrshort{pvt} criteria in the coplanar limit (Fig. \ref{fig:orbit7_desirable_suboptimal}) but evolves to satisfy the necessary conditions as the inclination on the arrival side increases.

\begin{figure}[htpb]
    \centering
    \begin{subfigure}[b]{0.48\textwidth}
        \centering
        \includegraphics[page=66, width=0.99\textwidth]{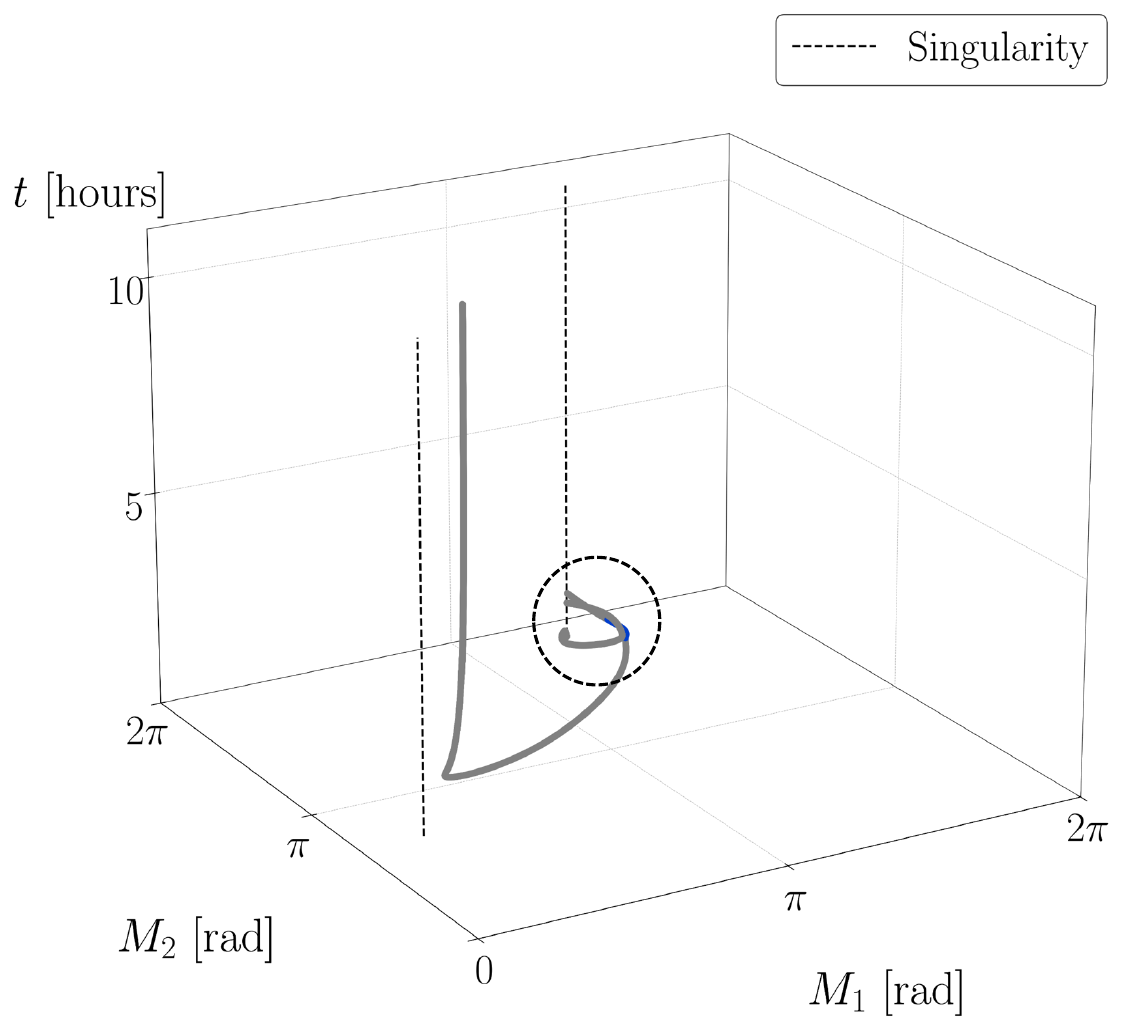}
        \caption{Root family 1 at $i_2 = 8^\circ$ and a nearby branch.}
        \label{fig:inc8_root1_bif}
    \end{subfigure}
    \begin{subfigure}[b]{0.48\textwidth}
        \centering
        \includegraphics[page=67, width=0.99\textwidth]{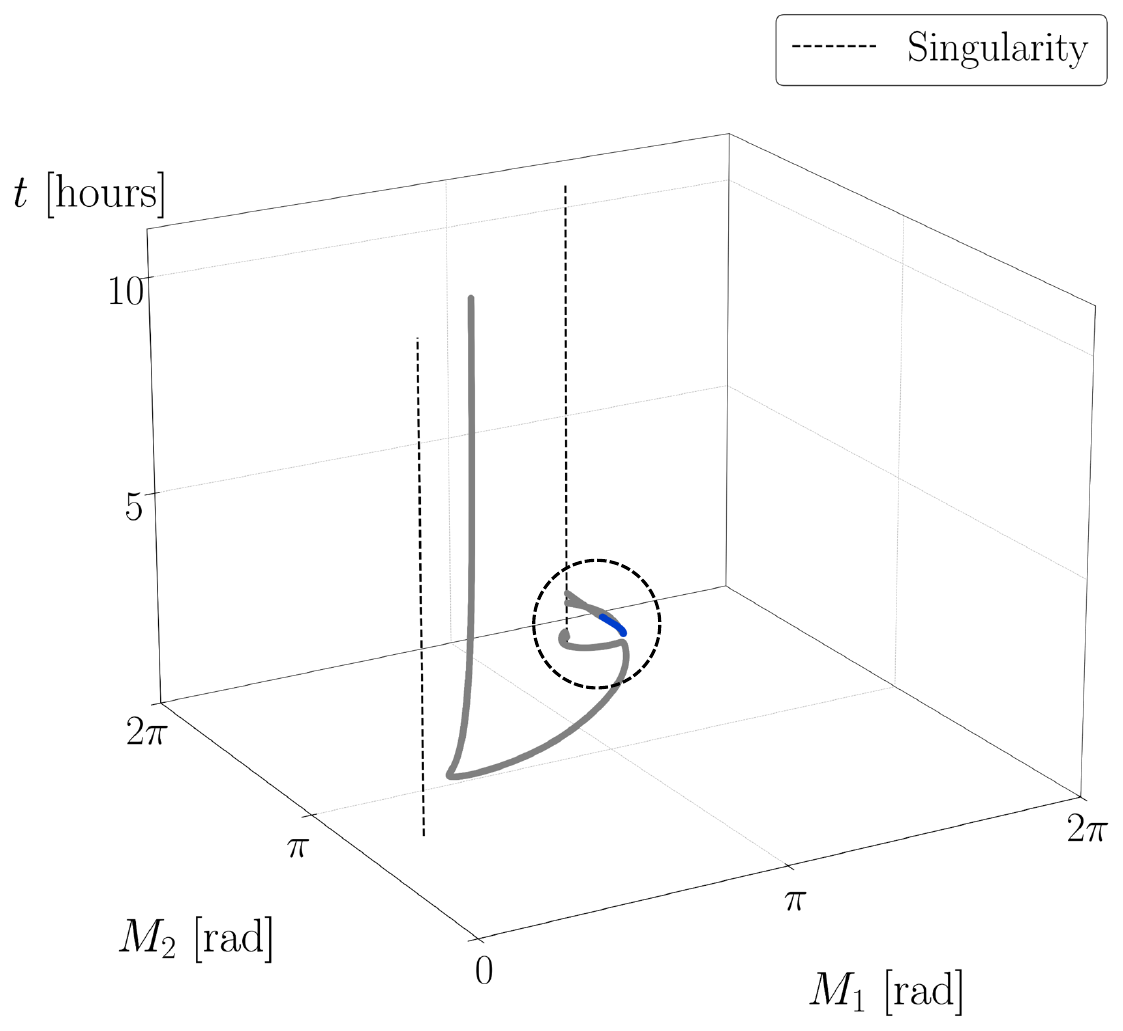}
        \caption{Root family 1 at $i_2 = 9^\circ$ and a reconfiguration.}
        \label{fig:inc9_root1_bif}
    \end{subfigure}
    \caption{Bifurcation of Root family 1 between $i_2 = 8^\circ$ and $i_2 = 9^\circ$. \taken{The local-minimum solutions satisfying PVT are displayed in blue, whereas the others are plotted in gray.}}
    \label{fig:root1_bif2}
\end{figure}

\subsection{Discussions}

The results of the present study suggest that viewing \acrshort{tiots} through a family-based lens provides insights that are both confirmatory and novel. This section organizes these findings along two axes: results that revisit and corroborate known optimal structures, and results that reveal new structural features challenging to access with conventional approaches. Directions for future research are then outlined.

\subsubsection{Revisiting Known Optimal Structures}

The family-based framework naturally recovers optimal transfer geometries that are consistent with prior analytical expectations. In particular, the first two min-$J$ \acrshort{tiots} identified in Section~\ref{sec:orbit1_optimal} (Figs.~\ref{fig:orbit1_desirable_1}--\ref{fig:orbit1_desirable_2}) correspond to near-nodal transfers, with departure and arrival positions lying close to the line of nodes connecting the two orbital planes. This geometry aligns with the ``modified Hohmann transfer'' and optimal nodal transfer concepts discussed by \citet{baker1966orbit} and \citet{vinh1988optimal}, confirming that the proposed framework is consistent with established heuristic strategies. However, whereas prior work identifies these configurations through geometric reasoning or case-specific optimization, the present analysis locates them as distinguished members of continuous families, thereby situating them within the broader optimality landscape. Furthermore, the long-horizon analysis in Section~\ref{sec:long_horizon} demonstrates that these near-nodal configurations recur frequently over extended planning horizons, quantitatively confirming their practical relevance for mission design.

\subsubsection{New Insights from the Family-Based Perspective}

Beyond recovering known structures, the family-based framework reveals features of the solution landscape that are not apparent from isolated optimization or grid-based searches. In particular, the emergence, disappearance, and multiplicity of solution basins observed in the parametric study are naturally interpreted as consequences of bifurcations, collisions, and reconnections of \acrshort{tiot} families, rather than as disconnected phenomena tied to specific initial conditions or grid resolutions. This viewpoint emphasizes that the number and structure of locally optimal solutions are global properties of the underlying dynamical and geometric configuration, and are therefore better understood by tracking continuous families than by enumerating discrete optima.

A notable example is the third min-$J$ \acrshort{tiot} (Fig.~\ref{fig:orbit1_desirable_3}), an isolated optimal branch that originates from a completely separate family and bears no obvious geometric connection to the nodal transfers. This solution would be difficult to anticipate from analytical heuristics and could easily be missed by grid searches that concentrate sampling near expected favorable geometries. Its discovery through family continuation underscores the value of a systematic, exhaustive exploration of the solution space.

From a practical standpoint, the family-based formulation captures the entire neighborhood of optimality, explicitly revealing how nearby configurations vary in both cost and geometry. The width of a family in the angular domain provides a natural measure of robustness (as quantified in Section~\ref{sec:long_horizon}); if a nominal launch or maneuver opportunity is missed, adjacent family members immediately suggest alternative near-optimal solutions, e.g., as illustrated in Fig.~\ref{fig:orbit1_desirable_planning}. Furthermore, since families are tracked continuously, the identification of the global minimum is less susceptible to convergence toward spurious local optima, a limitation of gradient-based or discretized search methods. In this sense, the proposed framework serves as a bridge between local analytical optimality conditions and the global numerical exploration required for comprehensive and reliable mission planning.

\subsubsection{\label{sec:conclusions_future_work}Future Work}

The present work primarily introduces and demonstrates the framework, and several directions for future research remain open. Classical analytically tractable scenarios, such as Hohmann or coaxial transfers, may be revisited from a family-based perspective to reinterpret known optimal solutions as limiting members of broader families. More importantly, the present results suggest that the creation, reconnection, or annihilation of optimal families may be governed by identifiable geometric or dynamical mechanisms, such as the appearance of additional parabolic asymptotes or singular configurations. A more rigorous treatment of the $\theta = 180^\circ$ singularity, as well as a deeper analysis of bifurcations arising in the coplanar transfer scenario, may enable semi-analytical predictions of family creation and reconnection. Such developments would strengthen the link between analytical optimality theory and global numerical exploration of orbital transfer problems.

\section{\label{sec:conclusions}Concluding Remarks} 

The classical fuel-optimal two-impulse rendezvous problem between Keplerian orbits is revisited from a family-based perspective. Rather than treating optimal transfers as isolated point solutions, the proposed framework analyzes continuous families of stationary transfers, supplying a global view of the optimality landscape connecting two Keplerian orbits. This shift in perspective enables systematic identification, classification, and continuation of optimal solutions across different parameter domains, each revealing complementary structural information. The framework naturally accommodates multiple analytical and numerical strategies, including Hessian-based classification, the Primer Vector Theory conditions, and direct juxtaposition with porkchop plots. By embedding conventional optimal solutions within higher-dimensional families, the approach clarifies the emergence, merging, and disappearance of local minima, saddle points, and disconnected solution basins as orbital parameters vary. From a practical standpoint, the family-based formulation complements existing theory and practice for fuel-optimal transfers by exposing the structure and robustness of nearby solutions, rather than isolating a single optimum. This global characterization is particularly valuable for mission design, where flexibility, sensitivity, and alternative near-optimal options are often as important as the nominal minimum-cost solution.


\backmatter

\begin{appendices}
\section{Gradient Information for the Continuation Scheme}
\label{app:gradient}

For a differential corrections process to successfully converge to $\fv^*$ that satisfies the constraints in Eq. \eqref{eq:c}, the Jacobian matrix $\partial \c / \partial \fv$ is required. Since the constraints themselves define the first-order optimality conditions (gradients of $J$ with respect to $\fv$), the required Jacobian involves the second-order derivatives of the cost function $J$ in terms of $\fv$.

\subsection{\label{sec:1st_deriv}First Order Derivatives: \texorpdfstring{$\nabla J $}{}}

The first order derivatives effectively constitute the optimality conditions in Eq. \eqref{eq:c}. They are evaluated as,
\begin{align}
    \nabla J = \frac{\partial J}{ \partial \fv } = \left[ \frac{\partial J }{\partial \ma{1}} , \frac{\partial J }{\partial \ma{2}} , \frac{\partial J }{\partial \tof} \right]^\transpose.
\end{align}
The derivatives with respect to $x_{1}$ and $x_{2}$ are trivially supplied via the chain rule that involves $\partial x_{1} / \partial \fv$ and $\partial x_{2} / \partial \fv$. The first two components require sensitivities of $J$ with respect to the departure and arrival locations that correlate to $\partial \vella{1}/ \partial \pos{1}$, $\partial \vella{1}/ \partial \pos{2}$, $\partial \vella{2}/ \partial \pos{1}$, and $\partial \vella{2}/ \partial \pos{2}$. There exist multiple options in the literature for supplying the sensitivities of $\vella{1}$ and $\vella{2}$ in terms of $\pos{1}$ and $\pos{2}$ within the context of the Lambert problem, leveraging the \acrfull{stm} \cite{schumacher2015uncertain, arora2015partial}, Lagrange's transfer-time formulation \cite{zhang2018covariance}, and vercosine formulation \cite{arora2015partial}. The algebraic complexity, robustness, and computational time may differ for each of these options. In the current analysis, the \acrshort{stm} information is leveraged. While the \acrshort{stm}-based Lambert arc sensitivities may be slower as opposed to other geometry-based methods, they offer an advantage in terms of implementation complexity; the \acrshort{stm} applies to both elliptic and hyperbolic Lambert arcs. In geometry-based algorithms, the transfer types must be determined a priori. Regardless, any strategies may be leveraged to provide the first-order sensitivities. \acrshort{stm} captures the linear variational information along a given Lambert arc $\la$ as,
\begin{align}
    \mb \delta \pos{2} \\ \delta \vella{2}  \me = \mb \bm{\phi}_{\bm{rr}} & \bm{\phi}_{\bm{rv}} \\ \bm{\phi}_{\bm{vr}} & \bm{\phi}_{\bm{vv}} \me \mb \delta \pos{1} \\ \delta \vella{1} \me,
\end{align}
where $\bm{\phi}$ is the \acrshort{stm} evaluated from $\t$ to $\t + \tof$ within the Keplerian dynamics. The subscripts denote the subcomponents of the \acrshort{stm}. For example, $\bm{rv}$ corresponds to the change in the position components upon arrival ($\delta \pos{2}$) with respect to the initial velocity variation ($\delta \vella{1}$). Then, following the process from \citet{schumacher2015uncertain}, the velocity variations are solved in terms of the position variations as \cite{schumacher2015uncertain, arora2015partial},
\begin{align}
    \mb \delta \vella{1} \\ \delta \vella{2} \me = \mb -\bm{\phi}_{\bm{rv}}^{-1} \bm{\phi}_{\bm{rr}} & \bm{\phi}_{\bm{rv}}^{-1} \\ \bm{\phi}_{\bm{vr}} - \bm{\phi}_{\bm{vv}} \bm{\phi}_{\bm{rv}}^{-1} \bm{\phi}_{\bm{rr}} & \bm{\phi}_{\bm{vv}} \bm{\phi}_{\bm{rv}}^{-1} \me \mb \delta \pos{1} \\ \delta \pos{2}\me,
\end{align}
supplying the variation of $\vella{1,2}$ with respect to $\pos{1,2}$. With this information, the first order derivatives of $J$ with respect to $\ma{1,2}$ are retrieved as,
\begin{align}
    \frac{\partial J}{\partial \ma{1}} & = \frac{\dv{1}^\transpose}{\norm{\dv{1}}} \left( \frac{d\vel{1}}{d\ma{1}} - \frac{\partial \vella{1}}{\partial  \pos{1}}\frac{d \pos{1}}{d \ma{1}}\right) + \frac{\dv{2}^\transpose}{\norm{\dv{2}}} \left(- \frac{\partial \vella{2}}{\partial \pos{1}} \frac{d \pos{1}}{d \ma{1}} \right), \\
    \frac{\partial J}{\partial \ma{2}} & = \frac{\dv{1}^\transpose}{\norm{\dv{1}}} \left(- \frac{\partial \vella{1}}{\partial \pos{2}}\frac{d \pos{2}}{d \ma{2}} \right) + \frac{\dv{2}^\transpose}{\norm{\dv{2}}} \left(\frac{d\vel{2}}{d \ma{2}} - \frac{\partial \vella{2}}{\partial \pos{2}}\frac{d \pos{2}}{d \ma{2}} \right).
\end{align}
Here, the derivatives of $\bm{r}$ and $\bm{v}$ with respect to $M$ are evaluated along the departure and arrival ellipses. As $\bm{r}, \bm{v}$ are defined within the respective (invariant) Keplerian orbits, they solely depend on their respective mean anomalies; e.g., $\pos{2}$ does not depend on $\ma{1}$. In the following, the departure side is examined without losing generality. First, Kepler's equation is solved to supply the eccentric anomaly,
\begin{align}
    E_1 = E_1(\ma{1}; e_1),
\end{align}
with $\frac{dE_1}{dM_1} = \frac{1}{1-e_1\cos E_1}$. Within the perifocal frame ($pqw$), position and velocity vectors are represented as,
\begin{align}
    \pos{1}^{pqw} = a_1\mb \cos E_1 -e_1 & \sqrt{1-e_1^2}\sin E_1 & 0 \me^\transpose, \\
    \vel{1}^{pqw} = \frac{\sqrt{\mu a_1}}{\norm{\pos{1}}} \mb -\sin E_1 & \sqrt{1-e_1^2}\cos E_1 & 0 \me^\transpose.
\end{align}
Derivatives with respect to $\ma{1}$ are then,
\begin{align}
    \label{eq:dr_pqw_M}\frac{d\pos{1}^{pqw}}{d\ma{1}} &= a_1 \frac{dE_1}{d\ma{1}}     \mb -\sin E_1 & \sqrt{1-e_1^2}\cos E_1 & 0 \me^\transpose, \\
    \frac{d\vel{1}^{pqw}}{d\ma{1}}  &= \frac{dE_1}{d\ma{1}}    \Bigg[
        -a_1\!\left(\frac{\sqrt{\mu a_1}\, e_1 \sin E_1}{\norm{\pos{1}}^2}\right)
        \mb -\sin E_1 & \sqrt{1-e_1^2}\cos E_1 & 0 \me^\transpose\nonumber
        \\
        &\ + \frac{\sqrt{\mu a_1}}{\norm{\pos{1}}}
        \mb -\cos E_1 & -\sqrt{1-e_1^2}\sin E_1 & 0 \me^\transpose
    \Bigg].\label{eq:dv_pqw_M}
\end{align}
Vectors in the perifocal frame and the inertial frame are related via $\pos{1} = \bm{C}_1\pos{1}^{pqw}$. The direction cosine matrix is supplied with the Euler 3-1-3 rotation sequence as,
\begin{align}
    \bm{C}_1 = 
    \mb
        \cos\Omega_1\cos\omega_1 - \sin\Omega_1\sin\omega_1\cos i_1
        & -\cos\Omega_1\sin\omega_1 - \sin\Omega_1\cos\omega_1\cos i_1
        & \sin\Omega_1\sin i_1 \\[2pt]
        \sin\Omega_1\cos\omega_1 + \cos\Omega_1\sin\omega_1\cos i_1
        & -\sin\Omega_1\sin\omega_1 + \cos\Omega_1\cos\omega_1\cos i_1
        & -\cos\Omega_1\sin i_1 \\[2pt]
        \sin\omega_1\sin i_1
        & \cos\omega_1\sin i_1
        & \cos i_1
    \me.
\end{align}
Then, the desired sensitivities are supplied as,
\begin{align}
    \frac{d\pos{1}}{d\ma{1}} = \bm{C}_1 \frac{d\pos{1}^{pqw}}{d\ma{1}}, \quad
    \frac{d\vel{1}}{d\ma{1}} = \bm{C}_1 \frac{d\vel{1}^{pqw}}{d\ma{1}}.
\end{align}
With these sensitivities in hand, the derivative of $J$ with respect to the time of flight, $\tof$, is addressed. This component is more nuanced; specifically, $\partial J /\partial \tof = \delta J / \delta \tof + \partial J/ \partial\ma{2} \cdot d\ma{2} / d \tof$. The first term ($\delta J / \delta \tof$) captures the explicit dependency on $\tof$ for a given Lambert arc, and the second term ($\partial J/ \partial\ma{2} \cdot d\ma{2} / d \tof$) corresponds to the indirect impact where $\tof$ changes the arrival location. In the time-variable formulation, the variation equation along the Lambert arc results in,
\begin{align}
    \mb \delta \pos{2} \\ \delta \vella{2}  \me = \mb \bm{\phi}_{\bm{rr}} & \bm{\phi}_{\bm{rv}} &  \frac{d \pos{2}}{d \tau}\\ \bm{\phi}_{\bm{vr}} & \bm{\phi}_{\bm{vv}} & \frac{d \vella{2}}{d \tau} \me \mb \delta \pos{1} \\ \delta \vella{1}  \\ \delta \tof \me.
\end{align}
Recall that $\tau$ is a generic time variable. In evaluating $\delta J / \delta \tof$, variations in $\pos{1,2}$ do not occur, i.e., $\delta \pos{1,2} = \bm{0}$. Thus,
\begin{align}
    \label{eq:dv1_dtof}\frac{\delta \vella{1}}{\delta \tof} & = -\bm{\phi}_{\bm{rv}}^{-1} \frac{d \pos{2}}{d\tau} = -\bm{\phi}_{\bm{rv}}^{-1}\vella{2}, \\
    \label{eq:dv2_dtof} \frac{\delta \vella{2}}{\delta \tof} & = -\bm{\phi}_{\bm{vv}} \bm{\phi}_{\bm{vr}}^{-1} \frac{d \pos{2}}{d\tau } + \frac{d \vella{2}}{d\tau},
\end{align}
from a straightforward extension of \citet{arora2015partial}\footnote{While in \citet{arora2015partial} the departure and arrival epochs are separate variables and sensitivities to these epochs are separately supplied, the impact of epochs is considered through $\ma{1,2}$ in the current analysis and only the sensitivity to $\tof$ is required, resulting in rather simple expressions as in Eqs. \eqref{eq:dv1_dtof}-\eqref{eq:dv2_dtof}.}. The explicit sensitivity of $J$ with respect to $\tof$ is then,
\begin{align}
    \label{eq:dJ_deltof}\frac{\delta J}{\delta \tof} = - \frac{\dv{1}^\transpose}{\norm{\dv{1}}}\frac{\delta \vella{1}}{\delta\tof} - \frac{\dv{2}^\transpose}{\norm{\dv{2}}}\frac{\delta\vella{2}}{\delta \tof}.
\end{align}
As units of $\partial J/\partial M$ are different from $\delta J/ \delta \tof$ (or $\partial J /\partial \tof$), a proper scaling for $\delta J/ \delta \tof$ reduces potential numerical instabilities. In the current analysis, $\tof$ is divided by $1000$ seconds to supply an arbitrary nondimensional time unit; such quantities may be redefined according to the periods of the departure and arrival ellipses.

\subsection{\label{sec:hessian}Second Order Derivatives: \texorpdfstring{$\nabla^2 J$}{}}

Second-order derivatives of $J$ serve two purposes: (1) classifying stationary solutions as minima, maxima, or saddle points via the Hessian eigenvalues (Eq.~\eqref{eq:hessian}), and (2) constructing the Jacobian $\partial \c / \partial \fv$ required by the differential corrections process to enforce Eq.~\eqref{eq:c}. The latter is evaluated as,
\begin{align}
    \label{eq:dc_dfv} \frac{\partial \c}{\partial \fv} = \mb \frac{\partial ^2J}{\partial x_1 \partial \ma{1}} & \frac{\partial ^2J}{\partial x_1 \partial \ma{2}}  & \frac{\partial ^2J}{\partial x_1 \partial \tof} \\ \frac{\partial ^2J}{\partial x_2 \partial \ma{1}} & \frac{\partial ^2J}{\partial x_2 \partial \ma{2}}  & \frac{\partial ^2J}{\partial x_2 \partial \tof}  \me.
\end{align}
Multiple strategies exist to (semi-)analytically supply the Hessian information. Examples include the state transition tensors \cite{park2007nonlinear}, differential algebra \cite{lizia2008application, shu2022higher}, and geometry-based methods, e.g., vercosine formulation as disclosed in the appended Fortran code from \citet{russell2019solution}. While these (semi-)analytical second-order derivatives may be considered more accurate, it is true that they entail additional implementation complexities. In the current analysis, such information is numerically delivered with the central finite differencing technique, perturbing with a small number $h = 1\cdot 10^{-7}$. For example,
\begin{align}
    \frac{\partial ^2 J}{\partial  \ma{1} \partial \ma{2}}\bigg \rvert_{(\ma{1}^*, \ma{2}^*, \tof^*)} \approx \left( \frac{\partial J}{\partial  \ma{1}} \bigg \rvert_{(\ma{1}^*, \ma{2}^*+h, \tof^*)}  - \frac{\partial J}{\partial \ma{1}}\bigg \rvert_{(\ma{1}^*, \ma{2}^*-h, \tof^*)}  \right) / (2h).
\end{align}
In the following numerical examples, this differencing technique is sufficient to deliver the required second-order derivatives in Eqs. \eqref{eq:hessian}-\eqref{eq:dc_dfv}. Thus, exploring alternatives for delivering the second-order derivatives remains out of the scope for the current analysis.

\section{Asymptotic Behavior at \texorpdfstring{$\tof \rightarrow 0$}{} in the Angular Domain}
\label{app:t0}

The asymptotic case at $\tof\rightarrow0$ implies an instantaneous change in position between two orbits at $\ma{1}, \ma{2}$ with $J \rightarrow \infty$ with $|\vella{1}|, |\vella{2}|$ approaching $\infty$. In such a limiting case, the following relationship holds,
\begin{align}
    J \approx |\vella{1}| + |\vella{2}| \approx 2\frac{\dist}{\tof},
\end{align}
where $\dist$ is the distance traveled along the Lambert arc as $\tof \rightarrow 0$. Two possible scenarios emerge,
\begin{align}
    \dist & = \norm{\pos{1}(\ma{1}) - \pos{2}(\ma{2})}, \quad \text{The solution does not go through the central body (short transfer)}, \\
    &  =  \norm{\pos{1}(\ma{1})} + \norm{\pos{2}(\ma{2})}, \quad \text{The solution goes through the central body (long transfer)}.
\end{align}
In the former, the Lambert arc consists of one (nearly-)straight line. In the latter, the arc results in two (nearly-)straight lines that instantaneously change direction with the intersection with the central body\footnote{At its limiting case, $\tof \rightarrow 0$ forces the trajectory to approach the central body (Earth) in a nearly rectilinear fashion; then, the central body pulls the trajectory with a turn angle that approximates an impulsive burn of infinite magnitude at the center of the body; obviously, the collision at the surface is ignored for a theoretical construction of such transfers.}. Thus, the optimality constraint from Eq. \eqref{eq:angular_optimality} interchanges with,
\begin{align}
    \label{eq:angular_optimality_dist} \left[ \frac{d \dist}{d\ma{1}}, \frac{d  \dist}{d\ma{2}}\right]^\intercal = \bm{0}, 
\end{align}
i.e., combinations of $\ma{1}, \ma{2}$ that result in stationary distance traveled are searched. For illustration, consider Fig. \ref{fig:t0_short}. Figure \ref{fig:t0_short_m1m2.png} takes a sample value of $\tof = 10$ seconds, a very small value to approximate the $\tof \rightarrow 0$ behavior. A contour plot of $J$ is generated for each combination of $\ma{1}, \ma{2}$, assuming short transfers. While all Lambert arcs are associated with impractical $J$ values, it is obvious that variations in $J$ are still present. The contour plot in Fig. \ref{fig:t0_short_m1m2_dist.png} depicts the distance $\dist =\norm{\pos{1}-\pos{2}}$ connecting two locations; two contour plots in Figs. \ref{fig:t0_short_m1m2.png} and \ref{fig:t0_short_m1m2_dist.png} demonstrate nearly identical behaviors. The long transfers are examined in Fig. \ref{fig:t0_long}. Figure \ref{fig:t0_long_m1m2.png} demonstrates a contour plot of $J$ at $\tof = 10$ seconds. It is clear that the function is an independent combination of $\ma{1}$ and $\ma{2}$; at the center of the plot, $\ma{1} = \ma{2} = \pi$, the Lambert arc connects the apogees of two orbits. Thus, maximal distance is traveled and results in the largest cost $J$. Similarly, $\ma{1} = \ma{2} = 0$ results in minimally traveled distance and the lowest cost $J$. The Lambert arc corresponding to the maximal cost is illustrated in Fig. \ref{fig:t0_long_lambert_arc.png} within the Earth-centered inertial frame. The \taken{dashed} lines denote the direction of the eccentricity vectors of the two orbits. Clearly, the Lambert arc links apogees of the orbits. In denoting these asymptotic cases of the optimal Lambert families at $\tof \rightarrow 0$, following notation is introduced, 
\begin{align}
    0^{a}_{b,c}.
\end{align}
Identical to Eq. \eqref{eq:asymptote_notation}, $\mathrm{a} = \mathrm{s}, \mathrm{l}$ denote the short and long transfer scenarios. The Hessian information is classified by $b$ with $m, M, \sigma$ denoting the local minimum, maximum, and saddle, respectively. Finally, $c \geq 1$ is the index to enumerate multiple cases under the same category. The asymptotic locations for the baseline scenario in Table \ref{table:sample_orbits_first} are included in Table \ref{table:t0}. 

\begin{table}[htpb]
\centering
\caption{\label{table:t0}Approximate asymptotic stationary points ($\tof \rightarrow 0$) in the angular domain for the baseline scenario (Table \ref{table:sample_orbits_first})}
\begin{tabular}{l c c }
\toprule
Label & $\ma{1}$ [rad] & $\ma{2}$ [rad] \\\midrule
$0^\mathrm{s}_{M,1}$ & $3.02$  & $4.71$  \\
$0^\mathrm{s}_{\sigma,1}$ & $5.65$ & $0.75$ \\
$0^\mathrm{s}_{\sigma,2}$ & $2.32$ & $0.69$ \\
$0^\mathrm{s}_{\sigma,3}$ & $5.65$ & $3.96$ \\
$0^\mathrm{s}_{m,1}$ & $0.57$  & $5.53$ \\
$0^\mathrm{s}_{m,2}$ & $4.21$ & $2.26$ \\
$0^\mathrm{l}_{M,1}$ & $\pi$  & $\pi$ \\ 
$0^\mathrm{l}_{\sigma,1}$ & $\pi$ & $0$  \\
$0^\mathrm{l}_{\sigma,2}$ & $0$ & $\pi$ \\
$0^\mathrm{l}_{m,1}$ & $0$ & $0$ \\\bottomrule
\end{tabular}
\end{table}

Two additional interesting insights exist for the $\tof \rightarrow 0$ asymptotic case. First, it serves as asymptotes for the reverse order transfers as well. Moreover, as the velocity along the transfer arc approaches infinity, the asymptotic behaviors ignore all the gravitational influences; thus, the asymptotic behaviors likely extrapolate to more complex dynamical environment, e.g., the Circular Restricted Three-Body Problem (CR3BP). A family of \acrshort{tiots} in the angular domain may be initiated whenever the distance between departure and arrival (periodic) orbits is stationary with respect to the nearby variations. 

 \begin{figure}[h!]
    \centering
    \begin{subfigure}[b]{0.48\textwidth}
        \centering
        \includegraphics[page=68, width=0.99\textwidth]{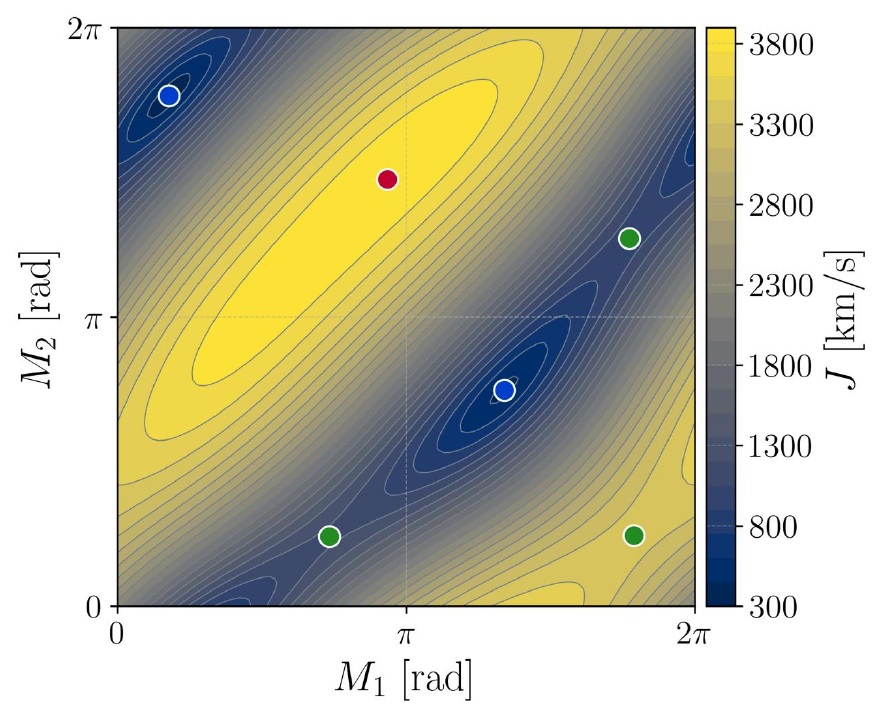}
        \caption{Contour plot of $J$ at $\tof = 10$ seconds.}
        \label{fig:t0_short_m1m2.png}
    \end{subfigure}
    \begin{subfigure}[b]{0.48\textwidth}
        \centering
        \includegraphics[page=69, width=0.99\textwidth]{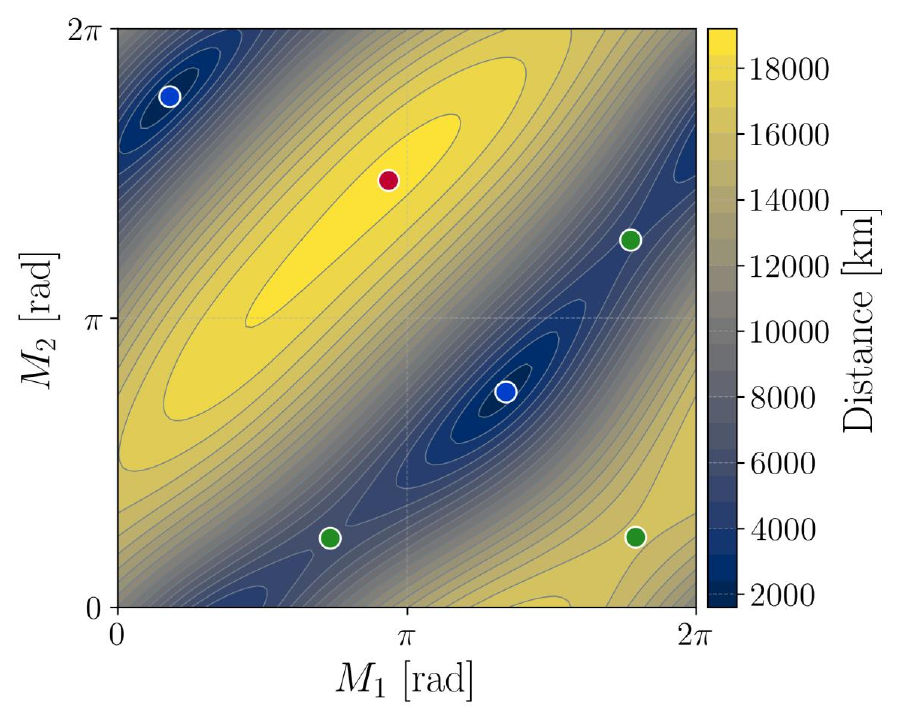}
        \caption{Contour plot of $\dist= \norm{\pos{1}-\pos{2}}$.}
        \label{fig:t0_short_m1m2_dist.png}
    \end{subfigure}
  \caption{Relationship between cost ($J$) and distance ($\dist$) at $\tof \rightarrow 0$ for $d_1 = \mathrm{s}$ \taken{(short transfer)}, with colored dots indicating Hessian information (Table \ref{table:colormap}).}
  \label{fig:t0_short}
\end{figure}

\begin{figure}[h!]
    \centering
    \begin{subfigure}[b]{0.48\textwidth}
        \centering
        \includegraphics[page=70, width=0.99\textwidth]{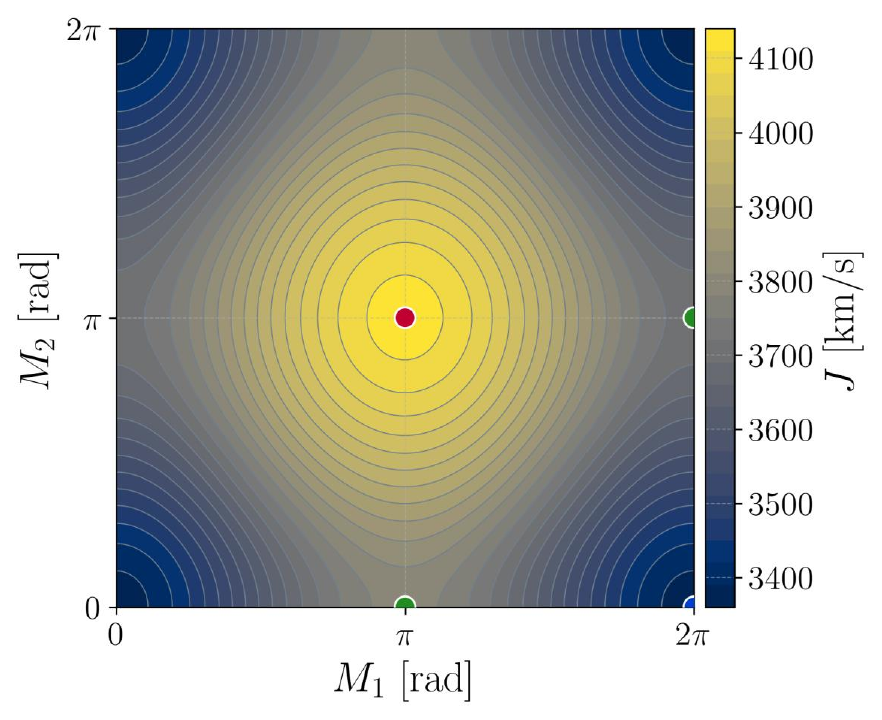}
        \caption{Contour plot of $J$ at $\tof = 10$ seconds.}
        \label{fig:t0_long_m1m2.png}
    \end{subfigure}
    \begin{subfigure}[b]{0.48\textwidth}
        \centering
        \includegraphics[page=71, width=0.99\textwidth]{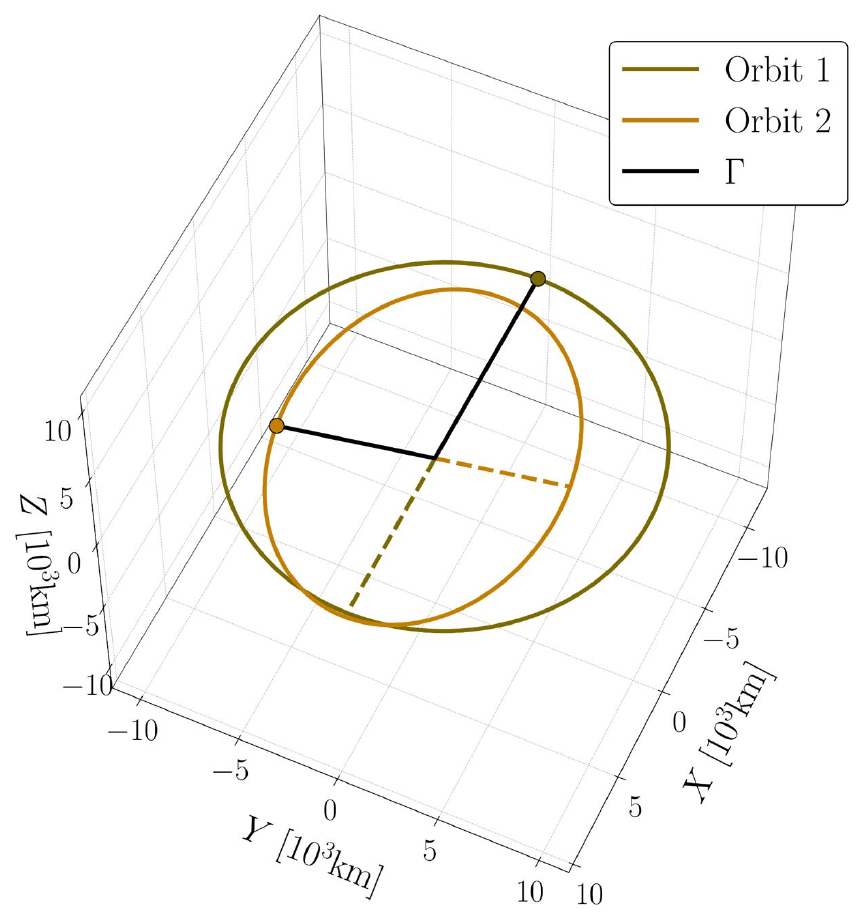}
        \caption{A representative trajectory near $0^\mathrm{l}_{M,1}$.}
        \label{fig:t0_long_lambert_arc.png}
    \end{subfigure}
  \caption{Asymptotic characteristics at $\tof \rightarrow 0$ ($d_1 = \mathrm{l}$\taken{, long transfer}), with colored dots indicating Hessian information (Table \ref{table:colormap}).}
  \label{fig:t0_long}
\end{figure}

\newpage
\section{Parabolic Lambert Arcs}
\label{app:parabola}

Along a parabolic arc connecting the departure and arrival orbits, the conic equation is given by
\begin{align}
    r_i = \frac{\mathsf{p}}{1 + \cos f_i}, \quad \text{for } i = 1,2,
\end{align}
where $r_i$ denotes the radial distance from the central body at the departure and arrival points, respectively. The quantities $\mathsf{p}$ and $f_i$ represent the semi-latus rectum and the true anomaly at each endpoint. Expanding this relation yields
\begin{align}
    \label{eq:parabola_eq1}
    r_1 \left( 1 + \cos f_1 \right)
    = r_2 \left( 1 + \cos f_2 \right)
    = r_2 \left( 1 + \cos ( f_1 + \theta ) \right),
    \quad \theta = \theta^{\mathrm{s}}, \theta^{\mathrm{l}},
\end{align}
where the transfer angle $\theta$ is defined as
\begin{align}
    \theta^{\mathrm{s}} = \arccos \left( \frac{\pos{1}^\top \pos{2}}{r_1 r_2} \right), \quad
    \theta^{\mathrm{l}} = 2\pi - \theta^{\mathrm{s}},
\end{align}
corresponding to the short and long transfers, respectively, i.e., $d_1 = \mathrm{s}, \mathrm{l}$. Rearranging Eq.~\eqref{eq:parabola_eq1} renders
\begin{align}
    \left( r_1 - r_2 \cos \theta \right) \cos f_1
    + r_2 \sin \theta \sin f_1
    = r_2 - r_1,
\end{align}
which may be expressed in a more compact form as
\begin{align}\label{eq:parabola_eq2}
    \cos \left( f_1 - \delta \right)
    = \frac{C}{\sqrt{A^2 + B^2}},
\end{align}
with
\begin{align}
    A = r_1 - r_2 \cos \theta, \quad
    B = r_2 \sin \theta, \quad
    C = r_2 - r_1, \quad
    \delta = \arctan_2 ( B, A ).
\end{align}
Here, $\arctan_2(\cdot)$ denotes the four-quadrant inverse tangent. As discussed in Section \ref{sec:domain}, parabolic Lambert arcs are employed to examine stationary transfer geometries in terms of the cost function $J$ at $\tof\to\infty$. For a given combination of departure and arrival positions, $\pos{1}$ and $\pos{2}$, two solutions exist for Eq. \eqref{eq:parabola_eq2} as
\begin{align}
    \label{eq:parabola_eq3}f_1 = \delta \pm \arccos \left( \frac{C}{\sqrt{A^2 + B^2}} \right),
\end{align}
where the sign is selected to yield the transfer geometry corresponding to the parabolic limit $\tof \to \infty$. This limiting behavior occurs only when the trajectory traverses $f=\pi$, a condition that uniquely determines the admissible branch and removes the sign ambiguity in Eq. \eqref{eq:parabola_eq3} by selecting the positive sign. Once $f_1$ is determined, the arrival true anomaly follows as $f_2 = f_1 + \theta$, and the semi-latus rectum is given by $p = r_1 ( 1 + \cos f_1 )$. These quantities fully determine the parabolic Lambert arc. Consequently, the transfer cost $J$ may be computed for each sampled combination of $\pos{1}$ and $\pos{2}$, yielding contour plots such as those demonstrated in Figs. \ref{fig:tinf_short_m1m2_analytic.png} (short transfers) and \ref{fig:tinf_long_m1m2_analytic.png} (long transfers) for the baseline transfer scenario (Table \ref{table:sample_orbits_first}).

\section{Proof of Proposition 1}
\label{app:proof}

First, consider the $3$ by $3$ block matrix, $\phi_{\bm{rv}}$, from the \acrshort{stm} (Eq. \eqref{eq:1st_deriv}) as $\tof \rightarrow \infty$. With the singular value decomposition, $\phi_{\bm{rv}} = \bm{U}\bm{S}\bm{V}^\transpose$ with standard definitions; $\bm{U}, \bm{V}$ are left and right singular vectors, respectively, and $\bm{S}$ is a diagonal matrix with singular values. Within the Keplerian dynamics, only the largest singular value $\sigma_1$ grows linearly with time ($O(\tof)$) \cite{tsuda2010series, born1970application}, corresponding to the phase shift due to the change in the semi-major axis; the other two singular values are on the order of $O(1)$ from the super-integrable nature of the dynamics. With the linear time growth, as $\tof \rightarrow \infty$, $\phi_{\bm{rv}}$ approaches the following rank-1 matrix, 
    \begin{align}
        \label{eq:phi_rv_infty}\phi_{\bm{rv}} \rightarrow  \sigma_1 \bm{U}_1 \bm{V}_1^\transpose,
    \end{align}
    where $\bm{U}_1$, $\bm{V}_1$ are left and right singular vectors corresponding to $\sigma_1$, respectively. The right singular vector $\bm{V}_1$ is determined via the gradient of semi-major axis ($a$) with respect to the initial velocity along the transfer arc, $\vella{1}$; such a direction incurs the largest change in $a$ without variation in the initial $\bm{r}$. As such, $\bm{V}_1 = \vella{1}/\norm{\vella{1}}$. The change in the final position vector with the initial variation into this velocity-aligned direction is,
    \begin{align}
        \label{eq:r_tof_infty}\delta \bm{r}(\tof) = \frac{\partial \bm{r}}{\partial M}\delta M = \frac{\bm{v}^\la(\tof)}{n}  \delta n \cdot \tof = -\frac{3a}{\mu}\bm{v}^\la(\tof)(\vella{1} \cdot \delta \vella{1}) \tof,
    \end{align}
where $\delta n = -\frac{3na}{\mu} (\vella{1} \cdot  \delta \vella{1})$ is leveraged. Then, Eq. \eqref{eq:r_tof_infty} is re-written \footnote{From a similar logic, any block matrix of the \acrshort{stm} may be decomposed into this most stretching direction for $\tof \rightarrow \infty$.} as,
\begin{align}
    \label{eq:phi_rv_infty_structured} \frac{\partial \pos{2}}{\partial \vella{1}} \sim  \frac{3a \norm{\vella{1}}\norm{\vella{2}} }{\mu} \left(- \frac{\vella{2}}{\norm{\vella{2}}} \right) \left( \frac{\vella{1}}{\norm{\vella{1}}}\right)^\transpose \tof,
\end{align}
noting that $\bm{v}^\la(\tof) = \vella{2}$. Comparing Eq. \eqref{eq:phi_rv_infty_structured} and \eqref{eq:r_tof_infty}, it is clear that
\begin{align}
    \sigma_1 = \frac{3a \norm{\vella{1}}\norm{\vella{2}} }{\mu}t, \quad \bm{U}_1 =  -\frac{\vella{2}}{\norm{\vella{2}}}, \quad \bm{V}_1 = \frac{\vella{1}}{\norm{\vella{1}}}.
\end{align}
Now, Eq. \eqref{eq:dv1_dtof} renders the following at $\tof \rightarrow \infty$,
\begin{align}
    \frac{\delta \vella{1} }{\delta \tof }
    =  \norm{\vella{2}}\bm{V} \bm{S}^{-1} \bm{U}^\transpose \bm{U}_1
    \sim  \frac{\norm{\vella{2}}}{\sigma_1}\bm{V}_1
    \rightarrow \bm{0},
    \quad \text{as } \tof \rightarrow \infty,
\end{align}
as $\sigma_1 \rightarrow \infty$ at $\tof \rightarrow \infty$. With the \acrshort{stm} propagated in reverse time from the final to the initial points along the Lambert arc $\Gamma$, it is trivial to demonstrate $\frac{\delta \vella{2} }{\delta \tof } \rightarrow \bm{0} $ as well. With these zero vectors, Eq. \eqref{eq:dJ_deltof} also renders a zero vector, i.e., sensitivity of $J$ with respect to $\tof$ disappears with $\tof \rightarrow \infty$. At such an asymptotic case, the two conditions in Eqs. \eqref{eq:dJ_dt}-\eqref{eq:dJ_dtof} become identical to conditions in Eq. \eqref{eq:angular_optimality}. Thus, the $\t-\tof$ domain approaches $\ma{1}-\ma{2}$ domain at $\tof \rightarrow \infty$. 

\section{Application of the Framework to the Angular Domain}
\label{app:angular_domain_family}

This appendix demonstrates the application of the \acrshort{tiot} family framework, developed in Section \ref{sec:framework}, to the angular domain. In contrast to the temporal domain analyzed in the main body, the present appendix focuses on families parameterized within the $\ma{1}-\ma{2}$ domain. \taken{Starting seeds are first supplied through a grid-search procedure. The angular variables $(\ma{1},\ma{2})$ are scanned uniformly over the departure- and arrival-orbit phases leveraging a $50\times50$ grid in the domain $[0,2\pi)\times[0,2\pi)$. For each grid point, several transfer-time values ($1800$, $2400$, $3000$, $3600$, $7200$, and $10800$~s) are utilized as initial guesses for the optimality equations. After continuation and duplicate removal, this procedure yields 15 distinct angular-domain families, namely Families~1--11 and 13--16 (indexed following Table~\ref{table:m1m2_families}). The resulting family structure is then assessed using the asymptotic seeds discussed in Sec.~\ref{sec:dimensionality} and the topology consistency from empirical results. Note that the complete set of asymptotic seeds associated with the limits $\tof\rightarrow0$ and $\tof\rightarrow\infty$ is displayed in Figs.~\ref{fig:t0_short} ($\tof\rightarrow0$, short transfers), \ref{fig:t0_long_m1m2.png} ($\tof\rightarrow0$, long transfers), \ref{fig:tinf_short_m1m2.png} ($\tof\rightarrow \infty$, short transfers), and \ref{fig:tinf_long_m1m2.png} ($\tof\rightarrow \infty$, long transfers). Two observations suggest that the solution set is incomplete. First, the maximum long-transfer extremal associated with the $\tof\rightarrow0$ limit is not connected to any recovered family. Second, inspection of the family topology suggests the existence of an additional branch paired with Family~5 through a singular $\tof\rightarrow0$ limit. Similar paired-family structures are observed elsewhere in the solution space, for example between Families~5 and~8 and between Families~6 and~7 in the time domain analysis. These observations motivate the introduction of an additional manually selected seed near the expected missing branch. Applying the continuation procedure from this seed successfully recovers Family~12. Asymptotic structures and family connectivity thus serve as a diagnostic for incomplete family networks, guiding the introduction of additional seeds when necessary.}

Through this process, a total of 16 unique \acrshort{tiot} families are identified in the angular domain. These families are indexed in Table \ref{table:m1m2_families}, along with their associated asymptotic behaviors and connections to other families. The last column of Table \ref{table:m1m2_families} displays the colors that visualize each family within the three-dimensional ($\ma{1}-\ma{2}-\tof$) space in Fig. \ref{fig:families_m1m2.png}. Overall, the angular domain results confirm that the proposed framework is not restricted to the temporal domain formulation. \taken{The continuation and classification procedures, including the grid-search seeding, are identical to those of the main body; the two formulations differ only in the enforced stationary conditions, with $\partial J/\partial \ma{1} = \partial J/\partial \ma{2} = 0$ (Eq.~\eqref{eq:angular_optimality}) imposed in place of $\partial J/\partial \t = \partial J/\partial \tof = 0$. In the angular domain the asymptotic limits are available at both the $\tof \rightarrow 0$ and $\tof \rightarrow \infty$ ends, the former having no counterpart in the temporal continuation, and serve here as a diagnostic for completeness rather than as the primary seeds. Because the two domains enforce different optimality conditions, they decompose the same solution landscape into different one-parameter families; the larger number identified here (16, versus 12 in the temporal domain) is a natural consequence of the differing formulations rather than an inconsistency.}
\begin{table}[htpb]
\centering
\caption{\label{table:m1m2_families}Summary of indexed \acrshort{tiot} families in the angular domain for the baseline scenario (Table \ref{table:sample_orbits_first})}
\begin{tabular}{l c c c c }
\toprule
Family No. & End point 1 & End point 2 & Connections (if any) & Color \\\midrule
1  & $\infty^\mathrm{l}_{m, 2}$      & $\pi_2$     & Family 14                           & \linefamsolid{colFamOne} \\
2  & $\infty^\mathrm{l}_{m,1}$ & $\pi_1$    & Family 8           & \linefamsolid{colFamTwo} \\
3  & $\infty^\mathrm{l}_{\sigma, 1}$& $\pi_1$     & -                                         & \linefamsolid{colFamThree} \\
4  & $\infty^\mathrm{l}_{\sigma, 2}$& $0^\mathrm{l}_{\sigma, 1}$             & -        & \linefamsolid{colFamFour} \\
5  & $\infty^\mathrm{l}_{\sigma, 3}$     & $\pi_2$             & -               & \linefamsolid{colFamFive} \\
6  & $\infty^\mathrm{l}_{M, 1}$& $\pi_1$             & Family 9                      & \linefamsolid{colFamSix} \\
7  & $\infty^\mathrm{s}_{\sigma, 2}$& $0^\mathrm{s}_{\sigma,1}$  & -                                            & \linefamsolid{colFamSeven} \\
8  & $\infty^\mathrm{s}_{\sigma, 4}$& $\pi_1$     & Family 2                                           & \linefamsolid{colFamEight} \\
9  & $\infty^\mathrm{s}_{\sigma, 1}$     & $\pi_1$            & Family 6                                  & \linefamsolid{colFamNine} \\
10 & $\infty^\mathrm{s}_{\sigma, 3}$          & $0^\mathrm{s}_{\sigma, 2}$      & -                                          & \linefamsolid{colFamTen} \\
11 &  $0^\mathrm{l}_{m,1}$         &   $0^\mathrm{l}_{\sigma,2}$        & -                                     & \linefamsolid{colFamEleven} \\
12 &    $0^\mathrm{l}_{M,1}$       &  $\pi_2$           & -                                  & \linefamsolid{colFamTwelve} \\
13 &  $0^\mathrm{s}_{\sigma,3}$         & $0^\mathrm{s}_{m,1}$            & -                                    & \linefamsolid{colFamThirteen} \\
14 &   $0^\mathrm{s}_{m, 2}$        &  $\pi_2$          & Family 1                                    & \linefamsolid{colFamFourteen} \\
15 &   $0^\mathrm{s}_{M, 1}$        &  $\pi_2$          & Family 16                                    & \linefamsolid{colFamFifteen} \\
16 &   $\pi_1$        &  $\pi_2$          & Family 15                                    & \linefamsolid{colFamSixteen} \\\bottomrule
\end{tabular}
\end{table}

\begin{figure}[htpb]
    \centering
    \includegraphics[page=73, width=0.5\textwidth]{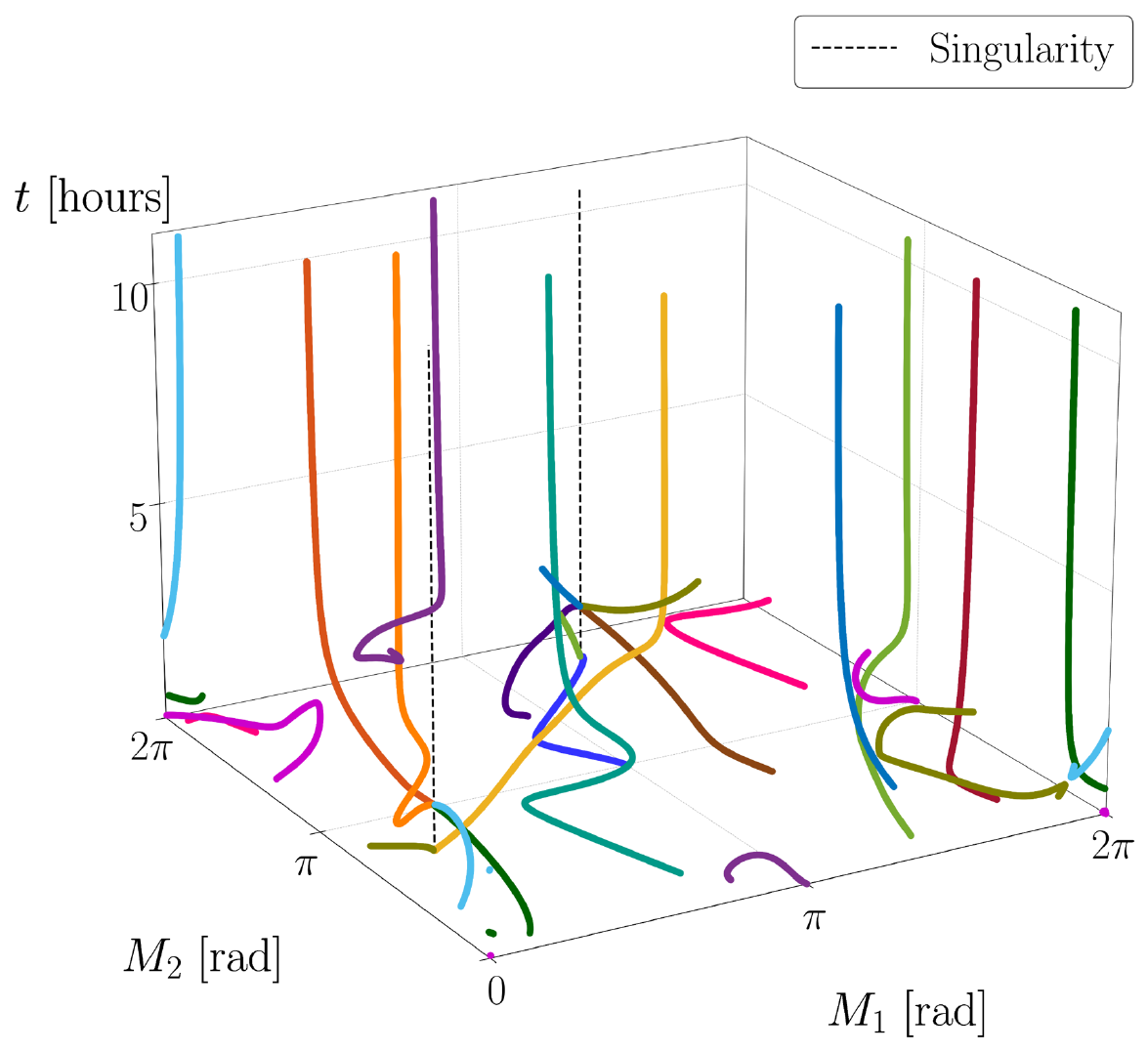}
    \caption{Global structure of \acrshort{tiot} families in the angular domain colored by their indices.}
    \label{fig:families_m1m2.png}
\end{figure}

\newpage
\section{Projection of Temporal Domain Families onto Porkchop Plots}
\label{app:family_porkchop}

Figure \ref{fig:orbit1_family_intersection} presents the projection of the families obtained in the temporal domain onto porkchop plots. The corresponding orbital geometries for each family are supplied in Fig.~\ref{fig:orbit1_family_geometry}.
\newpage

\begin{figure}[h!]
  \centering
  \begin{subfigure}[b]{0.99\textwidth}
    \centering
    \includegraphics[page=38, width=0.99\textwidth]{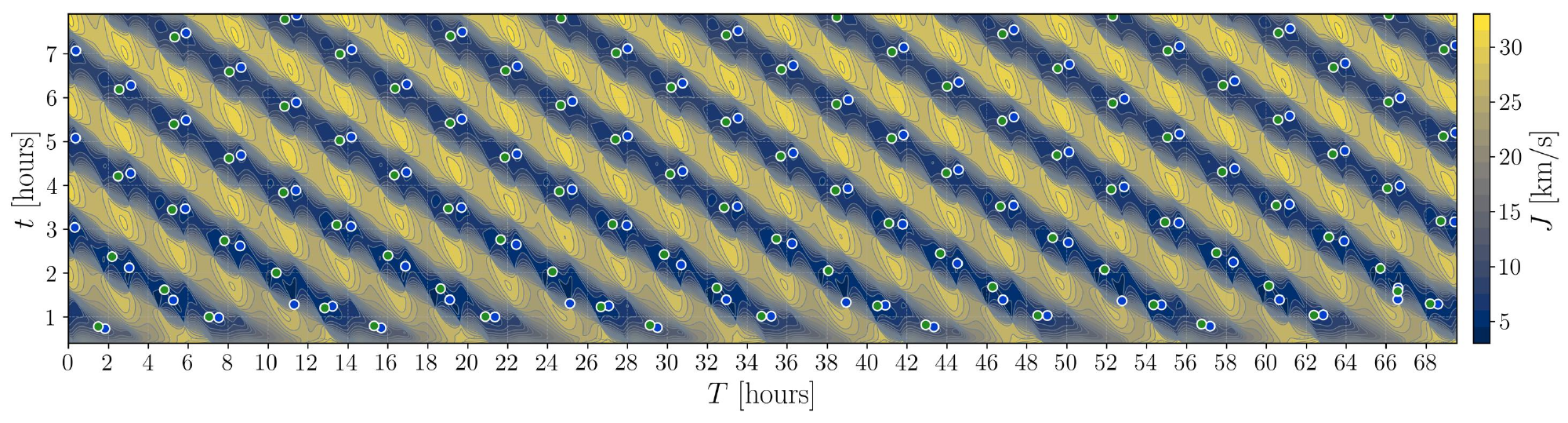}
    \caption{\label{fig:orbit1_family_intersection_seed8(family1)}Family 1.}
  \end{subfigure}
  \hfill
  \begin{subfigure}[b]{0.99\textwidth}
    \centering
    \includegraphics[page=39, width=0.99\textwidth]{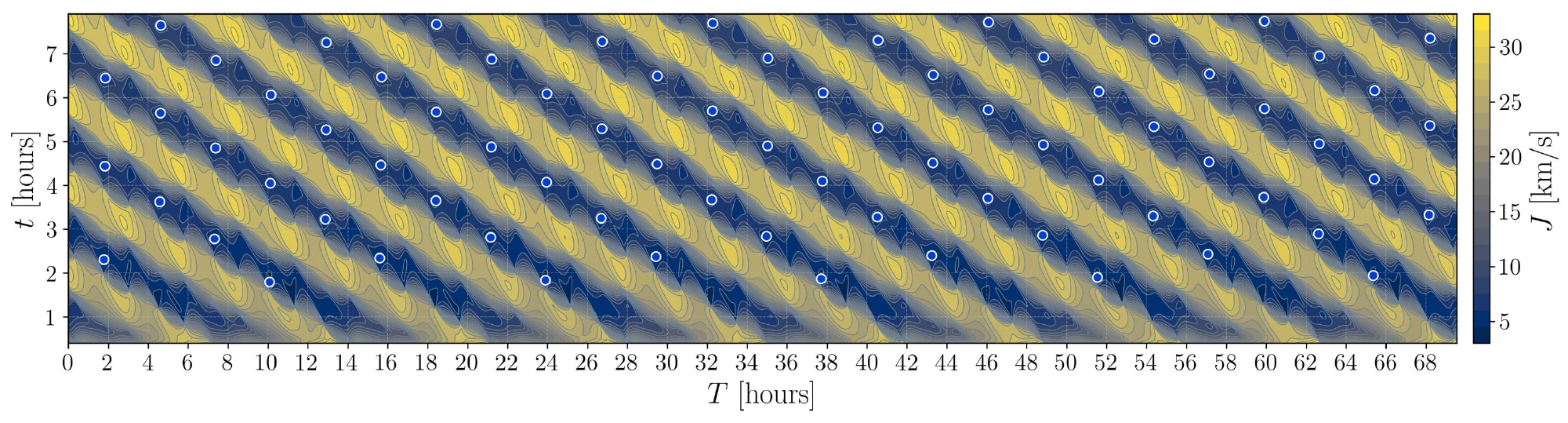}
    \caption{\label{fig:orbit1_family_intersection_seed9(family2)}Family 2.}
  \end{subfigure}
  \hfill
  \begin{subfigure}[b]{0.99\textwidth}
    \centering
    \includegraphics[page=40, width=0.99\textwidth]{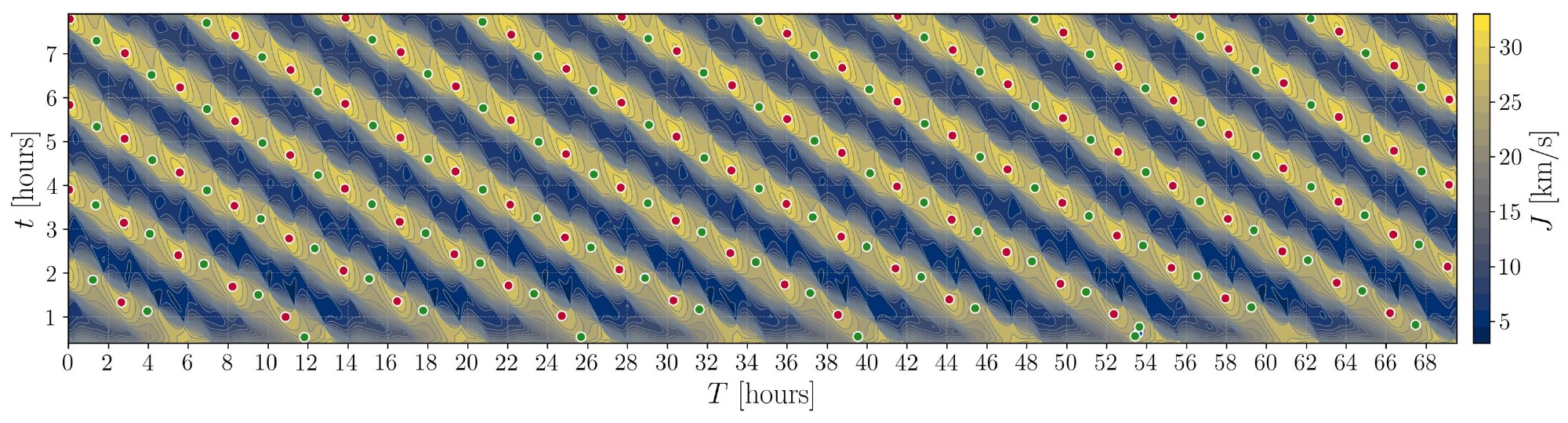}
    \caption{\label{fig:orbit1_family_intersection_seed6(family3)}Family 3.}
  \end{subfigure}
  \begin{subfigure}[b]{0.99\textwidth}
    \centering
    \includegraphics[page=41, width=0.99\textwidth]{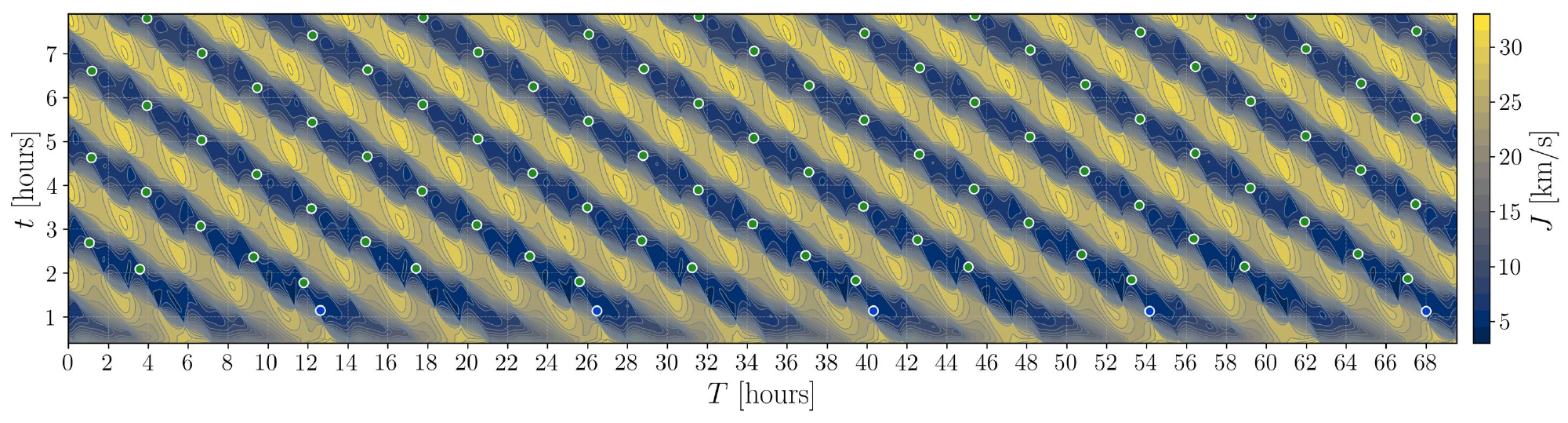}
    \caption{\label{fig:orbit1_family_intersection_seed7(family4)}Family 4.}
  \end{subfigure}
\end{figure}
\newpage

\begin{figure}[h!] \ContinuedFloat
  \centering
  \begin{subfigure}[b]{0.99\textwidth}
    \centering
    \includegraphics[page=42, width=0.99\textwidth]{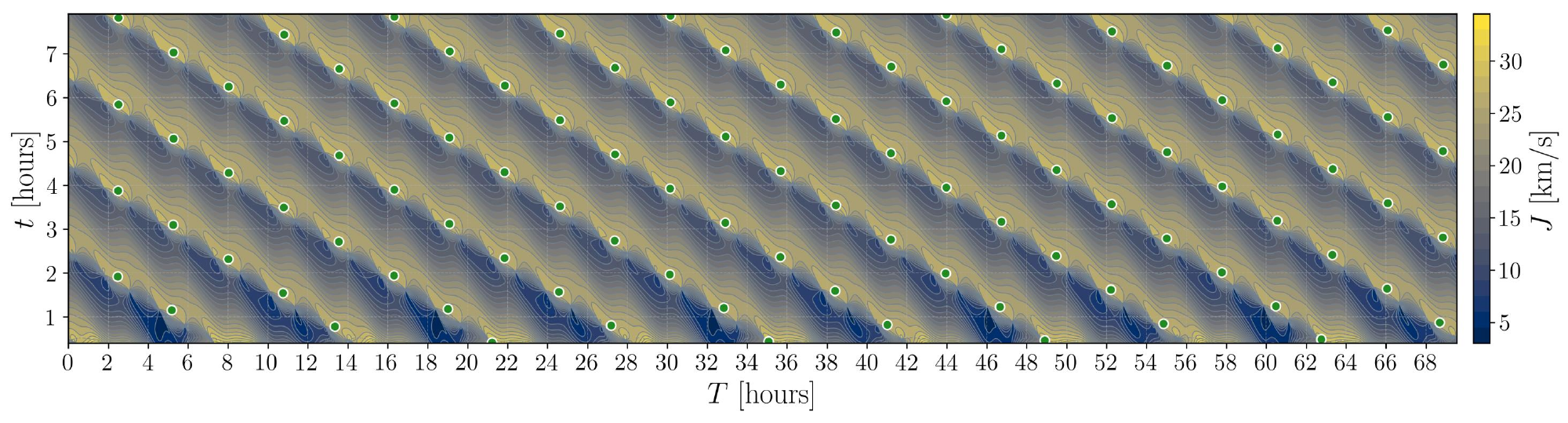}
    \caption{\label{fig:orbit1_family_intersection_seed4(family5)}Family 5.}
  \end{subfigure}
  \begin{subfigure}[b]{0.99\textwidth}
    \centering
    \includegraphics[page=43, width=0.99\textwidth]{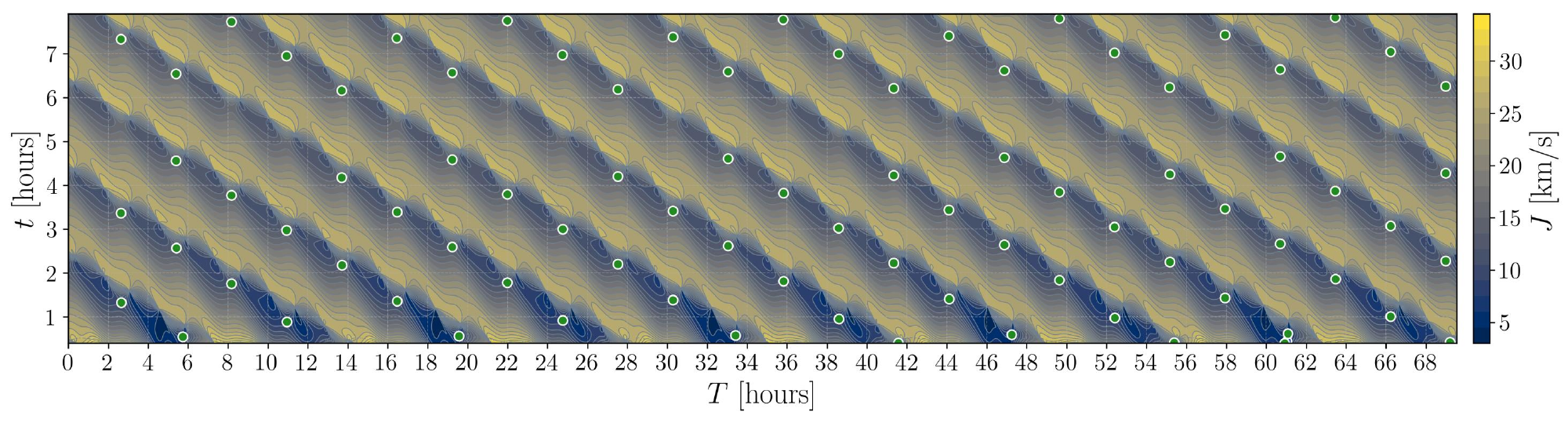}
    \caption{\label{fig:orbit1_family_intersection_seed5(family6)}Family 6.}
  \end{subfigure}
  \begin{subfigure}[b]{0.99\textwidth}
    \centering
    \includegraphics[page=44, width=0.99\textwidth]{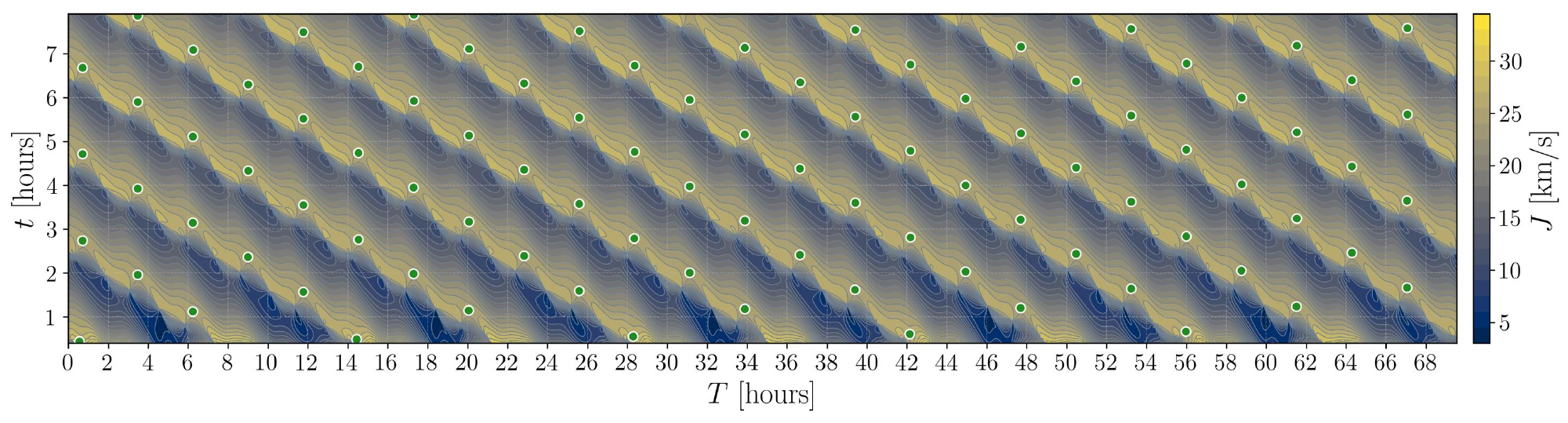}
    \caption{\label{fig:orbit1_family_intersection_seed10(family7)}Family 7.}
  \end{subfigure}
  \begin{subfigure}[b]{0.99\textwidth}
    \centering
    \includegraphics[page=45, width=0.99\textwidth]{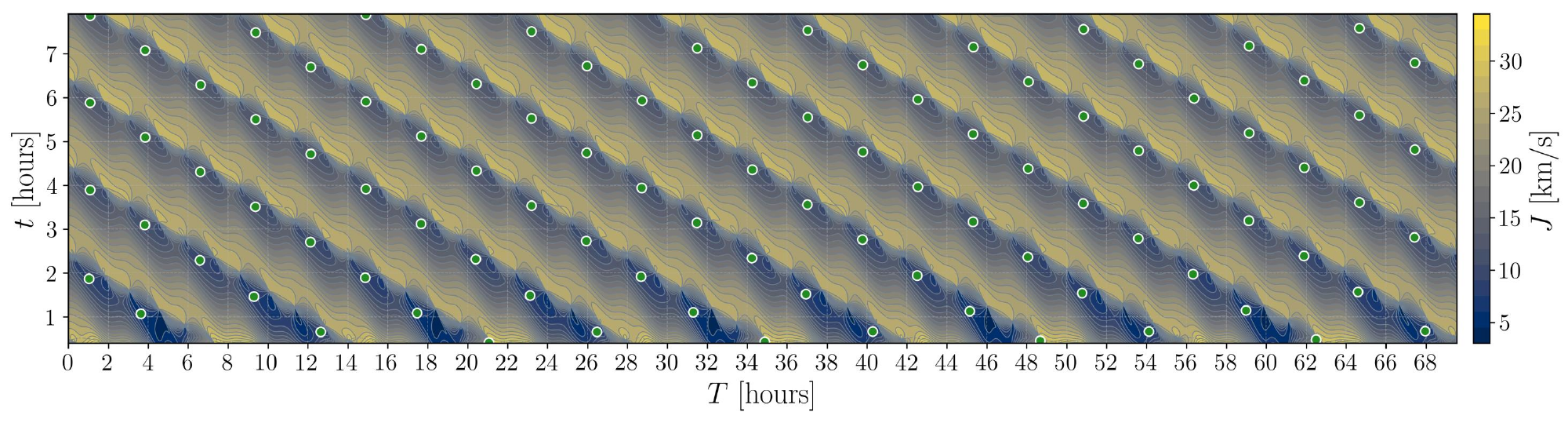}
    \caption{\label{fig:orbit1_family_intersection_seed0(family8)}Family 8.}
  \end{subfigure}
\end{figure}
\newpage

\begin{figure}[h!] \ContinuedFloat
  \centering
  \begin{subfigure}[b]{0.99\textwidth}
    \centering
    \includegraphics[page=46, width=0.99\textwidth]{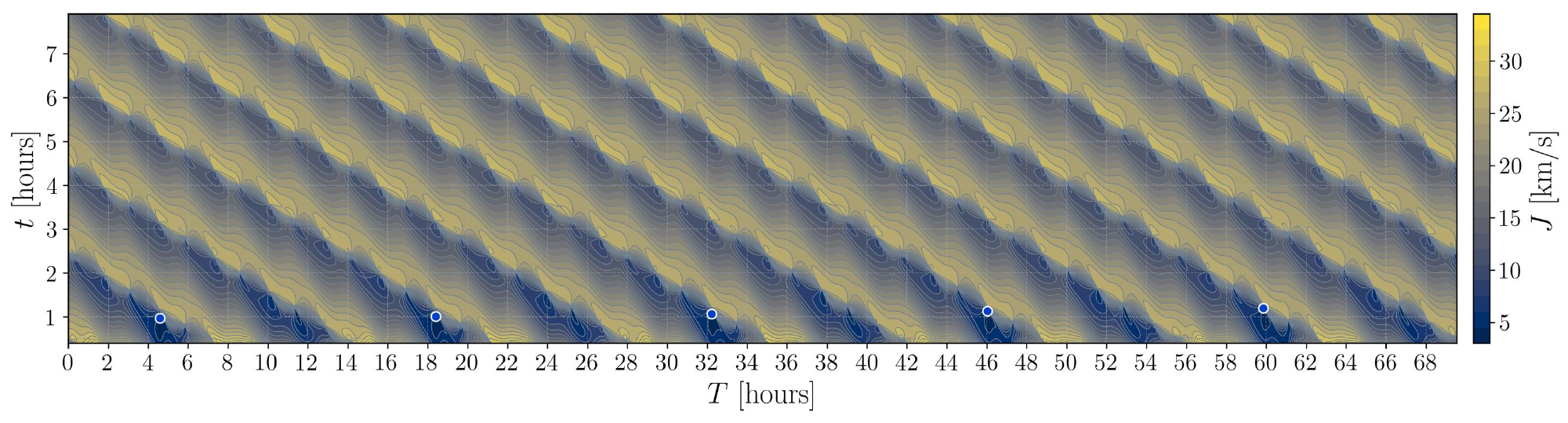}
    \caption{\label{fig:orbit1_family_intersection_seed1(family9)}Family 9.}
  \end{subfigure}
  \begin{subfigure}[b]{0.99\textwidth}
    \centering
    \includegraphics[page=47, width=0.99\textwidth]{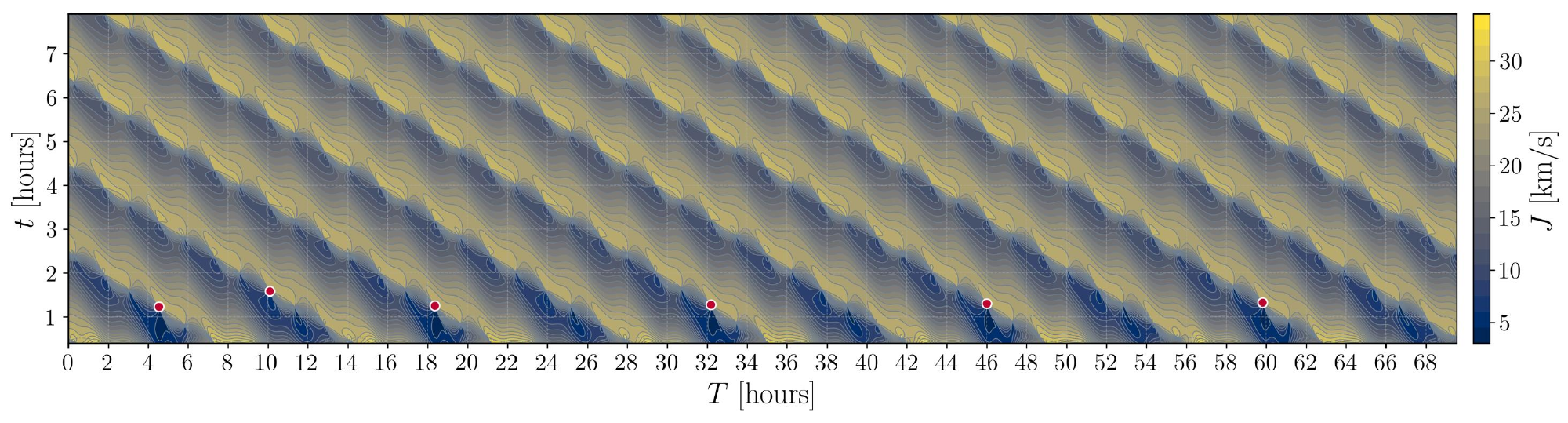}
    \caption{\label{fig:orbit1_family_intersection_seed3(family10)}Family 10.}
  \end{subfigure}
  \begin{subfigure}[b]{0.99\textwidth}
    \centering
    \includegraphics[page=48, width=0.99\textwidth]{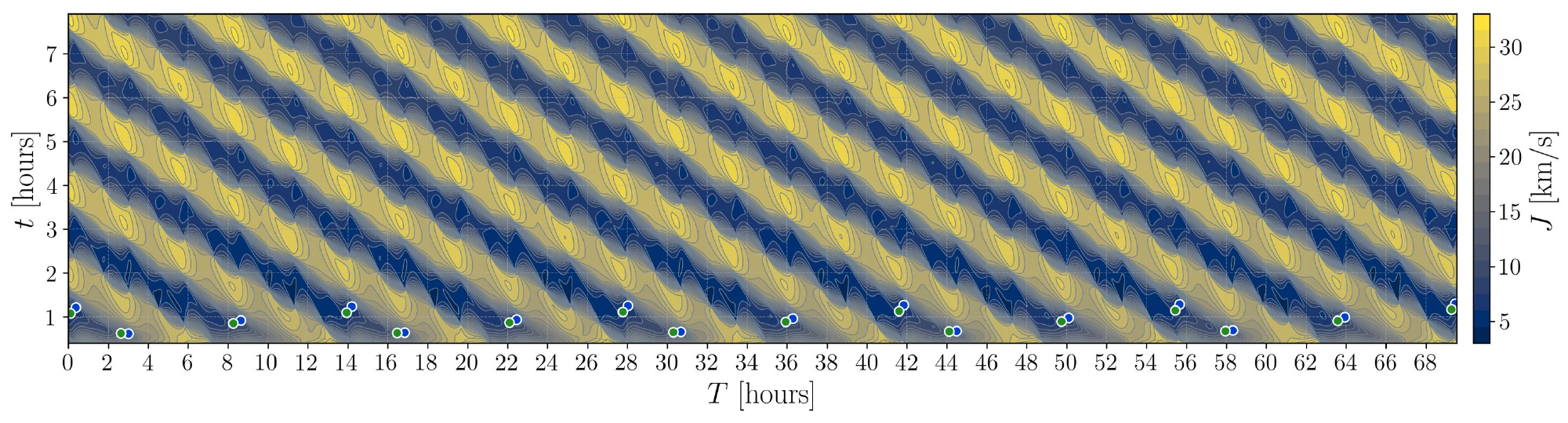}
    \caption{\label{fig:orbit1_family_intersection_seed2(family11)}Family 11.}
  \end{subfigure}
  \begin{subfigure}[b]{0.99\textwidth}
    \centering
    \includegraphics[page=49, width=0.99\textwidth]{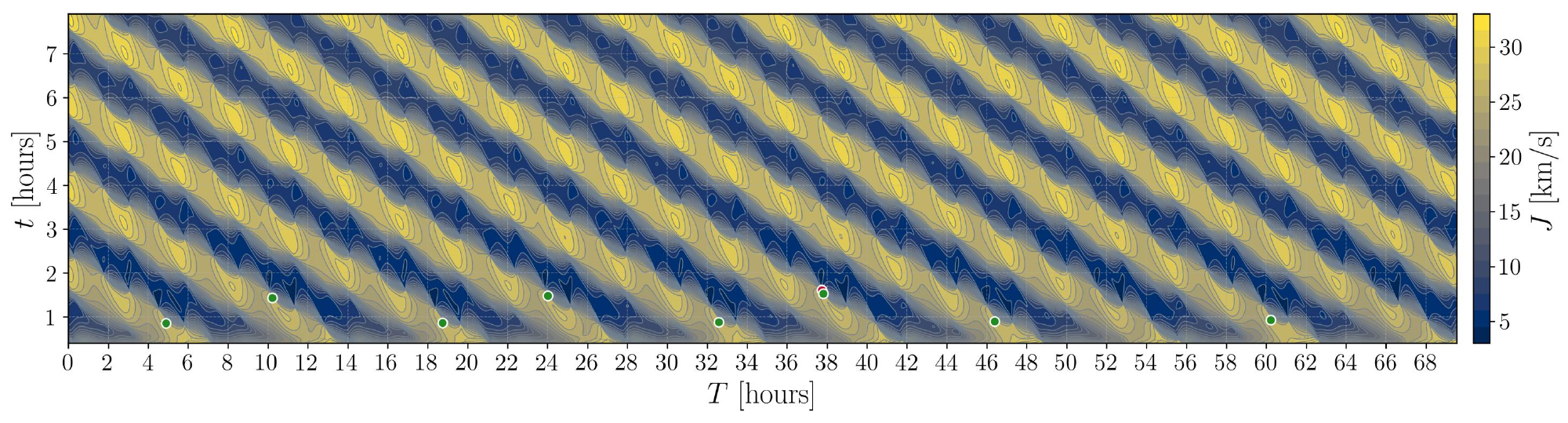}
    \caption{\label{fig:orbit1_family_intersection_seed11(family12)}Family 12.}
  \end{subfigure}
  \caption{Visual catalog of the 12 identified \acrshort{tiot} families (Fig. \ref{fig:family_connection1}) projected onto porkchop plots for the baseline scenario (Table \ref{table:family_index}).}
  \label{fig:orbit1_family_intersection}
\end{figure}
\end{appendices}
\newpage

\bmhead{Funding Sources}

Dr. Beom Park was partially supported by the Apollo 11 Postdoctoral Fellowship from the School of Aeronautics and Astronautics at Purdue University, USA. This work was supported by the InnoCORE program of the Ministry of Science and ICT of the Republic of Korea (KAIST N10250155). 

\bmhead{Acknowledgments}

Dr. Park would like to thank the members of the Strategic Aerospace Initiative (SAI) research group at KAIST for their hospitality and invaluable feedback during his stay in the summer of 2025. Dr. Park also acknowledges valuable discussions with Dr. Mitch Dominguez on the continuation scheme.

\bibliography{reference}

\end{document}